\documentclass[sn-mathphys-num]{sn-jnl}

\usepackage{graphicx}%
\usepackage{multirow}%
\usepackage{amsmath,amssymb,amsfonts}%
\usepackage{amsthm}%
\usepackage{mathrsfs}%
\usepackage[title]{appendix}%
\usepackage{xcolor}%
\usepackage{textcomp}%
\usepackage{manyfoot}%
\usepackage{booktabs}%
\usepackage{algorithm}%
\usepackage{algorithmicx}%
\usepackage{algpseudocode}%
\usepackage{listings}%

\usepackage{mathrsfs} 
\usepackage{bm}
\usepackage{a4wide} 
\usepackage{longtable}
\usepackage{tcolorbox}

\usepackage[shortlabels]{enumitem} 

\usepackage{ifthen}
\newboolean{show_comments}
\setboolean{show_comments}{true} 

\DeclareRobustCommand{\nick}[1]{
\ifthenelse{\boolean{show_comments}}
{\begingroup\color{blue}{[\textbf{Nick:} #1]}\endgroup}
{}
}
\DeclareRobustCommand{\maxim}[1]{
\ifthenelse{\boolean{show_comments}}
{\begingroup\color{green}{[\textbf{Maxim:} #1]}\endgroup}
{}
}

\numberwithin{equation}{section}
\allowdisplaybreaks[4]

\newtheoremstyle{mythm}
{5pt}
{5pt}
{\itshape}
{}
{\bfseries}
{.}
{.5em}
{}

\newtheoremstyle{mydef}
{5pt}
{5pt}
{}
{}
{\bfseries}
{.}
{.5em}
{}

\theoremstyle{mythm}
\newtheorem{proposition}{Proposition}[section]
\newtheorem{theorem}[proposition]{Theorem}
\newtheorem{lemma}[proposition]{Lemma}
\newtheorem{corollary}[proposition]{Corollary}
\theoremstyle{mydef}
\newtheorem{remark}[proposition]{Remark}
\newtheorem{example}[proposition]{Example}
\newtheorem{definition}[proposition]{Definition}

\renewenvironment{proof}{\smallskip\noindent\emph{\textbf{Proof.}}%
  \hspace{1pt}}{\hspace{-5pt}{\nobreak\quad\nobreak\hfill\nobreak%
    $\square$\vspace{2pt}\par}\smallskip\goodbreak}

\newenvironment{proofof}[1]{\smallskip\noindent{\textbf{Proof~of~#1.}}%
  \hspace{1pt}}{\hspace{-5pt}{\nobreak\quad\nobreak\hfill\nobreak%
    $\square$\vspace{2pt}\par}\smallskip\goodbreak}

\newtheorem{assumption}{Assumption}  

\newcommand{\Lip}{\mathrm{Lip}}
\renewcommand{\epsilon}{\varepsilon}
\renewcommand{\phi}{\varphi}

\newcommand{\spt}{\mathop{\rm spt}}

\renewcommand{\d}{\,d}
\renewcommand{\div}{\mathop{\rm div}}

\newcommand{\id}{\mathrm{id}}

\renewcommand{\mathbf}[1]{\bm{#1}}
\newcommand{\dom}{\mathrm{dom}}

\newcommand{\F}{\mathbb F}
\newcommand{\R}{\mathbb{R}}

\newcommand{\T}{\intercal}

\begin{document}

\title{Super-duality and necessary optimality conditions\\ of order ``infinity''
in optimal control theory}
\author*[1]{\fnm{Nikolay} \sur{Pogodaev}}\email{nikolay.pogodaev@unipd.it}
\equalcont{These authors contributed equally to this work.}

\author*[2]{\fnm{Maxim} \sur{Staritsyn}}\email{starmaxmath@gmail.com}
\equalcont{These authors contributed equally to this work.}

\affil[1]{\orgdiv{Department of Mathematics}, \orgname{University of Padova}, \orgaddress{\street{Via Trieste 63}, \city{Padova}, \postcode{35121}, \country{Italy}}}
\affil[2]{\orgdiv{Research Center for Systems and Technologies (SYSTEC)  and Advanced Production and Intelligent Systems (ARISE) Associated Laboratory, Department of Electrical and Computer Engineering}, \orgname{Faculty of Engineering, University of Porto}, \orgaddress{\street{s/n Rua Dr. Roberto Frias}, \city{Porto}, \postcode{4200-465}, 
\country{Portugal}}}

\abstract{
We systematically introduce an approach to the analysis and (numerical) solution of a broad class of nonlinear unconstrained optimal control problems, involving ordinary and distributed systems. Our approach relies on exact representations of the increments of the objective functional, drawing inspiration from the classical Weierstrass formula in Calculus of Variations. 
While such representations are straightforward to devise for state-linear problems (in vector spaces), they can also be extended to nonlinear models (in metric spaces) by immersing them into suitable linear ``super-structures''. 
We demonstrate that these increment formulas lead to necessary optimality conditions of an arbitrary order. Moreover, they enable to formulate optimality conditions of ``infinite order'',  incorporating a kind of feedback mechanism.
As a central result, we rigorously apply this general technique to the optimal control of nonlocal continuity equations in the space of probability measures. 
}

\keywords{Optimal control, Optimality conditions, Exact increment formulas, Nonlocal continuity equations}

\maketitle

\tableofcontents

\section{Introduction}\label{sec:Int}


A philosophical goal of this study is to bridge the gap between two foundational approaches in the mathematical theory of optimal control: Pontryagin's Maximum Principle (PMP) and the Dynamic Programming (DP) method. This involves refining the concept of a local extremum provided by the PMP~--- a first-order necessary optimality condition (NOC)~--- by reducing the set of non-optimal extrema, thereby bringing the notion of ``extremality'' somewhat closer to global optimality.

A natural and widely recognized approach involves developing higher-order NOCs, which are naturally expected to be more sensitive to the structure of nonlinear problems. However, in practice, this process often halts at second-order conditions due to the overwhelming complexity of higher-order variational analysis.

In this paper, we propose an alternative path originated in specific \emph{exact} representations of the increment in the objective functional, capturing cost variations of any order. These representations can be interpreted as ``infinite-order" variations of the functional, leading to non-standard necessary optimality conditions that remain local in nature but can approach the global optimum more closely than all archetypic NOCs. 

The exact increment formulas for \emph{state-linear} problems are derived by simple duality arguments. For fully \emph{nonlinear} problems, these formulas can be obtained by first ``linearizing'' the system, following the striking idea from Statistical Mechanics \cite{koopman1931hamiltonian}~--- a transformation of a nonlinear system in a metric (``physical'') space \( \mathcal{X} \) into a linear system in an infinite-dimensional Banach space \( \bm{X} \) of test functions (``observables''). We extend this approach by recognizing that the resulting model can be represented as a dynamical system on the dual space \( \bm{X}' \), allowing for a variational analysis of the corresponding control problem.

Our optimality principles, which closely resemble PMP in their form, incorporate a kind of feedback control mechanism similar to the DP approach. However, unlike the (practically utopian) solution to the nonlinear Hamilton-Jacobi equation, required by the latter, our method operates with a linear transport equation,  which always admits an explicit characteristic representation.

\subsection{Circle of ideas}

The paper deals with continuous-time controlled dynamical systems on various metric spaces \( \mathcal{X} = (\mathcal{X}, d_{\mathcal X}) \), acting in the natural (forward) or reverse (backward) timelines. A forward dynamical system on \( \mathcal{X} \) is identified with its \emph{(forward) flow},\footnote{
In the literature, the flows $\Phi$ and $\Psi$ are also referred to as forward/backward evolution operators, Cauchy operators, or propagators. 
} i.e., a two-parametric family \( \Phi \doteq (\Phi_{s,t})_{s \leq t} \) of maps \( \Phi_{s,t} \colon \mathcal{X} \to \mathcal{X} \), satisfying the algebraic axioms:
\begin{equation}
  \label{eq:family}
  \Phi_{\tau,t} \circ \Phi_{s,\tau} = \Phi_{s,t}, \quad \Phi_{t,t} = \id,
\end{equation}
for any \( s \leq \tau \leq t \).
Similarly, a backward dynamical system is represented by a \emph{backward flow}, \( \Psi = (\Psi_{s,t})_{t\leq s} \), satisfying~\eqref{eq:family} for any \( s \geq \tau \geq t \). Given the current position \( (s, \mathrm{x}) \), \( \Phi_{s,t}(\mathrm{x}) \) represents the system state at a future time \( t > s \), while \( \Psi_{s,t}(\mathrm{x}) \) represents the state at a past time \( t < s \).\footnote{These definitions of the forward/backward dynamics admit the presence of the corresponding ``arrow of time''. Such systems would be more appropriately referred to as \emph{semi-dynamical}, following, e.g., \cite{saperstone1981semidynamical}. However, for the sake of brevity, we adopt a simplified terminology.}

A highly desirable property of a dynamical system is its linearity. A forward (backward) system is said to be \emph{linear} if \( \mathcal{X} \) is a vector space and \( \Phi_{s,t} \) (respectively, \( \Psi_{s,t} \)) are linear maps. When \( \mathcal{X} \) lacks the structure of a vector space, any system on \( \mathcal{X} \) is inherently \emph{nonlinear}.

The first element of our approach is a transformation of a nonlinear system \( (\mathcal X, \Phi) \) on a metric space to a linear ``super-system'' in a dual Banach space.
To illustrate this passage, let \( \mathcal{X} = \mathbb{R}^n \) and define \( \Phi \) as the flow of a sufficiently regular time-dependent vector field \( f \colon I \times \mathcal X \to \mathcal X \), i.e., \( t \mapsto \Phi_{s,t}(\mathrm{x}) \) is a unique solution to the Cauchy problem:
\begin{equation}\label{iode}
    \dot{x} = f_t(x), \quad x(s) = \mathrm{x} \in \mathcal X.
\end{equation}
Although the state space is linear, the system $(\Phi, \mathcal X)$ inherits nonlinearity of the driving field $f$. 

Now, note that $\Phi$ generates a flow $\bm \Phi$ on the space $\mathcal{M}(\mathcal X)$ of finite signed (countably additive) Radon measures on \( \mathcal X \), which is isomorphic to the dual $\bm X'$ of the Banach space \(\bm X \doteq C_0(\mathcal X)\) of continuous functions \( \mathcal X \to \R \) vanishing at infinity. 
This flow of \emph{linear} maps $\bm \Phi_{s,t}\colon \bm X' \to \bm X'$ is defined by the action:
\begin{equation}\label{pushforward}
  \langle\bm{\Phi}_{s,t}(\mu), \phi\rangle \doteq \langle \mu, \Phi_{s,t}\phi\rangle \qquad \mu \in \bm X', \quad  \phi \in \bm X, 
\end{equation}
or, equivalently, using the \emph{pushforward} functor  \( \sharp \):
\[
  \bm{\Phi}_{s,t}(\mu)(E) = (\Phi_{s,t\sharp}\mu)(E) \doteq \mu(\Phi_{s,t}^{-1}(E))\quad \forall \mbox{  Borel set \( E \subset \mathbb{R}^n \)}.
\]

It is straightforward to verify that the family \( \bm{\Phi} \doteq (\bm \Phi_{s,t})_{s\leq t}\) satisfies the chain rule \eqref{eq:family}, i.e., \((\bm X', \bm{\Phi})\) does form a (linear) dynamical system. Moreover, the original flow \( \Phi \) is recovered from $\bm \Phi$ by taking \( \mu \in \delta(\mathcal X) \doteq \{\delta_{\mathrm x}\colon {\rm x} \in \mathcal X\} \), where \(\delta_{\mathrm{x}}\) stands for the Dirac measure concentrated at the point $\rm x$. 

We say that the nonlinear system \( (\mathcal X, \Phi) \) is \emph{immersed} into \( (\bm X', \bm \Phi) \).

\smallskip

The second ingredient of our analysis is a duality argument provided by the observation that, alongside the forward flow \( \bm \Phi \) on \( \bm X' \), the family \( \Phi \) induces a backward flow \( \bm \Psi \) on the pre-dual space \( \bm X \) via the \emph{pullback}~--- an operation, dual to the pushforward:
\begin{equation}
  \bm \Psi_{t,s}(\phi) = \Phi_{s,t}^\star \phi  \doteq  \phi\circ \Phi_{s,t}, \qquad \phi \in \bm X.\label{pull}
\end{equation}
The identities
\[
  \left<\bm\Phi_{s,\tau}(\mu), \bm \Psi_{t,\tau}(\phi)\right> \doteq \int \bm \Psi_{t,\tau}(\phi)\,\d \bm\Phi_{s,\tau}(\mu )\doteq \int \phi\circ \Phi_{\tau,t}\d \Phi_{s,\tau\sharp}\mu = \int \phi\circ\Phi_{s,t}\d \mu
\]
show that, for any fixed \( \mu \), \( \phi \) and \( s \le t\), the map \( \tau\mapsto  \left<\bm\Phi_{s,\tau}\mu, \bm \Psi_{t,\tau} \phi\right>\) remains constant, justifying the interpretation of  \((\bm X, \bm \Psi)\) as the \emph{adjoint system} of \((\bm X',\bm \Phi)\). 

Finally, when the vector field \( f \), the initial measure $\vartheta \in \bm X'$ and the cost function $\ell \in \bm X$ are sufficiently regular, an appropriate restriction of the domains of \( \mathbf{\Phi} \) and \( \bm \Psi\)  leads to the partial differential equation (PDE) representations for the corresponding trajectories \( \mu_t = \bm\Phi_{0,t}\vartheta \) and \( \bm p_t = \bm\Psi_{t,T}\ell \):
\[
  \partial_t \mu_t + \div (f_t\,\mu_t) = 0, \quad \mu_0 = \vartheta,\]
and
\[
  \partial_t \bm p_t + f_t \cdot \nabla \bm p_t = 0, \quad \bm p_T = \ell.
\]
The first PDE, known as the \emph{continuity} \emph{equation}, should be understood in the sense of distributions. The second, called the (non-conservative) \emph{transport equation}, is to be interpreted in the mild sense (see~Section~\ref{ssec:cnlp}).


Assume now that \( f \) can be controlled. That is, given a collection of functions \( u\colon t \mapsto u_t \), \( I \to U \), valued in a compact subset \( U \) of a certain topological vector space (see \S~\ref{ssec:controls}), we assume that the original vector field is modulated by \( u \), i.e., \( f_t \doteq f_t(u_t) \). This naturally implies the dependence: \( \Phi = \Phi^u \), $\bm \Phi = \bm \Phi^u$ and $\bm \Psi= \bm \Psi^u$.

Given a time moment \( T > 0 \), an initial condition \( \mathrm{x}_0 \in \mathcal X \), and a cost \( \ell \in C^1_c(\mathcal X) \), consider the following optimal control problem on the interval \( I \doteq [0, T] \): 
\[
    \inf\big\{\mathcal{I}[u] \doteq \ell(x^u(T))\colon x^u(t) = \Phi^u_{0,t}(\mathrm{x}_0), \ \, u \in \mathcal{U}\big\}.\tag{${P}$}
\]

By immersing \((\mathcal X, \Phi^u) \) into \((\bm X, \bm \Phi^u)\), we arrive at the optimal control problem:
\[
    \inf\big\{\mathcal{J}[u] \doteq \langle \mu^u_T, \ell \rangle\colon \mu^u_t = \bm{\Phi}^u_{0,t} \vartheta, \ \, u \in \mathcal{U}\big\}.\tag{${\bm{LP}}'$}
\]
Obviously, any solution of \( (\mathbf{LP}') \) corresponding to the initial condition \( \vartheta = \delta_{\mathrm{x}_0} \) is a solution to \( (P) \). 

What do we gain by passing from \( (P) \) to \( (\mathbf{LP}') \)? In contrast to \( (P) \), the latter problem is \emph{linear} in the corresponding {state variable $\mu$}. Therefore, the increment of its cost for a pair of controls \( u, \bar{u} \in \mathcal{U} \) can be represented exactly as in \cite{pogodaevExactFormulaeIncrement2024}:
\begin{equation}
  \label{eq:exact}
  \mathcal{J}[u] - \mathcal{J}[\bar{u}] = \int_0^T \left<\nabla \bar{\bm p}_t, f_t(u_t)-f_t(\bar{u}_t)\right>_{\mu_t} \, d t. 
\end{equation}
Here, \( \mu \doteq \mu^u \), \( \bar{\bm p} \doteq \bm p^{\bar u}\), and \( \left<\cdot,\cdot\right>_{\mu_t} \) denotes the scalar product in the Lebesgue space \( L^2_{\mu_t}(\mathbb{R}^n;\mathbb{R}^n) \).

\medskip

The \emph{exact} representation~\eqref{eq:exact} has several useful consequences:
\begin{enumerate}
  \item From~\eqref{eq:exact}, we can easily derive the first-order variation formula:
\begin{equation*}
  \mathcal{J}[u] - \mathcal{J}[\bar{u}] = \int_0^T \left<\nabla \bar{\bm p}_t, f_t(u_t)-f_t(\bar{u}_t)\right>_{\bar{\mu}_t} \, d t + o(\|u-\bar{u}\|), 
\end{equation*}
leading to the classical PMP, as discussed in Section~\ref{sec:odes}. Note that here, in contrast to~\eqref{eq:exact}, the integrand is evaluated along the reference trajectory \( \bar{\mu}_t = \Phi^{\bar{u}}_{0,t} \vartheta \). 

\item Taking \( \vartheta = \delta_{\mathrm{x}_0} \), we can return to the original problem \( (P) \) and express the corresponding exact increment formula as follows:
\begin{equation*}
  \mathcal{J}[u] - \mathcal{J}[\bar{u}] = \int_I \nabla \bar{\bm p}_t\left(x(t)\right)\cdot \left[f_t\left(x(t),u_t\right)-f_t\left(x(t),\bar{u}_t\right)\right] \, d t,
\end{equation*}
where \( x=x^u \) is a trajectory of~\eqref{iode} provided by the ``target'' control \( u \). It can be shown that this formula captures necessary conditions for the optimality of the reference control $\bar u$ of an \emph{arbitrary order}, provided by the corresponding regularity of the data \( \ell \) and \( f \).

\item Moreover, as it stands, the representation~\eqref{eq:exact} can be interpreted as a variation of the functional of ``order infinity'', enabling non-standard variational analysis. Given \( \bar{u} \) as a reference control tested for optimality, the pointwise minimization of the integrand,
\[
    \bar{H}_t(x(t), u_t) = \min_{\mathrm{u} \in U} \bar{H}_t(x(t), \mathrm{u}), \quad x = x^{u}; \quad \bar{H}_t(\mathrm{x}, \mathrm{u}) \doteq \nabla \bar{\bm p}_t\left(\mathrm{x}\right)\cdot f_t\left(\mathrm{x}, \mathrm{u}\right),
\]
offers the construction of a new control \( u \) in the form of a feedback loop, ensuring the property \( \mathcal{J}[u] \leq \mathcal{J}[\bar{u}] \). The optimality of \( \bar{u} \) then implies that \( \mathcal{J}[u] = \mathcal{J}[\bar{u}] \), and
\[
    \bar{H}_t(x(t), \bar{u}_t) = \min_{\mathrm{u} \in U} \bar{H}_t(x(t), \mathrm{u}),
\]
thus forming a type of ``feedback'' necessary optimality condition (see Section~\ref{ssec:cfbl}).

\item Finally, the above construction of \( u \) can be used to design a numerical algorithm for solving \( (\bm{LP}') \) \cite{pogodaevExactFormulaeIncrement2024}. The algorithm generates a sequence \( \left\{u_n\right\} \) of controls such that \( \mathcal{J}[u_{n+1}] \leq \mathcal{J}[u_n] \) for all \( n \in \mathbb{N} \). Unlike gradient-based methods, such algorithms do not require additional line search to ensure monotonicity. 
\end{enumerate}

\medskip

In this paper, we focus on the first three topics of this list, with a brief discussion of the fourth in the concluding section.

\subsection{Organization of the paper. Contribution and novelty}

For the classical optimal control problem \( (P) \), as described above, the results are largely based on our previous work~\cite{pogodaevExactFormulaeIncrement2024}. In Section~\ref{sec:odes}, we concentrate on the relationship between our approach and classical results of optimal control theory, such as PMP and second-order optimality conditions~--- questions that were not thoroughly addressed before. 

Already in~\cite{pogodaevExactFormulaeIncrement2024}, it became evident that the linear structure of the underlying space \( \mathcal{X} \) is redundant, and the results can be generalized to cases where \( \mathcal{X} \) is an arbitrary Riemannian manifold. 

In fact, \( \mathcal{X} \) can even be infinite-dimensional and need not possess any (e.g., Banach) manifold structure. As an important example of such a situation, we address in Section~\ref{sec:nonloc_cont} a fully nonlinear optimal control problem on the space \( \mathcal{X} = \mathcal{P}_2(\mathbb{R}^n) \) of probability measures with a finite second moment, where the dynamics are described by a nonlocal continuity equation.

This line of generalization culminates in Section~\ref{sec:generapp}, where we propose a systematic approach to the variational analysis of nonlinear control problems on an abstract metric space \( \mathcal{X} \), with dynamics represented by a general-form flow map \( \Phi \). 

We believe that our method is applicable to the optimal control of other prominent classes of distributed dynamical systems on metric or measure spaces, such as parabolic equations (e.g., nonlocal Fokker-Planck-Kolmogorov equations) and non-conservative hyperbolic equations (e.g., nonlocal balance equations), as discussed in the concluding Section~\ref{sec:concl}.

For completeness, the paper is supplemented with several appendices. These cover key definitions and necessary facts on manifold geometry (Appendix~\ref{app:geom}), the concept of weak*-continuous functions (Appendix~\ref{app:meas-w*}), as well as generalized (Appendix~\ref{appe:GenU}) and feedback (Appendix~\ref{ssec:fbm}) controls. To maintain a smoother narrative, the proofs of most technical results are also provided in Appendices~\ref{sec:Wflowdiff} and \ref{app:4}.

\subsection{Notations and conventions}

\begin{longtable}[\textwidth]{l p{218pt}}
General notations:\\[0.2cm]
$\mathbb F \in \{\R, \mathbb C\}$ & Scalar field\\
$\mathcal X = (\mathcal X, d_\mathcal X)$ & Base metric space\\
$\bm X$, $\bm Y$ etc. & Topological vector spaces (TVS) over $\mathbb F$ \\
$\bm X'$ & Continuous dual of a TVS $\bm X$\\
$\langle \cdot, \cdot\rangle$ & Duality pairing $\bm X' \times \bm X \to \mathbb F $\\
 \(\mathcal L(\bm X, \bm Y)\) & The space of linear bounded maps $\bm X \to \bm Y$\\
$C_b(\mathcal X)$ & The space of continuous bounded \break functions $\mathcal X \to \mathbb F$\\
$C_0(\R^n)$ & The space of continuous functions $\R^n \to \mathbb F$ vanishing at infinity\\
$C^1(\R^n)$ & The space of continuously differentiable functions $\R^n \to \mathbb F$\\
$\Lip_0(\R^n;\R^m)$, \(\displaystyle\|\cdot\|_{\bm{Lip}_0}\) & The space of Lipschitz functions $f\colon \R^n \to \R^m$ with $f(0) = 0$, and the corresponding norm\\
$\bm{\mathcal F}(\R^n)\simeq \Lip_0(\R^n;\R)'$ & The Arens-Eels (Lipschitz free) space over $\R^n$\\
$L^{\rm p}(I;\bm X)$, $\rm p\geq 1$ & The Lebesgue space of measurable $\rm p$-integrable functions $I \subset\R \to \bm X$\\
$L^\infty_{w*}(I;\bm X)$ & The space of weakly* measurable functions $I \to \bm X$\\ 
$L^\mathrm{p}_\mu=L^{\mathrm{p}}_\mu(\R^n;\R^n)$, \( p\ge 1 \) & The space of \( \mathrm{p} \)-integrable functions w.r.t. a measure $\mu$\\
$\langle \cdot, \cdot\rangle_\mu$, \( \|\cdot\|_\mu \) & Scalar product and norm in $L^2_\mu$ \\
$\mathcal{M}(\mathcal X)$ & The space of
signed bounded Radon measures on the space $\mathcal X$\\
$rba(\mathcal X)$ & The space of signed finite regular finitely additive measures on $\mathcal X$\\
$\mathcal P \doteq \mathcal P(\mathcal X)$ & The space of Borel probability measures on $\mathcal X$\\
$\mathcal P_2 \doteq \mathcal P_2(\R^n)$ & The space of Borel probability measures on $\R^n$ with finite 2-moment\\
\( W_1 \), \(W_2\) & \( L^1 \) and \(L^2\)-Kantorovich (Wasserstein) metrics on the space $\mathcal P_2$ \\
$\mathrm x \cdot \mathrm y$, $\nabla \phi$, $Df$, $\div \rho \doteq \nabla \cdot \rho$ & Scalar product, gradient, Jacobian, and divergence operator in $\R^n$\\
$\bm \nabla \xi$, $\bm D F$ & Gradient and Jacobian in $\mathcal P_2$\\
$F_\star$, $F^\star$ & Pushforward and pullback operations 
\\
\( F_{\sharp}\mu \doteq \mu \circ F^{-1} \) & Pushforward of a measure \( \mu\) through a function \(F\)
\\[0.2cm]
Control-theoretical notations:\\[0.2cm]
$I \doteq [0,T]$ & Time horizon\\
$\Delta$ & The set $\{(s,t) : 0 \leq s \leq t \leq T\}$\\
$\mathcal X=\{{\rm x}\}$ & State space and its elements\\
$\bm{\mathcal D}$ & The space of test function\\
$\Phi \doteq (\Phi_{s,t})_{(s,t) \in \Delta}$ &  The flow (propagator) of homeomorphisms $\mathcal X \to \mathcal X$ on $I$\\
$U=\{\mathrm u\}$ &  Control space and its elements\\
$x\colon t \mapsto x_t \doteq \Phi_{0,t}({\rm x})$, $I \to \mathcal X$ & Trajectories of $\Phi$\\
$u\colon t \mapsto u_t$, $I \to U$ & Control functions\\
$\mathcal U$ & The class of admissible controls\\
$\bar u$ & Given (reference) control\\
$u$ 
& Searched (target) control\\[0.2cm]
Optimal control problems:\\[0.2cm]
 $(P)$ & Basic (nonlinear) optimal control problem  \\ 
$(\bm{LP}')$ & Linear super-problem on \( \mathbf{X}' \)\\  
$(\bm{LP})$ & Dual of $(\bm{LP}')$\\  
$(\bm{LP}'|P)$ & Immersion of $(P)$ into $(\bm{LP}')$\\
$(\widetilde P)$ & Control-relaxation of the classical problem  $(P)$ on $\R^n$
\end{longtable}


\section{(Almost) classical setup}
\label{sec:odes}

To fix ideas, we begin with a rigorous exposition of the outlined concepts for a prototypical problem in Euclidean space, which serves as a foundation for more complex scenarios: 
\begin{align}
 \text{Minimize} \quad & \mathcal I[v] \doteq \ell\left(x(T)\right) &\nonumber\\
\text{subject to:} \quad & \dot x = f_t(x, v), \quad x(0) = \mathrm{x}_0 ,\tag{$P$}\\
& v \in L^\infty(I;U).\nonumber
\end{align}
Here, \( U \subset \mathbb{R}^m \) is a given compact set, the cost function $\ell$ belongs to the class \( C^1_c(\mathbb{R}^n) \), and the vector field $f$ satisfies the following regularity assumptions:
\begin{tcolorbox}
\begin{assumption}
\label{a0}
~
\begin{itemize}
  \item \( f\colon I \times (\mathbb{R}^n \times U)\to \mathbb{R}^n \) is a Carath\'{e}odory map. 
	
  \item \( f \) is bounded and Lipschitz in \( \mathrm x \), i.e., there exists \( M > 0 \) such that
	    \begin{gather*}
        \left|f_t(\mathrm x,\mathrm u)\right|\le M,\quad
	\left|f_{t}(\mathrm x,\mathrm u)-f_{t}(\mathrm x',\mathrm u)\right|\le M|\mathrm x-\mathrm x'|,
	  \end{gather*}
		for all \(\mathrm x,\mathrm x'\in \mathbb{R}^n \), \(t\in  I \), and \(\mathrm u \in U\).
  
  \item \( f \) is differentiable in \( \mathrm x \) and the derivative is Lipschitz:
  \begin{gather*}
    \left|Df_t(\mathrm x,\mathrm u) - Df_t(\mathrm x',\mathrm u)\right|\le M|\mathrm x-\mathrm x'|,
  \end{gather*}
  for all \(\mathrm x,\mathrm x'\in \mathbb{R}^n \), \(t\in  I \), and \(\mathrm u \in U\).
\end{itemize}
\end{assumption}
\end{tcolorbox}

\subsection{Generalized controls}

It is often beneficial to work with dynamical systems that are linear in the control variable. To achieve this, we use the standard relaxation approach, which involves transitioning to \emph{generalized controls}, i.e. Young measures $u$ on $I \times U$ with disintegration \( (u_t)_{t \in I} \subset \mathcal P(U) \)  (see Appendix~\ref{appe:GenU}). The class of these controls will serve as the actual set $\mathcal{U}$ of admissible inputs for the rest of this section.

\begin{remark}
    
    Any ordinary control \( v \in L^\infty(I;U) \) can be realized as a generalized control \( u[v] \) with disintegration \( (\delta_{v(t)})_{t \in I} \). In this sense, the role of generalized controls in mathematical control theory is analogous to that of mixed strategies in game theory. 
\end{remark}

A trajectory corresponding to \( u \in \mathcal U \) is a unique solution to the ODE
\begin{align}\label{Rlode}
\dot{x}(t) = f_t\left(x(t),u_t\right) \doteq \int_U f_t(x(t),\mathrm{u}) \, \mathrm{d} u_t(\mathrm{u}), \quad {x}(0) = \mathrm{x}_0,
\end{align}
where, with a slight abuse of notation, we use the same symbol \( f \) for the mapping \( I \times \mathbb{R}^n \times \mathcal{P}(U) \to \mathbb{R}^n \). Note that the dependence \(\upsilon \mapsto f_t(x,\upsilon) \) is linear as a map from the space of \emph{measures} to the space of vector fields. We emphasize, however, that a linear combination \( \lambda_1 u^1 + \lambda_2 u^2 \) of generalized controls is a generalized control only if it is a convex combination, i.e., \( \lambda_1, \lambda_2 \geq 0 \) and \( \lambda_1 + \lambda_2 = 1 \).

The reachable sets \( \{x_t^{u} \colon t \in I, \ u \in \mathcal{U}\} \subset \mathbb{R}^n \) of the generalized control system are compact at each time moment \( t \in I \). Moreover, since $x^v = x^{u[v]}$, the cost functional $\mathcal{I} \colon L^\infty(I;U) \to \mathbb{R}$ can be extended to a map $\widetilde{\mathcal{I}} \colon \mathcal{U} \to \mathbb{R}$ in a continuous manner (w.r.t. the weak* topology of $\mathcal{U}$). As a result, we have a {control-linear} problem $(\widetilde{P})$, which provides a relaxation of $(P)$ in the sense that $\inf (P) = \min(\widetilde{P})$. 

All results of this section, derived for $(\widetilde{P})$, are  translated to the original formulation by setting \( u_t = \delta_{v(t)} \).


\subsection{
Immersion}\label{ssec:cnlp}

Our analysis of the problem $(\widetilde P)$ begins with its `linearization'' according to the strategy discussed in the Introduction.

Let \( u \in \mathcal{U} \). Recall from~\cite{ABressan_BPiccoli_2007a} that, under the stated assumptions, the (time-dependent) vector field \( f_t(u_t) \doteq f_t(\cdot, u_t) \) generates a (non-autonomous) flow \( \Phi^{u} \) of \( C^1 \)-diffeomorphisms, \( \Phi^{u}_{s,t} \colon \mathbb{R}^n \to \mathbb{R}^n \). This flow induces forward and backward flows, \( \bm \Phi^{u} \) and \( \bm \Psi^{u} \), on the vector spaces \( \mathcal{M}(\mathbb{R}^n) \) and \( C_0(\mathbb{R}^n) \), respectively:
\begin{equation}
\label{eq:systems}
\bm \Phi^{u}_{s,t}(\mu) \doteq \Phi^{u}_{s,t\sharp} \mu, 
\quad 
\bm \Psi^{u}_{t,s}(\phi) \doteq \phi \circ \Phi^{u}_{t,s}.
\end{equation}

To proceed, we will need differential representations of \( \bm \Phi^{u} \) and \( \bm \Psi^{u} \). For this purpose, their domains are restricted as follows:
\[
\dom(\bm \Phi^{u}) = \mathcal{P}_2(\mathbb{R}^n), \quad \dom(\bm \Psi^{u}) = C^1_c(\mathbb{R}^n),
\]
where \( \mathcal{P}_2 = (\mathcal{P}_2(\mathbb{R}^n), W_2) \) denotes the space of probability measures with finite second moments, endowed with the 2-Kantorovich distance (see Section~\ref{sec:nonloc_cont} for detailed definitions). Note that, under the regularity assumptions \ref{a0}, the sets \( \mathcal{P}_2(\mathbb{R}^n) \) and \( C^1_c(\mathbb{R}^n) \) are invariant under the actions of \( \bm \Phi^{u} \) and \( \bm \Psi^{u} \), respectively, for all \( u \in \mathcal{U} \).

The desired differential representations are now established as follows: For any \( \vartheta \in \mathcal{P}_2(\mathbb{R}^n) \), the curve \( \mu = \mu^u \colon I \to \mathcal{P}_2 \), defined by 
\begin{equation}
\label{LC}
\mu_t = \bm \Phi^{u}_{0,t} \vartheta, \quad t \in I,
\end{equation}
is a unique continuous \emph{distributional solution} to the initial value problem~\cite[Proposition 2.12]{ambrosioGradientFlowsMetric2005}:
\begin{equation}
\label{eq:cont}
\partial_t \mu_t + \div\left(f_t(u_t) \, \mu_t\right) = 0, \quad \mu_0 = \vartheta.
\end{equation}
This means that, for any test function \( \phi \in C^1_c(\mathbb{R}^n) = \dom(\bm \Psi^{u}) \), the following holds for a.a.\ \( t \in I \):
\begin{equation}
    \frac{d}{dt} \int \phi \, d\mu_t = \int \nabla \phi \cdot f_t(u_t) \, d\mu_t.\label{eq:contW}
\end{equation}

Similarly, for any \( \ell \in C^1_c(\mathbb{R}^n) \), the curve \( \bm{p} \colon I \to C^1_c(\mathbb{R}^n) \), introduced as 
\begin{equation}
\label{LCadj}
\bm{p}_t = \bm \Psi^{u}_{T,t}(\ell) \doteq \ell \circ \Phi^{u}_{t,T}, \quad  t \in I,
\end{equation}
is a continuous \emph{mild solution} to the backward problem: 
\begin{equation}
\label{eq:transp}
\partial_t \bm{p}_t + \nabla \bm{p}_t \cdot f_t(u_t) = 0, \quad \bm{p}_T = \ell,
\end{equation}
that is, for any \( \mathrm{x} \in \mathbb{R}^n \), the following holds for a.a.\ \( t \in I \):
\[
\frac{\partial}{\partial t} \bm{p}_t(\mathrm{x}) = -\nabla \bm{p}_t(\mathrm{x}) \cdot f_t(\mathrm{x}, u_t).
\]

\begin{remark}[Operator notation]\label{opr}
It is often convenient to write the PDEs in~\eqref{eq:cont} and~\eqref{eq:transp} in a concise form: 
\begin{align*}
\partial_t \mu_t = \mathfrak{L}_t'(u_t) \, \mu_t,
\quad \text{and} \quad 
\left\{\partial_t + \mathfrak{L}_t(u_t)\right\} \bm{p}_t = 0,
\end{align*}
using the linear differential operators, known as the Lie derivative of \( \phi \) w.r.t.\ the vector field \( f_t(u_t) \):
\[
\mathfrak{L}_t(u_t) \, \phi \doteq \nabla \phi \cdot f_t(u_t), \quad  C^1_c(\mathbb{R}^n) \to C_c(\mathbb{R}^n),
\]
and their formal adjoint:
\[
\mathfrak{L}_t'(u_t) \, \mu = -\div(f_t(u_t)\mu), \quad \mathcal{P}_2(\mathbb{R}^n) \to \mathcal{M}(\mathbb{R}^n).
\]
This formalism will be instrumental in subsequent sections as we transition to more general frameworks.
\end{remark}

To complete the first step, it remains to represent the cost functional \( \widetilde{\mathcal{I}} \) as a linear form \( \langle \mu_T, \ell \rangle \). In this way, we transform the nonlinear optimal control problem \( (\widetilde{P}) \), originally defined over solutions to an ODE, into a \emph{bilinear} optimal control problem for the continuity equation in the space of probability measures:
\begin{align}
\text{Minimize} \quad & \mathcal{J}[u] \doteq \langle \mu_T, \ell \rangle \nonumber \\
\text{subject to: } \ \, & \partial_t \mu_t + \mathrm{div}\left(f_t(\cdot, u_t) \mu_t \right) = 0, \quad \mu_0 = \vartheta, \tag{\(\mathbf{LP}'\)} \\
& u \in \mathcal{U}. \nonumber
\end{align}

Finally, note that any control process \((x, v)\) in the original problem \( (P) \) can be reconstructed from \((\bm{LP}')\) by setting \( \vartheta = \delta_{\mathrm{x_0}} \) and \( u = u[v] \).

\begin{remark}[Random initial state. Double relaxation]
The problem $(\bm{LP}')$ is equivalent to the problem $(\widetilde{P})$ with an uncertain initial state (and the same deterministic control vector field). Indeed, let $\mathrm x_0$ be a random variable $ \Omega \to \R^n$ on a complete probability space $(\Omega, \mathcal F, \mathbb P)$. Then the curve $t \mapsto \mu_t$ of probability measures $\mu_t$, defined as the laws $x(t)_\sharp\mathbb P$ of the random state $x(t)\colon \Omega \to \R^n$, will be the unique distributional solution to the same continuity equation on $\mathcal P(\R^n)$ with $\vartheta \doteq (\mathrm x_0)_\sharp \mathbb P$ \cite{cavagnariLagrangianApproachTotally2023}.

This observation introduces a different type of probabilistic relaxation of the classical optimal control problem, provided by a ``mixed'' framework. Overall, the transition from \((P)\) to \((\bm{LP}')\) represents a kind of double relaxation, achieved by ``mixing'' both states and controls. 

The \((\delta_{x_0}, u[v])\)-version of the bi-linear problem \((\mathbf{LP}')\) can be viewed as a ``canonical representation'' of the (equivalent) fully nonlinear problem \((P)\). This demonstrates that, in a sense, bi-linear problems form the only fundamental class of optimal control problems.

\end{remark}
\begin{remark}[Alternative domain choices]
The presented specification of \( \dom(\bm \Phi^{u}) = \mathcal{P}_2(\mathbb{R}^n) \) and \( \dom(\bm \Psi^{u}) = C^1_c(\mathbb{R}^n) \) is not mandatory. For instance, one could instead take \( \dom(\bm \Phi^{u}) = \mathcal{P}_c(\mathbb{R}^n) \) (the set of compactly supported probability measures) and \( \dom(\bm \Psi^{u}) = C^1_b(\mathbb{R}^n) \). More general settings, such as \( \dom(\bm \Phi^{u}) \) being the space of all nonnegative measures or the space of signed regular measures with finite first moments, are explored in~\cite{BOGACHEV20153854,rippa2012,pogodaevNonlocalBalanceEquations2022}. 

The choice of \( \dom(\bm \Phi^{u}) \) and \( \dom(\bm \Psi^{u}) \) is constrained by the need to satisfy certain properties, which are crucial for our analysis. Specifically:
\begin{enumerate}
    \item \( \dom(\bm \Phi^{u}) \) and \( \dom(\bm \Psi^{u}) \) must be invariant under the corresponding dynamical systems.
    \item \( \dom(\bm \Phi^{u}) \) must include \( \delta(\mathbb{R}^n) \) (Dirac measures), ensuring that the original problem \( (\widetilde{P}) \) can be recovered.
    \item The bilinear maps \( \mu \mapsto \langle \mu, \phi \rangle \) and \( \phi \mapsto \langle \mu, \phi \rangle \) must be continuous on \( \dom(\bm \Phi^{u}) \) and \( \dom(\bm \Psi^{u}) \), respectively, where \( \dom(\bm \Phi^{u}) \) is equipped with the weak* topology of \( \mathcal{M}(\mathbb{R}^n) \) and \( \dom(\bm \Psi^{u}) \) with the strong topology of \( C_b(\mathbb{R}^n) \).
    \item Elements of \( \dom(\bm \Psi^{u}) \) must be admissible test functions for~\eqref{eq:cont} when \( \vartheta \in \dom(\bm \Phi^{u}) \).
\end{enumerate}
\end{remark}

\subsection{Exact increment formula}
\label{subsec:ee_simple}

As a second step, we use the linearity of $(\mathbf{LP}')$ to derive an exact increment formula for the cost functional $\widetilde{\mathcal{I}}$ of the problem $(\widetilde{P})$ (and, consequently, for the objective functional $\mathcal I$ of the problem $(P)$).

Fix an arbitrary pair of generalized controls $\bar u, u \in \mathcal U$, and denote by $\mu = \mu^u$,  $\bar \mu = \mu^{\bar u}$ and $\bar{\bm p} =\bm p^{\bar u}$ the corresponding solutions to the forward and backward PDEs \eqref{eq:cont} and \eqref{eq:transp} (having explicit representations \eqref{LC} and \eqref{LCadj}). 

Provided by assumptions \ref{a0}, it is easy to show (refer, e.g., to~\cite{pogodaevExactFormulaeIncrement2024}) that the map
\(
    t \mapsto \langle \mu_t, \bar{\bm p_t} \rangle
\)
is absolutely continuous on $I$, and its derivative enjoys the product rule:
\begin{align}\label{prodr}
    \frac{d}{dt}\langle \mu_t, \bar{\bm p}_t \rangle = \langle \mu_t, \partial_t \bar{\bm p}_t\rangle + \frac{d}{dt}\langle \mu_t, \bar{\bm p}_\tau \rangle\Big|_{\tau = t}.
\end{align}
Combining this relation with \eqref{eq:contW} and \eqref{eq:transp}, and recruiting the operator notation introduced by Remark~\ref{opr}, we derive:
\[
        \frac{d}{dt}\langle \mu_t, \bar{\bm p}_t \rangle = \left\langle\mu_t, \left\{\partial_t + \mathfrak L_t(u_t)\right\} \bar{\bm p}_t \right \rangle = \big\langle \mu_t, \mathfrak{L}_t(u_t - \bar u_t) \,  \bar{\bm{p}}_t \big\rangle.
\]
Taking $\bar u = u$ in this expression shows that the action map $t \mapsto \langle \bar \mu_t, \bar{\bm p}_t \rangle$ is constant on $I$. We call this relationship the \emph{duality argument}.

Now, consider the increment 
$\mathcal J[u] - \mathcal J[\bar u] $
of the cost functional in the problem $(\bm{LP}')$ on the pair $(\bar u, u)$. Having in mind the duality argument 
\[
    \langle \bar{\mu}_T, \bar{\bm p}_T\rangle = \langle \bar{\mu}_0, \bar{\bm p}_0\rangle,
\]
we can express: 
\begin{align*}
\mathcal J[u] - \mathcal J[\bar u] & = \langle \mu_T - \bar{\mu}_T, \bar{\bm p}_T\rangle - \underbrace{\langle\mu_0 - \bar{\mu}_0, \bar{\bm p}_0\rangle}_\text{$ = 0$}\\
&= \langle \mu_T, \bar{\bm p}_T\rangle - \langle \mu_0, \bar{\bm p}_0\rangle\\
&= \int_I \frac{\d}{\d{t}} \langle \mu_t, \bar{\bm p}_t\rangle \d t\\
& = \int_I \big\langle \mu_t, \mathfrak{L}_t(u_t - \bar u_t) \,  \bar{\bm{p}}_t \big\rangle \d t.
\end{align*}

Expanding the operator notation and abbreviating by \(\bar{H}_t = H_t\big|_{\psi = \nabla\bar{\bm p}_t}\) a contraction of the Hamilton-Pontryagin functional 
\[
H_t({\mathrm x},\psi,\upsilon) \doteq \psi \cdot f_t({\mathrm x},\upsilon), \quad {\mathrm x}, \psi \in \mathbb{R}^n, \quad \upsilon\in \mathcal P(U),
\]
to the vector field $\nabla\bar{\bm p}_t$, we finally come to the exact increment formula in the linear problem \((\mathbf{LP}')\): 
\begin{align}\label{eq:LCexact}
  \mathcal J[u] - \mathcal J[\bar{u}]
    = \int_I \big\langle\mu_t, \bar{H}_t(\cdot, u_t - \bar{u}_t)\big\rangle \d t.
\end{align}

\begin{remark}[\({L}^2_{\mu}\)-regularity of $\nabla \bar{\bm p}$]
The inclusion \( \nabla \bar{\bm p}_t \in L^2_{\mu_t}\) is another consequence of the representation \(\bar{\bm p}_t = \ell \circ \bar \Phi_{t,T}\), where $\bar \Phi \doteq \Phi^{\bar {u}}$.
Indeed, the gradient $\nabla \bar{\bm p}_t$ is computed by the chain rule as
  \[
    \nabla \bar{\bm p}_t = \left[D \bar \Phi_{t,T}\right]^\T \nabla \ell(\bar\Phi_{t,T}),
  \]
and
  \[
    \int |\nabla \bar{\bm p}_t|^2\d \mu_t \doteq  \int \left|\nabla \bar{\bm p}_t(\bar \Phi_{0,t})\right|^2\d \vartheta \leq \int\left|\nabla \ell\circ \bar\Phi_{0,T}\right|^2 \left|D \bar\Phi_{t,T}\circ \Phi_{0,t}\right|^2\d \vartheta.
  \]
  The inclusion $\ell \in C^1_c(\R^n)$ implies that the support $\spt|\nabla \ell|$ of the function $|\nabla \ell|$ is a compact set. Recalling that $\bar \Phi_{0,T}$ is a homeomorphism of \(\R^n\) onto itself, the preimage $\bar \Phi_{0,T}^{-1}(\spt|\nabla \ell|)$ is also compact. It remains to notice that, under assumptions \ref{a0}, the map $\left|D \bar\Phi_{t,T}\circ \Phi_{0,t}\right|^2$ remains continuous, and estimate:
  \begin{gather*}
    \int\left|\nabla \ell\circ \bar\Phi_{0,T}\right|^2 \left|D \bar\Phi_{t,T}\circ \Phi_{0,t}\right|^2\d \vartheta = \int_{\bar \Phi_{0,T}^{-1}(\spt|\nabla \ell|)} \left|\nabla \ell\circ \bar\Phi_{0,T}\right|^2 \left|D \bar\Phi_{t,T}\circ \Phi_{0,t}\right|^2\d \vartheta\\
    \leq \sup_{{\mathrm x}\in \bar \Phi_{0,T}^{-1}(\spt|\nabla \ell|)} \left|\nabla \ell(\bar\Phi_{0,T}({\mathrm x}))\right|^2\left|(D \bar\Phi_{t,T}\circ \Phi_{0,t})({\mathrm x})\right|^2 < +\infty.
  \end{gather*}
  \end{remark}
The representation \eqref{eq:LCexact} includes, as a special case, the desired exact increment formula for the cost functional \(\widetilde{\mathcal{I}}\) in the relaxed problem \((\widetilde{P})\). Specifically, taking in this expression \(\mu_t = \delta_{x(t)}\), we obtain:
\begin{align}
\widetilde{\mathcal{I}}[u] - \widetilde{\mathcal{I}}[\bar{u}]
  & = \int_I \bar{H}_t\left(x(t), u_t - \bar{u}_t\right) \, dt\nonumber \\
  & = \int_I \left[\mathfrak{L}_t(u_t - \bar u_t) \,  \bar{\bm{p}}_t\right]\left(x(t)\right) \, dt\nonumber \\
  & \doteq \int_I \nabla \bar{\bm{p}}_t\left(x(t)\right) \cdot f_t\left(x(t), u_t - \bar{u}_t\right) \, dt.\label{incr}
\end{align}

This expression can be viewed as a control-theoretical analog of the classical Weierstrass formula in the Calculus of Variations~\cite{lavrentiev1938}. 

\begin{remark}[Connection to the Weierstrass formula]

Recall that the Weierstrass formula represents the increment \( \mathcal{F}(\gamma) - \mathcal{F}(\bar \gamma) \) of the functional 
\[
\mathcal{F}(\gamma) = \int_0^1 L(\gamma, \dot{\gamma}) \, \mathrm{d} t
\]
for a pair  \( (\gamma, \bar{\gamma}) \) of $C^1$ curves defined over some region \( Q \) and sharing common boundary points. In this representation, the curve \( \bar{\gamma} \) belongs to a certain \emph{field of extremals} $\bar \Gamma$ covering \( Q \). This field acts as a coordinate system: at any moment, $\gamma$ hits some curve from $\bar \Gamma$. By summing up the ``deviations'' between the tangents to \( \gamma \) and the corresponding field curves at the intersection points, we can exactly compute the value \( \mathcal{F}(\gamma) - \mathcal{F}(\bar{\gamma}) \).

In our framework, the field $\bar \Gamma$ is replaced by the flow \( \bar{\Phi} \) of~\eqref{Rlode}, corresponding to the \emph{reference} control \( \bar{u} \). The flow provides a covering of the state space by a \emph{reference field} of trajectories \( t \mapsto \bar{\Phi}_{0,t}(\mathrm{x}) \) when $\mathrm{x}$ runs over \(\mathbb{R}^n \), and the information about this field is encoded in the super-adjoint \( \bar{\mathbf{p}} \). 

For any other control \( u \), the associated trajectory \( x = x^u \) hits, at each time moment, some curve belonging to the reference field. This enables us to trace the ``deviations'' between the tangents to \( x \) and the corresponding field curves at the intersection points, ultimately leading to the formula \eqref{incr}.
\end{remark}

Representations of the form \eqref{eq:LCexact} are well-known in the context of classical state-linear and linear-quadratic problems, particularly in numerical algorithms for optimal control \cite{GK-SIAM1972, srochko1982computational}. However, the extension \eqref{incr} to a general nonlinear problem $(\widetilde{P})$ was obtained only recently \cite{pogodaevExactFormulaeIncrement2024}.  

Despite their utility, the theoretical importance of exact increment formulas remains somewhat underestimated. For instance, it is never explicitly stated in the literature that \eqref{incr} provides a direct and elegant pathway to the PMP and higher-order optimality conditions.

\subsection{1-variation and PMP}\label{sec:pmp-c}

We now explore the relationship between \eqref{incr} and classical results in optimal control theory.

Let $\bar{u}$ be a given (reference) control, suspected to be optimal, $\bar x \doteq x^{\bar u}$ be the corresponding solution to the ODE \eqref{Rlode}, while $\bar \mu$ and $\bar{\bm p}$ be introduced as above. 

Note that the expression \eqref{eq:LCexact} holds for any element $u \in \mathcal{U}$. In particular, $u$ can be replaced by a convex combination:  
\begin{equation}
    \label{eq:perturb}
    u^\varepsilon = \bar{u} + \epsilon (u - \bar{u}), \quad \epsilon \in [0,1].  
\end{equation}
Controls of this form are called \emph{weak variations} of $\bar u$.  

Substituting \eqref{eq:perturb} into \eqref{eq:LCexact}, dividing by $\varepsilon$, and taking the limit as $\epsilon \to 0$, yields the first variation formula:
\begin{equation}
    \label{eq:LCfirst}
    \frac{\d}{\d \epsilon} \left\{ \mathcal{J}[u^\varepsilon] - \mathcal{J}[\bar{u}] \right\} \Big|_{\varepsilon = 0} =
    \int_I \left\langle \bar{\mu}_t, \bar{H}_t(u_t - \bar{u}_t) \right\rangle \d t.
\end{equation}
Here, we employ the fact that \( \mu^\varepsilon \doteq \mu^{u^\varepsilon} \to \bar{\mu} \) as \( \varepsilon \to 0 \) (in the sense of weak* convergence of probability measures), and drop ``running'' arguments for brevity. 

Using \eqref{eq:LCfirst} and reproducing the machinery of \cite{Pogodaev2018}, we arrive at the optimality condition:
\[
  \min_{\upsilon \in \mathcal{P}(U)} \left\langle \bar{\mu}_t, \bar{H}_t(\upsilon - \bar{u}_t) \right\rangle \doteq \min_{\upsilon \in \mathcal{P}(U)} \left\langle \bar{\mu}_t, H_t(\cdot, \nabla \bm{p}_t(\cdot), \upsilon - \bar{u}_t) \right\rangle = 0 \quad \text{for a.a. } t \in I,
\]
which implies, upon specifying \( \bar{\mu}_t = \delta_{\bar{x}(t)} \),
\begin{gather}
\label{SPMP}
  \min_{\upsilon \in \mathcal{P}(U)} H_t\left(\bar{x}(t), \nabla \bar{\bm{p}}_t(\bar{x}(t)), \upsilon - \bar{u}_t \right) = 0 \quad \text{for a.a. } t \in I.
\end{gather}

In parallel, we recall from  \cite[Theorem~7.1]{Gamkrelidze1978} the PMP for the problem \( (\widetilde{P}) \):
\begin{equation}
\label{pmp}
  \min_{\upsilon \in \mathcal{P}(U)} H_t\left( \bar{x}(t), \bar{p}(t), \upsilon - \bar{u}_t \right) = 0 \quad \text{for a.a. } t \in I,
\end{equation}
where \( \bar{p} \doteq p^{\bar{u}} \) denotes the co-trajectory of the reference state \( \bar{x} \), i.e., a solution to the backward ODE
\begin{equation}
\label{eq:NLadjoint}
    \dot{p} = - Df_t(\bar{x}(t), \bar{u}_t)^\T p, \quad p(T) = \nabla \ell(\bar{x}(T)).
\end{equation}

\begin{remark}
    As demonstrated in \cite[p. 116]{Gamkrelidze1978}, the condition \eqref{pmp} is equivalent to the usual minimum condition:
    \[
        \min_{\mathrm u \in U} H_t\left(\bar{x}(t), \bar p(t), \delta_{\mathrm u}- \bar u_t\right) = 0 
        \quad \text{for a.a. } t \in I.
    \]
\end{remark}

Comparing \eqref{SPMP} and \eqref{pmp} suggests that the equality 
$$
\nabla \bar{\bm{p}} \circ \bar{x} = \bar{p}
$$
must hold. This fact will be rigorously established below.

\begin{remark}[Geometric language]
  In what follows, we will primarily use the language of geometric control theory, involving such objects as the \emph{pullback} \( F^\star \) and \emph{pushforward} \( F_\star \)  of tensors through a map \( F \) between manifolds. The corresponding definitions and notations are collected in Appendix~\ref{app:geom}, and further details can be found, e.g., in \cite{agrachevControlTheoryGeometric2004}.

  Although in this section the manifold is merely \( \mathbb{R}^n \), and the geometric language is not strictly necessary, this language will become natural when formulating the results in Section~\ref{sec:nonloc_cont}.
\end{remark}
\begin{proposition}
  \label{prop:NLdiff}
  Let \( u \in \mathcal{U} \) be fixed, \( (x, p) \) be the corresponding state and adjoint trajectories, and $\bm p$ be defined by \eqref{LCadj}. Then, the equality \( \nabla \bm{p}_t(x(t)) = p(t) \) holds for all \( t \in I \).
\end{proposition}

The proof of this assertion relies on the following lemma:
\begin{lemma}[\!\!{\cite[Theorem 2.3.2]{ABressan_BPiccoli_2007a}}]
  Let \( u \in \mathcal U\), and \( \Phi \) be the flow of the vector field \( (t,\mathrm{x}) \mapsto f_t(\mathrm{x}, u_t) \).
  Then, for each \( s, t \in I \), the map \( \Phi_{s,t} \colon \mathbb{R}^n \to \mathbb{R}^n \) is differentiable. Its derivative at a point \( \mathrm{x} \in \mathbb{R}^n \) is given by the map \( (\Phi_{s,t})_{\star,\mathrm{x}} \colon T_{\mathrm{x}} \mathbb{R}^n \to T_{\Phi_{s,t}(\mathrm{x})} \mathbb{R}^n \), which acts as \( w(s) \mapsto w(t) \), where \( w \) satisfies the linear equation
  \begin{equation}
    \label{eq:NLlin}
    \dot{w}(t) = D_\mathrm{x} f_t\left( \Phi_{s,t}(x), u_t \right) w(t).
  \end{equation}
\end{lemma}

\begin{proofof}{Proposition \ref{prop:NLdiff}}
  Let \( \Phi \) be the flow of \( (t,\mathrm{x}) \mapsto f_t(\mathrm{x}, u_t) \). The representation formula~\eqref{LCadj} yields:
  \[
    \bm{p}_t = \ell \circ \Phi_{t,T} = \Phi_{t,T}^\star \bm{p}_T, \quad t \in I.
  \]
  By differentiating with respect to \( t \) and applying Lemma~\ref{lem:commute}, we obtain:
  $$
    \mathrm{d} \bm{p}_t = \mathrm{d} \left( \Phi_{t,T}^\star \bm{p}_T \right) = \Phi_{t,T}^\star \left( \mathrm{d}\bm{p}_T \right).
  $$

  Let \( w \) be any function satisfying~\eqref{eq:NLlin}.
  According to Proposition~\ref{prop:NLdiff}, we have \( w(T) = (\Phi_{t,T})_{\star,\mathrm{x}} w(t) \).
  Substituting this into the previous identity and recalling the definition of the pullback for \( 1 \)-forms, we get
  \[
    (\mathrm{d} \bm{p}_t)_{\mathrm{x}}(w(t)) = (d \bm{p}_T)_{\Phi_{t,T}(\mathrm{x})}(w(T)).
  \]
  Replacing \( \mathrm{x} \) with \( x(t) \) results in:
  \[
    (\mathrm{d} \bm{p}_t)_{x(t)}(w(t)) = (d \bm{p}_T)_{x(T)}(w(T)),
  \]
  or equivalently,
  \[
    \nabla \bm{p}_t(x(t)) \cdot w(t) = \nabla \bm{p}_T(x(T)) \cdot w(T).
  \]
  Now, defining \( \widetilde{p}(t) \doteq \nabla \bm{p}_t(x(t)) \), the identity becomes:
  \[
    \widetilde{p}(t) \cdot w(t) = \widetilde{p}(T) \cdot w(T) = \ell \cdot w(T).
  \]
  By direct computation, one can check that 
  \( p(t) \cdot w(t) = \ell\cdot w(T)\), for all \( t\in I \).
  Thus,
  \[
    \widetilde{p}(t) \cdot w(t) = p(t) \cdot w(t).  
  \]
  Since \( w \) was an arbitrary solution to~\eqref{eq:NLlin}, we conclude that \( \widetilde{p} = p \), as required.
\end{proofof}

We recognize that the adjoint trajectory \( \bar{p} \) corresponds to the linear approximation of \( \bar{\bm{p}} \) in the neighborhood of the corresponding state trajectory \( \bar{x} \):
\[
\bar{\bm{p}}_t(\mathrm{x}) - \bm{p}_t(\bar{x}(t)) \approx \nabla \bar{\bm{p}}_t(\bar{x}(t)) \cdot (\mathrm{x} - \bar{x}(t)) \doteq \bar{p}(t) \cdot (\mathrm{x} - \bar{x}(t)).
\]

It is now natural to suppose that the quadratic approximation of \( \bm{p} \) should give rise to the second-order adjoint, leading to a second-order variation formula and the corresponding NOC. The next section delves into this ansatz.

\subsection{2-variation and second-order optimality conditions}\label{ssec:2nd}

In addition to~\ref{a0}, suppose that \( f \) and \( \ell \) are smooth in the variable \( \mathrm{x} \), which enables the application of certain useful results from the framework of Chronological Calculus~\cite[Chapter~2]{agrachevControlTheoryGeometric2004}. 

In this formalism, both vector fields and diffeomorphisms are treated as linear operators on the Fr\'{e}chet space \( C^{\infty}(\mathbb{R}^n) \). A vector field \( g\colon \mathbb{R}^n\to \mathbb{R}^n  \) is associated with the differential operator 
  \[
    \widehat{g} \doteq g^1 \partial_{\mathrm{x}^1} + \cdots g^n \partial_{\mathrm{x}^n}\colon C^{\infty}(\mathbb{R}^n)\to C^{\infty}(\mathbb{R}^n)
  \]
  and a diffeomorphism \( \Phi\colon \mathbb{R}^n\to \mathbb{R}^n \) is associated with the map 
  \[
    \widehat{\Phi} \phi \doteq \Phi^\star\phi = \phi\circ \Phi, \quad \phi\in C^{\infty}(\mathbb{R}^n).
  \]

We require the following facts:
  \begin{itemize}
    \item Let \( \bar{u} \) be a fixed generalized control.  
      Under our assumptions the family of vector fields \( f_t\doteq f_t(\cdot, \bar u_t) \) generates the backward flow \( \Phi_{t,T} \) satisfying 
      \[
        \frac{d}{dt}{\widehat{\Phi}}_{t,T} = - \widehat{f}_t\circ \widehat{\Phi}_{t,T}, \quad \widehat{\Phi}_{T,T}=\id,
      \]
      for a.e. \( t\in I \).
      In particular, it follows from \( \bar{\mathbf{p}}_t = \ell\circ\Phi_{t,T} =\widehat{\Phi}_{t,T}\ell \) that 
      \begin{equation}
        \label{eq:strong}
        \frac{d}{dt} \bar{\mathbf{p}}_t(x) = - \nabla \bar{\mathbf{p}}_t(x)\cdot f_t(x) \quad \forall x\in \mathbb{R}^n. 
      \end{equation}
      Note that this is a stronger condition than~\eqref{eq:transp}, because the latter holds only in the mild sense.

    \item Let \( A_t, B_t \colon C^\infty(\mathbb{R}^n)\to C^{\infty}(\mathbb{R}^n) \) be two continuous families of linear operators which are differentiable at \( t_0 \). Then the family \( A_t \circ B_t \) is differentiable at \( t_0 \) and satisfies the Leibnitz rule (see~\cite[Section~2.3]{agrachevControlTheoryGeometric2004}):
      \[
        \frac{d}{dt}\Big|_{t=t_0}(A_t\circ B_t) = \left(\frac{d}{dt}\Big|_{t=t_0} A_t\right)\circ B_t + A_t \circ \left(\frac{d}{dt}\Big|_{t=t_0} B_t\right).
      \]
      We apply this formula to the families \( A_t = \partial_{\mathrm{x}^i}\circ \partial_{\mathrm{x}^j} \) with \( 1\le i,j\le n \) and \( B_t = \widehat{\Phi}_{t,T} \):  
      \[
        \frac{d}{dt}\Big|_{t=t_0}\left(\partial_{\mathrm{x}^i}\circ \partial_{\mathrm{x}^j}\circ \widehat{\Phi}_{t,T}\right) = \partial_{\mathrm{x}^i}\circ \partial_{\mathrm{x}^j}\circ \left(\frac{d}{dt}\Big|_{t=t_0} \widehat{\Phi}_{t,T}\right).
      \]
      In particular, if \( \widehat{\Phi}_{t,T} \) is differentiable at \( t_0 \) (which holds for a.e. \( t_0\in I \)), then both \( t\mapsto \bar{\mathbf{p}}_t(\mathrm{x}) \) and  \( t\mapsto\partial_{\mathrm{x}^i}\partial_{\mathrm{x}^j}\bar{\mathbf{p}}_t(\mathrm{x})\) are differentiable at \( t_0 \) for any \( \mathrm{x}\in \mathbb{R}^n \) and
      \begin{equation}
        \label{eq:commute}
        \frac{d}{dt}\Big|_{t=t_0}\left(\partial_{\mathrm{x}^i}\partial_{\mathrm{x}^j}\bar{\mathbf{p}}_t(\mathrm{x})\right) = \partial_{\mathrm{x}^i}\partial_{\mathrm{x}^j}\left(\frac{d}{dt}\Big|_{t=t_0}\bar{\mathbf{p}}_t(\mathrm{x})\right) \quad \forall \mathrm{x}\in \mathbb{R}^n.
      \end{equation}
      Remark that~\eqref{eq:commute} is not an obvious identity, since \(t\mapsto \bar{\mathbf{p}}_t(x)\) is merely absolutely continuous. 
  \end{itemize}

\begin{proposition}\label{propo2ord}
  In addition to~\ref{a0}, suppose that \( f \) and \( \ell \) are \( C^\infty \) maps in \( \mathrm x \). 
  Then, for \( u^{\varepsilon} \) defined as in~\eqref{eq:perturb}, it holds
  \begin{align}
    \widetilde{\mathcal I}[u^{\varepsilon}] - \widetilde{\mathcal I}[\bar{u}]
    &= \varepsilon\int_I \bar{p}(t)\cdot f_t \left(\bar{x}(t), u_t - \bar{u}_t \right)\, dt \notag\\
    &+\varepsilon^2\int_I \left( \bar{P}(t)^{\T} f_t\left(\bar{x}(t),u_t - \bar{u}_t \right) + Df_t\left(\bar{x}(t),u_t - \bar{u}_t \right)^{\T}\bar{p}(t)  \right)\cdot \bar{y}(t)\, dt\notag \\
    &+ o(\varepsilon^2),
    \label{eq:2order}
  \end{align}
where \( \bar{p} \) is the co-trajectory corresponding to \( \bar{u} \), \( \bar{P} \) satisfies the matrix Riccati equation
\begin{equation}
  \label{eq:riccati}
\dot P = - Df[t]^{\T} P - P Df[t] - D^2H[t], \quad P(T)=-D^2\ell(\bar{x}(T)),
\end{equation}
and \( \bar{y} \) is a solution of the linearized equation
\begin{equation}
  \label{eq:lin}
\dot y = Df[t] \, y + f_t(\bar{x},u_{t}-\bar{u}_{t}),\quad y(0) = 0,
\end{equation}
with
\[
  Df[t] \doteq Df_t(\bar{x}(t), \bar{u}_t), \quad
D^2H[t] \doteq D^2H_t(\bar{x}(t),\bar{p}(t),\bar{u}_t).
\]

Moreover, it holds
\[
  \bar{p}(t) = \nabla \bar{\bm p}_t(\bar{x}(t)),\quad
  \bar{P}(t) = D^2 \bar{\bm p}_t(\bar{x}(t)), \quad t \in I,
\]
where $D^2$ stands for the Hessian matrix.
\end{proposition}

\begin{proof}
  By plugging \( u^{\varepsilon} \) into \eqref{eq:LCexact} and abbreviating \( x^{\varepsilon} = x^{u^{\varepsilon}} \), we obtain:
  \[
   \widetilde{\mathcal I}[u^{\varepsilon}] - \widetilde{\mathcal I}[\bar{u}]
   = \epsilon\int_I \nabla \bar{\bm p}_t(x^{\varepsilon}(t))\cdot f_t\left(x^{\varepsilon}(t), u^{\varepsilon}_t - \bar{u}_t\right) \d t.
 \]

Our regularity assumptions on \( f \) and \( \ell \), together with the representation~\eqref{LCadj}, imply that \( \mathrm x \mapsto \nabla \bm p_t(\rm x) \) is $C^\infty$, and therefore,
  \[
    \nabla \bar{\bm p}_t(x^{\varepsilon}(t)) = \nabla \bar{\bm p}_t(\bar{x}(t)) + D^2 \bar{\bm p}_t(\bar{x}(t)) \left[x^{\varepsilon}(t) - \bar{x}(t)\right] + O \left( \left|x^{\varepsilon}(t) - \bar{x}(t) \right|^2 \right).
  \]
  On the other hand, from the variational formula, established, e.g., in~\cite[Section~2.7]{agrachevControlTheoryGeometric2004}, it follows that \[
  x^{\varepsilon}(t) - \bar{x}(t) = \varepsilon \bar{y}(t) + o(\varepsilon).
  \]
  Therefore,
  \[
    \nabla \bar{\bm p}_t(x^{\varepsilon}(t)) = \nabla\bar{\bm p}_t(\bar{x}(t)) + \varepsilon D^2 \bar{\bm p}_t(\bar{x}(t)) \, \bar{y}(t) + O(\varepsilon^2).
  \]
  Similarly, 
  \begin{align*}
    f_t(x^{\varepsilon}(t), u^{\varepsilon}_t - \bar{u}_t) = & \
    \varepsilon f_t(x^{\varepsilon}(t), u_t - \bar{u}_t)\\ = & \ 
    \varepsilon f_t(\bar{x}(t), u_t - \bar{u}_t) + \varepsilon^2 Df_t(\bar{x}(t), u_t - \bar{u}_t) \, \bar{y}(t) + O(\varepsilon^3).
  \end{align*}

Combining these expansions, we conclude that the first-order term is equal to
  \[
    \nabla \bar{\bm p}_t(\bar{x}(t)) \cdot f_t(x^{\varepsilon}(t), u^{\varepsilon}_t - \bar{u}_t)
  \]
  and the second-order term to
  \[
    f_t(\bar{x}(t), u_t - \bar{u}_t) \cdot D^2 \bar{\bm p}_t(\bar{x}(t)) \, \bar{y}(t) + \nabla \bar{\bm p}_t(\bar{x}(t)) \cdot Df_t(\bar{x}(t), u_t - \bar{u}_t) \, \bar{y}(t).
  \]

  By Proposition~\ref{prop:NLdiff}, we have \( \nabla \bar{\bm p}_t(\bar{x}(t)) = \bar{p}(t) \), and to complete the proof it remains to verify that the map
  \[
    t \mapsto P(t) \doteq D^2 \bar{\bm p}_t(\bar{x}(t))
  \]
  satisfies the Riccati equation~\eqref{eq:riccati}. This can be checked by direct computation.

Indeed, \( P \) obviously satisfies the terminal condition. Let us compute the time derivative of \( P(t) \).
To that end, we will use a tensor partial derivative notation \( \frac{\partial}{\partial x^{j}} T = T_{,j} \) and adopt the Einstein summation convention, summing over terms with the same upper and lower indices.

With this notation, the coordinates of \( P(t) \) are given by \( P_{ij}(t) = \bar{\bm p}_{,ij}(t,\bar x(t)) \).
Hence
\begin{equation}
  \label{eq:P1}
  \frac{d}{dt} P_{ij}(t) = \left(\partial_t \bar{\bm p}_{,ij} + \bar{\bm p}_{,kij} f^k\right)(t,\bar x(t)) \doteq a_{ij}(t),
\end{equation}
for a.e. \( t \in I \). 

On the other hand, we know from~\eqref{eq:strong} that the equality
\[
\partial_t \bar{\bm p}_t + \nabla \bar{\bm p}_t \cdot f_t(\cdot,\bar{u}_t) = 0
\]
holds on $\R^n$ for a.a. \( t \in I \). In coordinates, this reads:
\[
b \doteq \partial_t \bar{\bm p} + \bar{\bm p}_{,k} f^{k} = 0.
\]
Let us compute \( b_{,ij} \).
For the first derivative, we have
\[
b_{,i} = (\partial_t \bar{\bm p})_{,i} + \bar{\bm p}_{,ki} f^{k} + \bar{\bm p}_{,k} f^k_{,i}.
\]
Now, for the second derivative
\[
b_{,ij} = (\partial_t \bar{\bm p})_{,ij} + \bar{\bm p}_{,kij} f^{k} + \bar{\bm p}_{,ki} f^k_{,j} + \bar{\bm p}_{,kj} f^k_{,i} + \bar{\bm p}_{,k} f^k_{,ij} = 0.
\]
Along the trajectory \(\bar x \), this reads:
\begin{align*}
b_{,ij}\big|_{(t,\bar x(t))}  = & \left((\partial_t \bar{\bm p})_{,ij} + \bar{\bm p}_{,kij} f^{k}\right)\big|_{(t,\bar x(t))}\\ & +
P_{ki} f^k_{,j}\big|_{(t,\bar x(t))} + P_{kj} f^k_{,i}\big|_{(t,\bar x(t))} + \bar{p}_{k}(t) f^k_{,ij}\big|_{(t,\bar x(t))}\\ 
 = & 0,
\end{align*}
where we use the notation \( P_{ij}(t) = \bar{\bm p}_{,ij}\big|_{(t,\bar x(t))} \) and the established identity 
\[
    \bar{p}_k(t) = \bar{\bm p}_{,k}\big|_{(t,\bar x(t))}.
\]
Therefore,
\begin{align*}
  a_{ij}(t) & = a_{ij}(t) - b_{,ij}\big|_{(t,\bar x(t))}\\
  & = \left(\partial_t \bm p_{,ij} - (\partial_t \bm p)_{,ij}\right)\big|_{(t,\bar x(t))} - \left(P_{ki} f^k_{,j} + P_{kj}f^k_{,i} + p_{k}f^k_{,ij}\right)\big|_{(t,\bar{x}(t))}.
\end{align*}
It follows from~\eqref{eq:commute} that 
\[
    \left(\partial_t \bm p_{,ij} - (\partial_t \bm p)_{,ij}\right)\big|_{(t,\bar x(t))} = 0.
\]

Thus, \eqref{eq:P1} can be rewritten as
\[
\dot P = - (D f_t(\bar x))^{\T} P - P Df_t(\bar x) - D^2 H_t(\bar x,\bar p, \bar u_t)
\]
in the matrix notation. This observation finishes the proof.
\end{proof}

\begin{remark}
  If \( \bar{u} \) is optimal and \( u \) satisfies the minimum condition~\eqref{pmp}, then the second-order variation formula~\eqref{eq:2order} implies that 
\[
\int_I \left( \bar{P}(t)^{\T} f_t\left( \bar{x}(t), u_t - \bar{u}_t \right) + Df_t\left( \bar{x}(t), u_t - \bar{u}_t \right)^{\T} \bar{p}(t) \right) \cdot \bar{y}(t) \, dt \geq 0,
\]
which coincides with Formula (6.2) in~\cite{frankowskaPointwiseSecondorderNecessary2017}. From this inequality, one may deduce a second-order necessary optimality condition similar to~\cite[Theorem~3.1]{frankowskaPointwiseSecondorderNecessary2017} by employing, for instance, the standard needle variation technique.

\end{remark}

Following this progression, one can establish the relationship between higher-order terms in the Taylor expansion of $\bar{\bm{p}}$ near $\bar{x}$~--- serving as ``higher-order adjoints'' of $\bar{x}$~--- and the corresponding higher-order variation formulas. However, this approach is generally avoided because the resulting analysis becomes technically intricate, and the derived NOCs are expected to be challenging to apply in practice.

An alternative path, pursued in the next section, is to work directly with the representation \eqref{incr}. This representation captures the entirety of the function \( \bar{\bm{p}} \circ \bar{x} \) and, consequently~--- when contracted to the class \eqref{eq:perturb} of weak variations of \( \bar{u} \)~--- can be naturally interpreted as an ``infinite-order'' variation of the objective functional \( \widetilde{\mathcal{I}} \) w.r.t. the control variation \( \epsilon \mapsto u^\varepsilon \).

In this context, the function \( \bar{\bm{p}} \), which encapsulates information about the adjoints of \( \bar{x} \) at all orders, can be treated as an ``infinite-order adjoint'' of \( \bar{x} \). Since, at the same time, \( \bar{\bm{p}} \) acts as the adjoint of the corresponding state trajectory \( \bar{\bm{\mu}} \) in the ``super-version'' \( (\bm{LP}') \) of \( (P) \), we will refer to it as the \emph{super-adjoint} of \( \bar{x} \).

\subsection{$\infty$-variation and feedback optimality conditions}\label{ssec:cfbl}

We now return to the problem \( (\bm{LP}') \), and suppose that the function \( u \), participating  in \eqref{eq:LCexact}, meets the following condition for a.e. \( t \in I \):
\begin{align}
    \left\langle \mu_t, \bar{H}_t(u_t) \right\rangle = \min_{\upsilon \in \mathcal{P}(U)} \left< \mu_t, \bar{H}_t(\upsilon) \right>.%
    \label{clucomp}
\end{align}
As immediately follows from the representation \eqref{eq:LCexact}, the inequality \[ \mathcal{J}[u] \leq \mathcal{J}[\bar{u}] \] holds. Moreover, if \( \bar{u} \) is optimal for \( (\bm{LP}') \), we must have equality: 
\[ 
    \mathcal{J}[u] = \mathcal{J}[\bar{u}],
\]
which, due to the non-positivity of the integrand, implies the following condition:
\begin{align}
    \left\langle \mu_t, \bar{H}_t(\bar{u}_t) \right\rangle = \left\langle \mu_t, \bar{H}_t(u_t) \right\rangle = \min_{\upsilon \in \mathcal{P}(U)} \left< \mu_t, \bar{H}_t(\upsilon) \right> \quad \text{for a.a. } \ t \in I. \label{clFBM}
\end{align}

As a particular case of the problem $(\widetilde{P})$, we have:
\begin{align}
    \bar{H}_t(x(t), \bar{u}_t) = \bar{H}_t(x(t), u_t) = \min_{\upsilon \in \mathcal{P}(U)} \bar{H}_t(x(t), \upsilon) \quad \text{for a.a. } t \in I \label{lfbm}
\end{align}
which follows from the condition:
\begin{align}
    \bar{H}_t(x(t), u_t) = \min_{\upsilon \in \mathcal{P}(U)} \bar{H}_t(x(t), \upsilon) \quad \text{for a.a. } t \in I. \label{lu-com}
\end{align}

These straightforward arguments lead to the conjecture:
\begin{theorem}[Feedback NOC]\label{thm:fbc}
    Assume that \( \bar{u} \) is optimal for \( (\bm{LP}') \) {\rm (}respectively, for \( (\widetilde{P}) \){\rm )}. Then, the relations \eqref{clFBM} (resp., \eqref{lfbm}) hold for any \( u \in \mathcal{U} \) satisfying \eqref{clucomp} {\rm (}resp., \eqref{lu-com}\textrm{)} for a.a. \( t \in I \), and \( \mathcal{J}[u] = \mathcal{J}[\bar{u}] \) {\rm (}resp., \( \widetilde{\mathcal{I}}[u] = \widetilde{\mathcal{I}}[\bar{u}] \){\rm )}.
\end{theorem}

This theorem provides a non-classical necessary condition for the optimality of \( \bar{u} \), closely resembling PMP. The only distinction is that \( \bar{\mu} \) (respectively, \( \bar{x} \)) is replaced by a trajectory \( \mu \) (resp., \( x \)) of a new ``comparison process'', which is derived by solving \eqref{clucomp} (resp., \eqref{lu-com}). 

It should be noted, however, that \eqref{clucomp} involves a feedback loop \( \mu = \mu^u \), making it an operator equation on \( \mathcal{U} \). This justifies attributing the derived optimality principle to the class of ``feedback NOCs'' \cite{Dykhta2014}.

Finally, note that any PMP-extremal \( \bar{u} \) is itself a trivial solution to \eqref{clucomp}. In Section~\ref{ssec:infinite}, we demonstrate that the set of solutions \( u \in \mathcal{U} \) to \eqref{clucomp} is non-empty for any \( \bar{u} \) within a much broader class of optimal control problems in metric spaces. Furthermore, in Appendix~\ref{ssec:fbm}, we show that the desired ``comparison process'' can be synthesized using a ``model-predictive'' discretization approach similar to \cite{krasovskii2011game}.

\begin{remark}[The full context of $(\bm{LP}')$ is not needed]\label{re:00}
As one might have noticed, the formulation of the increment formula \eqref{incr} and the subsequent variational analysis of the problem $(\widetilde{P})$ does not require the full statement of its ``super-counterpart''. What we really need is the family $(\mathfrak L_t(u_t))_{t \in I}$ of differential operators and the class \( \bm{\mathcal D} \doteq \dom(\bm \Phi^{u}) \) of test functions.
\end{remark}

\section{Mean-field control setup}
\label{sec:nonloc_cont}

We now take the next step up the generalization ladder by applying our approach to a nonlinear version of the super-problem \( (\bm{LP}') \) from Section~\ref{sec:odes}: 
\begin{align}
  \text{Minimize}\quad & \mathcal I[u] \doteq  \ell(\mu_T^u)\nonumber \\
  \text{subject to}: \quad &\partial_t \mu_t + \mathrm{div}\left(F_t(\mu_t,u_t)\, \mu_t\right) = 0, \quad \mu_0=\vartheta,\label{eq:PDE}\\
  &u\in \mathcal{U}.\notag
\end{align}
This problem is formulated over the space \( \mathcal{P}_2 \doteq \mathcal{P}_2(\mathbb{R}^n) \) of probability measures \( \mu \) on \( \mathbb{R}^n \) with finite second moments \( \int |\mathrm x|^2 \d \mu({\rm x)} \). Throughout this section, we adopt the notations:
$$\int = \int_{\mathbb{R}^n}, \quad \iint = \int_{\mathbb{R}^n \times \mathbb{R}^n},\quad \text{ etc.}
$$ 

We initially adhere to the class \( \mathcal{U} \) of generalized controls, where \( U \) is a compact set. The map \( \ell \colon \mathcal{P}_2 \to \mathbb{R} \) is a given cost functional, and \[ F \colon I \times (\mathbb{R}^n \times \mathcal{P}_2 \times U) \to \mathbb{R}^n \] is a controlled \emph{nonlocal vector field}. 

The PDE \eqref{eq:PDE} is commonly referred to as a \emph{nonlocal continuity equation}; as before, it is understood in the distributional sense  \eqref{eq:contW}, provided by setting $f_t({\rm u})=F_t(\mu_t, {\rm u})$.

\begin{remark}
As above, the measure \( \vartheta \) can be interpreted as the distribution law of a random initial state \( x_0 \) on a probability space \( (\Omega, \mathcal{F}, \mathbb{P}) \), giving rise to an ensemble of sample paths \( x_t \colon \Omega \to \mathbb{R}^n \) with laws \( \mu_t = (x_t)_\sharp \mathbb{P} \). The term ``nonlocal'' reflects the fact that the driving vector field depends on the distribution \( \mu_t \) of the samples \( x_t \) across the entire space \( \mathbb{R}^n \).

In this framework, the choice of the control set \( \mathcal{U} \) --- being independent of the phase variable \( x \in \mathbb{R}^n \) and acting uniformly and equivalently on all samples \( x \) --- aligns the problem with the field of ensemble control \cite{liEnsembleControlFiniteDimensional2011}.
\end{remark}

Recall that \( \mathcal{P}_2 \) is naturally equipped with the \( L_2 \)-Kantorovich (Wasserstein) distance:
\[
  W_2(\mu, \nu) = \left(\inf_{\Pi \in \Gamma(\mu, \nu)} \iint |x-y|^2 \d \Pi(x, y)\right)^{1/2},
\]
where \( \Gamma(\mu, \nu) \) denotes the set of probability measures \( \Pi \) on \( \mathbb{R}^n \times \mathbb{R}^n \) whose marginal projections coincide with \( \mu \) and \( \nu \), i.e., \( \pi^1_{\sharp} \Pi = \mu \) and \( \pi^2_{\sharp} \Pi = \nu \).
The elements of \( \Gamma(\mu, \nu) \) are called transport plans between \( \mu \) and \( \nu \). It is well known that the metric space \( (\mathcal{P}_2, W_2) \) is complete. Moreover, \( \mu_k \overset{W_2}{\to} \mu \) if and only if 
$$\int \phi \, \d\mu_k \to \int \phi \, \d\mu$$ 
for all quadratically bounded continuous functions  
\(
  |\phi(x)| \leq C \left(1+|x|^2\right)\), \(C > 0
\)
(see~\cite{ambrosioGradientFlowsMetric2005, panaretosInvitationStatisticsWasserstein2020, santambrogioOptimalTransportApplied2015} for further details).

The stated problem now replaces the base problem \( (\widetilde{P}) \) from the previous section, with the PDE \eqref{eq:PDE} serving as the ``ordinary system''. For simplicity, we will denote this problem by \( (P) \) (dropping the ``tilde''). 

This new nonlinear problem will be immersed into the corresponding linear super-problem $(\bm{LP}')$ of the next-level abstraction. As explained in Remark~\ref{re:00}, our analysis does not require the complete formulation of the super-problem; we only need to figure out the generating family of the flow in \eqref{eq:PDE} and an appropriate class of test functions.

\subsection{Differentiability in $\mathcal P_2$}\label{ssec:dif-W}

Along with the metric structure, \( \mathcal{P}_2 \) admits a kind  of \emph{differential} structure~\cite{ottoGeometryDissipativeEvolution2001,ambrosioGradientFlowsMetric2005}, provided by a natural notion of tangent space. 
The \emph{tangent space} to \( \mathcal{P}_2 \) at a point \( \mu \) is a weighted Lebesgue space \( {L}^2_{\mu} \doteq {L}^2_{\mu}(\mathbb{R}^n;\mathbb{R}^n) \) modulo the following equivalence relation: functions \( v_1,v_2 \in {L}^2_{\mu} \) are set to be  equivalent if \( \mathrm{div}(v_1\mu) = \mathrm{div}(v_2\mu) \) in the sense of distributions, i.e.,
\[
\int (v_1 \cdot \nabla\eta) \d \mu =
\int (v_2 \cdot \nabla\eta) \d \mu \quad \forall \eta \in C^1_c(\mathbb{R}^n).
\]

\begin{remark}
  The equivalence relation on \( L^2_{\mu} \) can be explained as follows. 
  It is known~\cite{ambrosioGradientFlowsMetric2005} that for any absolutely continuous curve \( t \mapsto \mu_t \) on \( \mathcal{P}_2 \), there exists a family \( v_t \in L^2_{\mu} \) (for a.a. \( t \)) such that the equation
  \[
    \partial_t \mu_t + \div(v_t \mu_t) = 0
  \]
  holds in the distributional sense. 
  This family is non-unique because any other family \( w_t \in L^2_{\mu} \) such that \( \div(v_t \mu_t) = \div(w_t \mu_t) \) for a.e. \( t \in{I}\) produces the same curve.
  Introducing the above equivalence relation on \( L^2_{\mu} \), enables treating all such families as equivalent. 
  The corresponding equivalence class \( v_t \in L^2_{\mu} \) (we use the same notation for both the equivalence class and its representatives) can be considered as a unique tangent vector to the curve \( \mu_t \) at the time \( t \). 
  This tangent vector is usually denoted by \( \dot\mu_t \).
\end{remark}

The explicit definition of the tangent space endows  $\mathcal P_2$ with a structure akin to that of a Hilbert manifold, leading to the following ``geometric'' concept of differentiability.

\begin{definition}[Differentiable functions]\label{def:diff_fn}
~%
  \begin{enumerate}
    \item (\emph{Differentiability}) A function \( \phi\colon \mathcal{P}_2\to \mathbb{R} \) is called \emph{differentiable} at \( \mu\in \mathcal{P}_2 \) if there exists a linear bounded map \( \mathbf{d}\phi_{\mu} \colon {L}^2_{\mu}\to \mathbb{R}\) such that for any \( v\in L^2_{\mu} \) it holds
\[
  \lim_{\varepsilon\to 0+}\frac{1}{\varepsilon}\left|\phi\left((\id+\varepsilon v)_\sharp \mu\right)- \phi(\mu)- \varepsilon\, \mathbf{d}\phi_{\mu}(v) \right| = 0.
\]
The map \( \mathbf{d}\phi_{\mu} \) is termed the \emph{differential} (or \emph{derivative}) of \( \phi \) at \( \mu \).

Since \( {L}^2_{\mu} \) is a Hilbert space, for any \( v\in {L}^2_{\mu} \), there exists a unique element \( w\in {L}^2_{\mu} \) such that \( \mathbf{d}\phi_{\mu}(v) = \left< v,w \right>_{\mu} \).
This element \( w \) is called the \emph{gradient} of \( \phi \) at \( \mu \), and is denoted by \( \bm \nabla \phi(\mu) \).

 \item (\emph{Uniform differentiability}) We say that a function \( \phi\colon \mathcal{P}_2\to \mathbb{R} \) is \emph{uniformly differentiable} if it is differentiable at every \( \mu\in \mathcal{P}_2 \), and there exists \( C>0 \) such that
  \[
    \left|\phi\left((\id +\varepsilon v)_{\sharp}\mu\right) - \phi(\mu) - \varepsilon\mathbf{d}\phi_{\mu} (v)\right|\le C\|v\|^{2}_{\mu} \varepsilon^2,
  \]
  for all \( \mu\in \mathcal{P}_2 \), \( v\in {L}^2_{\mu} \) and \( \varepsilon\in \mathbb{R} \). 

  \item (\emph{Uniformly equidifferentiable families}) A family $(\phi_\alpha)_{\alpha\in \mathcal A}$ of functions \( \phi_\alpha\colon \mathcal{P}_2\to \mathbb{R} \) is called uniformly equidifferentiable if all $\phi_\alpha$ are uniformly differentiable with a common constant $C>0$, independent of $\alpha$. 
  \end{enumerate}
\end{definition}

\begin{definition}[Differentiability of maps]
\label{def:diff_mp}
\quad
    \begin{enumerate}
        \item (\emph{Differentiability}) A map \( \Phi \colon \mathcal{P}_2 \to \mathcal{P}_2 \) is called \emph{differentiable} if, for any point $\mu\in \mathcal{P}_2$, there exists a linear bounded map \(\Phi_{\star,\mu} \colon {L}^2_{\mu} \to L^2_{\Phi(\mu)} \) such that for any \( v \in {L}^2_{\mu} \) the function \( w = \Phi_{\star,\mu}v \) satisfies the relation:
\[
  \lim_{\varepsilon\to 0+}\frac{1}{\varepsilon}W_1\left(\Phi\left((\id+\varepsilon v)_\sharp \mu\right), (\id +\varepsilon w)_\sharp \Phi(\mu)\right) = 0.
\]
The map \(\Phi_{\star,\mu} \) is called the \emph{derivative} of \( \Phi \) at \( \mu \).

\item (\emph{Uniform differentiability}) A map \( \Phi\colon \mathcal{P}_2\to \mathcal{P}_{2} \) is \emph{uniformly differentiable} if it is differentiable at any \( \mu\in \mathcal{P}_2 \), and there exists \( C>0 \) such that
  \[
    \|v\|_{\Phi(\mu)}\le C\|w\|_{\mu}\quad \mbox{ and }\quad
    W_1\left(\Phi\left((\id+\varepsilon v)_{\sharp}\mu\right),(\id+\varepsilon w)_{\sharp}\Phi(\mu)\right)\le C\|v\|^{2}_{\mu}\varepsilon^2,
  \]
  for all \( \mu\in \mathcal{P}_2 \), \( v\in {L}^2_{\mu} \), and \( \varepsilon\in \mathbb{R} \), where \( w \doteq \Phi_{\star,\mu} v \).

  \item (\emph{Uniformly equidifferentiable families}) A family $(\Phi_\alpha)_{\alpha\in \mathcal A}$ of maps \( \Phi_\alpha\colon \mathcal{P}_2\to \mathcal{P}_2 \) is called uniformly equidifferentiable if all $\Phi_\alpha$ are uniformly differentiable with a common constant $C>0$, independent of $\alpha$. 
    \end{enumerate}
\end{definition}

\begin{remark}
  \label{rem:stronger}
  One might be tempted to use \( W_2 \) in the above definition. However, in this case, we may not be able to prove the differentiability of nonlocal flows, as discussed in Remark~\ref{rem:problem}.
\end{remark}

As in the case of Riemannian manifolds, differentiable maps \( \Phi \colon \mathcal{P}_2 \to \mathcal{P}_2 \) act on functions, vector fields and \( 1 \)-forms. Precisely, we can introduce the corresponding pushforward and pullback operations as follows:
  \begin{enumerate}
    \item A \emph{pullback} of a function \( \phi\colon \mathcal{P}_2 \to \mathbb{R}\) through $\Phi$ is the function \( \Phi^{\star}\phi\colon \mathcal{P}_2\to \mathbb{R} \) defined by \( \Phi^{\star}\phi\doteq \phi\circ \Phi \).
    \item A \emph{pushforward} of a vector field \( v \) on \( \mathcal{P}_2 \) is the vector field \( \Phi_{\star}v \) on \( \mathcal{P}_2 \) defined by
    \[
    (\Phi_{\star} v)_{\Phi(\mu)} \doteq \Phi_{\star,\mu}(v_{\mu}).
    \]

    \item A \emph{pullback} of a \( 1 \)-form \( \omega \) on \( \mathcal{P}_2 \) is the \( 1 \)-form \( \Phi^{\star}\omega \) on \( \mathcal{P}_2 \) defined by
  \[
    (\Phi^{\star}\omega)_\mu(v) \doteq \omega_{\Phi(\mu)}(\Phi_{\star,\mu}v), \quad v\in L_{\mu}^2.
  \]
  \end{enumerate}
  
  Remark that, in the above definitions, the regularity of functions, vector fields, and \( 1 \)-forms does not play an essential role. A key factor is the regularity of \( \Phi \).

  \begin{lemma}\label{lem:Wcommut}
    The following statements hold:
    \begin{enumerate}[{\rm (i)}]
      \item Let \( \Phi\colon \mathcal{P}_2\to \mathcal{P}_2 \) be uniformly differentiable and \( \phi\colon \mathcal{P}_2\to \mathbb{R} \) be uniformly differentiable and Lipschitz w.r.t. the distance \( W_1 \).
  Then, the pullback \( \Phi^{\star}\phi\colon \mathcal{P}_2\to \mathbb{R} \) is uniformly differentiable, moreover
\[
  \Phi^\star(\mathbf{d} \phi) = \mathbf{d}(\Phi^\star \phi).
\]

\item Given $\phi$ as above, let a family $(\Phi_t)_{t \in I}$ be uniformly equidif\-fe\-ren\-ti\-able. Then, such is the family \(\left(\Phi^{\star}_t\phi\right)_{t\in I}\).
    \end{enumerate}
\end{lemma}
\begin{proof}
  \textbf{1.} The Lipschitz continuity of \( \phi \) with respect to the distance \( W_1 \) implies that
   \[
     \left|\phi\left(\Phi((\id + \varepsilon v)_\sharp\mu)\right)-\phi\left((\id + \varepsilon w)_\sharp \Phi(\mu)\right)\right|\le
     \Lip_{W_1}(\phi) W_1 \left(\Phi((\id + \varepsilon v)_\sharp\mu),(\id + \varepsilon w)_\sharp \Phi(\mu)\right),
   \]
   for any \( v\in {L}^2_{\mu} \) and \( w\in L_{\Phi(\mu)}^2 \).
   If \( w = \Phi_{\star,\mu}(v) \), the uniform differentiability of \( \Phi \) implies that
   \[
     \|w\|_{\Phi(\mu)}\le C_1\|v\|_{\mu}\quad \text{ and }\quad
    W_1 \left(\Phi((\id + \varepsilon v)_\sharp\mu),(\id + \varepsilon w)_\sharp \Phi(\mu)\right)\le C_1\|v\|^2_{\mu}\varepsilon^2,
  \]
  for some \( C_1>0 \), and all \( \mu\in \mathcal{P}_2\),  \(v\in {L}^2_{\mu}\), and \(\varepsilon\in \mathbb{R}\).
   In particular,
  \[
    \left|\phi\left(\Phi((\id + \varepsilon v)_\sharp\mu)\right)-\phi\left((\id + \varepsilon w)_\sharp \Phi(\mu)\right)\right|\le \Lip_{W_1}(\phi) \, C_1 \|v\|^2_{\mu}\varepsilon^{2},
  \]
  for all \( \mu\in \mathcal{P}_2\),  \(v\in {L}^2_{\mu}\), and \(\varepsilon\in \mathbb{R}\). Since \( \phi \) is also uniformly differentiable, it holds
  \[
    \left| \phi\left((\id + \varepsilon w)_\sharp \Phi(\mu)\right) - \phi(\Phi(\mu)) - \varepsilon (\mathbf{d}\phi)_{\Phi(\mu)}(w)\right| \le C_2\|w\|^2_{\Phi(\mu)}\varepsilon^{2},
  \]
  for each \( \mu\in \mathcal{P}_2\),  \(w\in L^2_{\Phi(\mu)}\), and \(\varepsilon\in \mathbb{R}\).
  Taking \( w = \Phi_{\star,\mu}(v) \), and combining the above inequalities, we get
  \[
    \left| \phi\left(\Phi((\id + \varepsilon v)_\sharp\mu)\right) - \phi(\Phi(\mu)) - \varepsilon (\mathbf{d}\phi)_{\Phi(\mu)}(w)\right| \le \left(\Lip_{W_1}(\phi )C_1 + C_1^2C_2\right)\|v\|^2_{\mu}\varepsilon^{2}.
\]
Thus, \( v \mapsto (\mathbf{d}\phi )_{\Phi(\mu)}(\Phi_{\star,\mu}(v))\) satisfies the approximation property of the derivative of \( \Phi^{\star}\phi \) at \( \mu \).
  It remains to note that this map is linear and bounded as the composition of linear bounded maps.

  \textbf{2.} Due to the uniformity of the constants $C_1$ and $C_2$ w.r.t. $t \in I$, all previous estimates remain intact, leading to the second assertion.  
\end{proof}

\subsection{Test functions}

In this section, we present an important example of uniformly differentiable functions. Specifically, we demonstrate that every element of the class \( \bm{\mathcal D} \), defined below belongs, to this category (cf.~\cite{CardMaster2019}).

\begin{definition}
  We say that a function \( \phi\colon \mathcal{P}_2\to \mathbb{R} \) belongs to the class \( \bm{\mathcal D} \) if it satisfies the following properties:
\begin{enumerate}[(i)]
  \item There exists a map \( \frac{\delta \phi}{\delta \mu}\colon \mathcal{P}_2 \times \mathbb{R}^n\to \mathbb{R} \), called the flat derivative of \( \phi \), such that
\[
  \phi(\mu_{1}) - \phi(\mu_{0}) = \int_0^1\int \frac{\delta \phi}{\delta\mu}\left((1-t)\mu_0 + t\mu_1,\mathrm{x}\right)\d (\mu_1-\mu_0)(\mathrm{x})\d t
\]
and
    \[
    \int \frac{\delta \phi}{\delta \mu}(\mu,\mathrm{x})\d \mu(\mathrm{x}) = 0
        \]
        for all \( \mu_0,\mu_1\in \mathcal{P}_c \).

\smallskip
        
      \item The map \( \frac{\delta \phi}{\delta \mu} \) is continuous, bounded, and differentiable in \( \mathrm{x} \).

\smallskip

      \item The map \( \nabla\frac{\delta \phi}{\delta \mu} \colon \mathcal{P}_2 \times \mathbb{R}^n\to \mathbb{R}^n\) (the gradient of \( \frac{\delta \phi}{\delta \mu} \) in \( \mathrm{x} \)) is Lipschitz continuous and bounded.
\end{enumerate}
\end{definition}

The elements \(\phi \in \bm{\mathcal D}\) will be referred to as \emph{test functions}.\footnote{Strictly speaking, a more consistent terminology would be ``super test functions,'' emphasizing the distinction from test function --- elements of the space $ C^1_c(\mathbb{R}^n)$ --- appearing in the definition \eqref{eq:contW} of the distributional solution to the continuity equation. However, we will omit the qualifier ``super'' for brevity.}
Some of their remarkable properties are collected in the following lemma.

\begin{lemma}
  \label{lem:testlip}
  Any function \( \phi \in  \bm{\mathcal D}\) has the following properties:
  \begin{enumerate}[{\rm (1)}]
    \item \( \phi \) is Lipschitz w.r.t. the distances \( W_1 \) and \( W_2 \).
    \item \( \phi \) is differentiable in the sense of Definition~\ref{def:diff_fn} and its gradient \( \bm \nabla \phi \) coincides with \( \nabla\frac{\delta \phi}{\delta \mu}  \).
    \item More specifically, \( \phi \) is uniformly differentiable, i.e.,   for all \( \mu\in \mathcal{P}_2 \), \( v \in {L}^2_{\mu} \), and \( \varepsilon\in \mathbb{R} \), we have
    \begin{displaymath}
    \left|\phi \left((\id+\varepsilon v)_{\sharp}\mu  \right) - \phi(\mu) - \varepsilon \mathbf{d}\phi_{\mu}(v)\right| \le 2\Lip(\bm\nabla \phi)\|v\|^2_{\mu} \varepsilon^2.
  \end{displaymath}
    \item The composition \( t\mapsto \phi(\mu_t) \) with any absolutely continuous curve \( \mu\colon I\to \mathcal{P}_2 \) is absolutely continuous, and
\begin{equation}
    \frac{d}{dt} \phi(\mu_t) = (\mathbf{d}\phi)_{\mu_t}(\dot \mu_t) = \left< \bm \nabla \phi(\mu_t), \dot \mu_t \right>_{\mu_t},\label{form-diff}
\end{equation}
for a.e. \( t\in I \). In fact, this holds even if \( \phi \) is merely Lipschitz continuous and differentiable.
  \end{enumerate}
\end{lemma}
\begin{proof}
  \textbf{1.} Recall that 
\[
  \phi(\mu_{1}) - \phi(\mu_{0}) = \int_0^1\int \frac{\delta \phi}{\delta\mu}\left((1-t)\mu_0 + t\mu_1,x\right)\d (\mu_1-\mu_0)(x)\d t.
\]
By the third property of test functions, \( \nabla\frac{\delta \phi}{\delta \mu} \) is bounded. Therefore, \( \mathrm{x} \mapsto \frac{\delta \phi}{\delta\mu}\left(\mu,\mathrm{x}\right)  \) is Lipschitz, uniformly for all \( \mu\in \mathcal{P}_2 \).
Now the \( W_1 \)-Lipschitz continuity of \( \phi \) follows from the dual representation of \( W_1 \):
\[
  W_1(\mu_1,\mu_2) = \sup\left\{\int \phi\d(\mu_1-\mu_2)\;\colon\;\phi\in C(\mathbb{R}^n;\mathbb{R}),\;\Lip(\phi)\le 1 \right\}.
\]
Since \( W_1(\mu_0,\mu_1)\le W_2(\mu_1,\mu_2) \), the function \( \phi \) is Lipschitz w.r.t. \( W_2 \) as well.

\textbf{2.}  Differentiability (assertion (2)) follows from~\cite[Proposition 2.1]{chertovskihOptimalControlNonlocal2023}. 
  The proof of assertion (3) is similar to that of Lemma~\ref{lem:O2vf} below.

   \textbf{3.} It remains to prove Assertion (4). The map \( t\mapsto \phi(\mu_t) \) is absolutely continuous as a composition of a Lipschitz and an absolutely continuous function.
Let \( v_t\doteq\dot \mu_t \).
Thanks to~\cite[Proposition 8.4.6]{ambrosioGradientFlowsMetric2005}, for a.e. \( t \), it holds \( W_2(\mu_{t+h}, (\id +h v_t)_{\sharp}\mu_t) = o(h) \) as \( h\to 0 \). Fix some \( t \) with this property.
Then
\begin{equation}
\phi(\mu_{t+h}) - \phi(\mu_t) = \left[\phi(\mu_{t+h}) - \phi((\id+h v_t)_{\sharp}\mu_t)\right] + \left[\phi\left((\id+hv_t)_{\sharp}\mu_t\right) - \phi(\mu_t)\right].
  \label{eq:xixi}
\end{equation}
The first difference from the right is \( o(h) \), because
\[
\left|\phi(\mu_{t+h}) - \phi((\id+hv_t)_{\sharp}\mu_t)\right|\le \Lip(\phi)\, W_2(\mu_{t+h},(\id+hv_t)_{\sharp}\mu_t) = o(h).
\]
The latter formula holds due to the Lipschitz continuity of \( \phi \) and the choice of \( t \).
The second difference in the right-hand side of~\eqref{eq:xixi} equals
\[
h(\mathbf{d} \phi)_{\mu_t}(v_t) + o(h),
\]
due to the differentiability of \( \phi \).
\end{proof}

\subsection{Metric properties of flows}

We begin with the standard regularity assumptions on \( F \) that date back at least to~\cite{piccoliTransportEquationNonlocal2013}.

\begin{tcolorbox}
\begin{assumption}
    \label{F1}
    ~
\begin{itemize}
	\item \( F \) is measurable in \( t \);
	\item \( F \) is bounded and Lipschitz in \( \mathrm{x} \) and \( \mu \), i.e., there exists \( M > 0 \) such that
	    \begin{gather*}
        \left|F_t(\mathrm{x},\mu)\right|\le M,\quad
	\left|F_{t}(\mathrm{x},\mu)-F_{t}(\mathrm{x}',\mu')\right|\le M\left(|\mathrm{x}-\mathrm{x}'|+W_{2}(\mu,\mu')\right),
	  \end{gather*}
		for all \( \mathrm{x},\mathrm{x}'\in \mathbb{R}^n \), \( \mu,\mu'\in \mathcal{P}_2 \), \( t\in I \).
\end{itemize}
\end{assumption}
\end{tcolorbox}
\begin{proposition}
  \label{prop:basic}
  Let Assumptions \ref{F1} hold.
\begin{enumerate}[\rm (1)]
  \item For each \( \mu\in \mathcal{P}_{2} \), there exists a unique map \(X^{\mu}\colon  I \times I\times \mathbb{R}^d \to \mathbb{R}^d \) satisfying the differential equation
\[
  \frac{d}{dt} X^{\mu}_{s,t} = F_t\left(X^{\mu}_{s,t},X^{\mu}_{s,t\sharp}\mu\right),\quad X^{\mu}_{s,t}=\id.
				\]
  \item The map \( \Phi\colon  I \times  I \times \mathcal{P}_2\to \mathcal{P}_2 \), defined by the rule \( \Phi_{s,t}(\mu) \doteq (X^{\mu}_{s,t})_\sharp \mu \), is the flow of the dynamical system 
        \begin{equation}
          \partial_t \varrho_t + \div(F_t(\varrho_t)\,\varrho_t) = 0.
          \label{eq:nlcont}
        \end{equation}
In other words, \(t \mapsto \mu_t = \Phi_{s,t}(\vartheta) \) is a unique distributional solution to the equation~\eqref{eq:nlcont} with the initial condition \( \mu_{s} = \vartheta \).

\item The map \( (s,t,\mu)\mapsto \Phi_{s,t}(\mu) \) is \( M \)-Lipschitz as a function of \( t \) or \( s \), and \( L \)-Lipschitz as a function of \( \mu \), with some \( L>0 \) depending only on \( M \) and \( T \).
\end{enumerate}
\end{proposition}
For a proof, see, e.g.,~\cite{piccoliTransportEquationNonlocal2013}.

\subsection{Differential properties of flows}

To study the differential properties of \( \Phi \), we require an additional regularity assumption:

\begin{tcolorbox}
\begin{assumption}
\label{F2}
~
\begin{itemize}
  \item \( F \) is differentiable in \( \mathrm{x} \) and \( \mu \), and the corresponding derivatives are Lipschitz and bounded:
  \begin{gather*}
    \left|DF_t(x,\mu)\right|\le M, \quad \left|\mathbf{D}F(x,\mu,y)\right|\le M,\\
    \left|DF_t(\mathrm{x},\mu) - DF_t(\mathrm{x}',\mu')\right|\le M\left(|\mathrm{x}-\mathrm{x}'|+W_{2}(\mu,\mu')\right),\\
    \left|\mathbf{D}F_t(\mathrm{x},\mu,\mathrm{y}) - \mathbf{D}F_t(\mathrm{x}',\mu',\mathrm{y}')\right|\le M\left(W_{2}(\mu,\mu') + |\mathrm{x}-\mathrm{x}'|+ |\mathrm{y}-\mathrm{y}'|\right),
  \end{gather*}
  for all \( \mathrm{x},\mathrm{x}',\mathrm{y},\mathrm{y}'\in \mathbb{R}^n \), \( \mu,\mu'\in \mathcal{P}_2 \), \( t\in I \).
\end{itemize}
\end{assumption}
\end{tcolorbox}

The following two technical results are central to the subsequent analysis. 
To streamline the presentation, their proofs are deferred to Section~\ref{sec:Wflowdiff} of the Appendix.
\begin{proposition}[Pushforward of vectors]
  \label{prop:Wflowdiff}
  Let the nonlocal vector field \( F \) satisfy \ref{F1} and \ref{F2}, and let \( \Phi \) denote its flow.
  Then, for each \( a,b\in I \), the function \( \Phi_{a,b} \colon \mathcal{P}_2 \to \mathcal{P}_2 \) is uniformly differentiable with a constant \( C>0 \) depending only on \( F \).
  Its derivative \( (\Phi_{a,b})_{\star,\mu}\colon {L}^2_{\mu}\to L^{2}_{\Phi_{a,b}(\mu)} \) acts as \( v_a \mapsto v_b \).
Here, the time-dependent vector field \( v_t\in L^{2}_{\mu_t} \) is defined along the curve \(t \mapsto \mu_{t}= \Phi_{a,t}(\mu)\) as
 \( v_t \doteq w_t\circ (X^{\mu}_{a,t})^{-1} \), where \( w \) is a unique solution to the linear equation
  \begin{align}
    \partial_t w_t(\mathrm{x}) = & \ DF_t\left(X^{\mu}_{a,t}(\mathrm{x}),X^{\mu}_{a,t\sharp}\mu\right)w_t(\mathrm{x}) \nonumber\\
    + &  \,\int\mathbf{D}F_t\left(X^{\mu}_{a,t}(\mathrm{x}),X^{\mu}_{a,t\sharp}\mu, X^{\mu}_{a,t}(\mathrm{y})\right)w_t(\mathrm{y})\d\mu(\mathrm{y}),\label{eq:Ww}
  \end{align}
  with the initial condition \( w_{a}(\mathrm{x}) = v_a(\mathrm{x}) \).
\end{proposition}

Proposition~\ref{prop:Wflowdiff} describes how the map \( \Phi_{a,b,\star} \) pushes forward vector fields.
Now, we need to understand how \( \Phi_{a,b}^{\star} \) pulls back \( 1 \)-forms.
Recall that any \( 1 \)-form \( \omega \) on the space of measures can be characterized by a vector field \( p \) with a property that \( \omega_{\mu}(v_{\mu})=\left< p_{\mu},v_{\mu} \right>_{\mu}\), for any vector field \( v \).

\begin{proposition}[Pullback of covectors]
  \label{prop:Wpull}
Let the nonlocal vector field \( F \) satisfy assumptions~\ref{F1}
and \ref{F2}, and let \( \Phi \) denote its flow.
  The map \( (\Phi_{a,b})^{\star}_{\mu}\colon L^2_{\Phi_{a,b}(\mu)}\to L^{2}_{\mu} \) acts as \(
  p_b \mapsto p_a \).
  Here, the time-dependent vector field \( p_t\in L^2_{\mu_t} \) is defined along the curve \( \mu_t = \Phi_{a,t}(\mu) \) by the formula
  \[
  p_t({\rm x}) \doteq \int {\rm y}\d \gamma^{\rm x}_t({\rm y}), 
  \]
with \( \gamma \) being a unique solution to the Hamiltonian equation
  \begin{equation}
    \label{eq:ham}
\partial_t \gamma_t + \mathrm{div}_{({\rm x,y})}\left( {\mathbb J}_{2n}\bm \nabla\mathcal{H}_t(\gamma_{t})\right) = 0,\quad \gamma_b = (\id, p_b)_{\sharp}\mu_b,
  \end{equation}
where
\[
\mathcal{H}_t(\gamma) = \iint{\rm  y}\cdot F_t(\pi^1_{\sharp}\gamma,{\rm x})\d \gamma({\rm x,y}),
\]
the family \( \gamma^{\rm x}_t \) is obtained by disintegration of \( \gamma_t \) w.r.t. \(\mu_t\), and \( \mathbb J_{2n} \) is the symplectic matrix $\left(\begin{array}{cc}
    0 & \id  \\
   -\id  & 0
\end{array}\right)$ of dimension \( {2n} \).
\end{proposition}

\begin{remark}
  Above, we introduced two vector fields \( v_t \) and \( p_t \), acting along the curve \( t \mapsto \mu_t = \Phi_{a,t}(\mu) \).
  The first one is uniquely determined by an initial tangent vector \( v_a\in L^2_{\mu_a} \), while the second one
  is uniquely determined by a terminal tangent vector \( p_{b}\in L^2_{\mu_b} \).
  No matter which \( v_a \) and \( p_b \) are selected, the corresponding vector fields \( v_t \) and \( \psi_t \) satisfy the identity
  \[
    \left< p_t,v_t \right>_{\mu_t} = \text{const} \quad \forall t\in [a,b]
  \]
  (see Fig.~\ref{fig:diff}).
  We do not know whether the curves \( t\mapsto v_t \) and \( t\mapsto p_t \) are themselves trajectories of certain dynamical systems.
  However,  they can be expressed as follows: \( v_t = w_t\circ (X^{\mu}_{0,t})^{-1} \) and \( p_t(\mathrm{x}) = \displaystyle \int \mathrm y\d\gamma_t^{\rm x}(\mathrm y) \), where \( w_t \) satisfies~\eqref{eq:Ww} and \( \gamma_t \) obeys~\eqref{eq:ham}.
  The details are discussed in the proof of Proposition~\ref{prop:Wpull}.
\end{remark}

\begin{figure}
  \centering
  \includegraphics[width=0.9\textwidth]{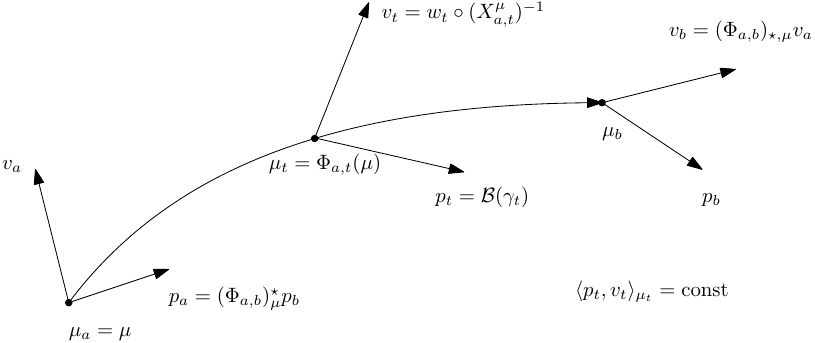}
  \caption{Pushforward of vectors and pullback of covectors.}
  \label{fig:diff}
\end{figure}

\begin{remark}
  Let \( \mu\in \mathcal{P}_2(\mathbb{R}^n) \) and \( \gamma \) be any lift of \( \mu \) to \( \mathcal{P}_2(\mathbb{R}^n \times \mathbb{R}^n) \), that is, \( \gamma \) is an arbitrary transport plan such that \( \pi^1_\sharp \gamma = \mu \).
  The quantity 
  \[
    \mathcal{B}(\gamma)(\mathrm x) \doteq \int {\rm y}\d \gamma^{\rm x}(\mathrm y) 
  \]
  is well-known in the theory of Wasserstein spaces~\cite{ambrosioGradientFlowsMetric2005}, where it is called the \emph{barycentric projection}.
  In fact, if we introduce the set 
  \[
    \mathrm{Lift}_\mu\doteq \left\{\gamma\in \mathcal{P}_2(\mathbb{R}^n \times \mathbb{R}^n)\;\colon\; \pi^1_\sharp \gamma = \mu\right\},
  \]
comprising all possible lifts of \( \mu \), then, as one can easily check, the barycentric projection becomes a well-defined map from \( \mathrm{Lift}_{\mu} \) to the tangent space \( L^2_{\mu} \), i.e., \( \mathcal{B}\colon \textrm{Lift}_{\mu}\to L^2_{\mu} \).

Finally, as noted in~\cite[Section 3.3]{chertovskihOptimalControlNonlocal2023}, any solution \( \gamma_t \) of the Hamiltonian equation~\eqref{eq:ham} is a lift of the initial trajectory \( \mu_t = \Phi_{a,t}(\mu) \) to \( \mathcal{P}_2(\mathbb{R}^n \times \mathbb{R}^n) \). 
\end{remark}

\begin{remark}
The proposed geometric approach essentially relies on the ``local'' structure of the space $\mathcal{P}_2$. However, the actual implementation of our control-theoretical framework only requires differentiating test functions $\phi \in \bm{\mathcal{D}}$ along the trajectories of the control system. This observation allows Definitions~\ref{def:diff_fn} and \ref{def:diff_mp} to be replaced by a purely analytical notion of ``differentiability w.r.t. the flow'', which does not exploit the local properties of the state space. This concept will be introduced in Section~\ref{sec:generapp} for an abstract dynamical system on a general metric space.
\end{remark}

\subsection{Regularity assumptions}

We are now ready to state basic regularity assumptions for the main optimization problem of this section.

\begin{tcolorbox}
\begin{assumption}
\label{a3}
~
\begin{itemize}
  \item \( F\colon I \times (\mathbb{R}^n \times \mathcal{P}_2 \times U) \to \mathbb{R}^n \) is a Carath\'eodory map.
  \item For each \( u \in \mathcal{U} \), the nonlocal vector field \( (t,\mathrm x,\mu)\mapsto F_t(\mathrm{x},\mu,u_t)  \) satisfies the assumptions \ref{F1} and \ref{F2} with a constant \( M \), common for all \( u \).
  \item The cost function \( \ell\colon \mathcal{P}_2\to \mathbb{R} \) belongs to the class \( \bm{\mathcal D} \).
\end{itemize}
\end{assumption}
\end{tcolorbox}

For the remainder of this section, we work under Assumption \ref{a3}.

\subsection{Super-adjoint equation}

We associate equation~\eqref{eq:PDE} with a family of  flows, \( \Phi^u \), parameterized by the generalized control \( u \in {\mathcal U} \).  Formula~\eqref{form-diff} enables differentiation of a test function \( \phi \in \bm{\mathcal D} \) along the flow, thereby introducing a family \( \mathfrak{L}^u \) of linear operators acting on test functions:  
\begin{align}\label{Lnlc}
\begin{array}{c}
\displaystyle\left(\mathfrak{L}_t^u\phi\right)(\mu) \doteq \frac{\partial}{\partial h}\big|_{h=0}\phi\left(\Phi_{t,t+h}[u](\mu)\right) 
  = \left<\bm \nabla \phi(\mu), F_t(\mu,u_t)\right>_{\mu},\\[0.4cm]
  \mu \in \mathcal{P}_2, \ \text{a.a. } t \in I.
  \end{array}
\end{align}

Following the general approach, we define the super-adjoint \( \bm p \) as a mild solution to the backward problem:
\begin{equation}
  \label{eq:Wadjoint}
  \left\{\partial_t + \mathfrak{L}_t^u\right\}\bm p 
  \doteq \partial_t \bm p_t(\mu) + \left< \bm \nabla \bm p_t(\mu), F_t(\mu,u_t) \right>_{\mu} = 0, 
  \quad \bm p_T(\mu) = \ell(\mu).
\end{equation}

Drawing parallels with the classical case in Section~\ref{ssec:cnlp}, a promising candidate for the role of a mild solution is the composition:
\begin{equation}\label{p-repr}
  \bm p_t(\mu) = \ell(\Phi_{t,T}(\mu)), \quad (t, \mu) \in I \times \mathcal{P}_2.
\end{equation}
Validation of this ansatz involves the following technical result, showing that the gradient \( \bm \nabla \bm p_t \) is continuous in the sense that the \( 1 \)-form \( \mathbf{d}\bm p_t \) maps bounded Lipschitz vector fields to continuous functions.

\begin{proposition}
  \label{prop:dp_cont}
  Let \( u \in {\mathcal U}\), and \( \Phi \doteq \Phi^u \) be the flow of \( F_t(\cdot,u_t) \). 
Let \( g \) be a time-dependent vector field on \( \mathcal{P}_2 \) satisfying Assumption \ref{F1}.
Then, 
\begin{enumerate}[{\rm (1)}]
  \item The family $\{\mathbf{p}_t\}_{t \in I}$, defined by~\eqref{p-repr}, is uniformly equidifferentiable. 

  \item For any \( s\in I \), the map \( I \times \mathcal{P}_2 \to \mathbb{R} \) given by 
    \[
      (t,\mu) \mapsto \left< \bm \nabla \mathbf{p}_{t}(\mu), g_s(\mu) \right>_{\mu}
    \]
    is continuous.

  \item For any compact set  \( \mathcal{K} \subset \mathcal{P}_2 \), the family of functions \( \mathcal{K} \to \mathbb{R} \) given by
    \[
    \mu \mapsto \left< \bm \nabla \mathbf{p}_{t}(\mu), g_t(\mu) \right>_{\mu}, \quad t\in I,
    \]
    is equicontinuous.
\end{enumerate}
\end{proposition}
The first part of this assertion is a trivial consequence of Proposition~\ref{prop:Wflowdiff} and Lemmas~\ref{lem:Wcommut}, \ref{lem:testlip}. Proofs of the remaining parts are exposed in Appendix~\ref{app:pr-dp_cont}.

\begin{proposition}\label{lem:psol}
  The map \eqref{p-repr} is a mild solution to~\eqref{eq:Wadjoint}.
\end{proposition}
\begin{proof}
  Fix \( \mu\in \mathcal{P}_2 \).
  Proposition~\ref{prop:basic}(3) together with Lemma~\ref{lem:testlip}(1) imply that the map \( t\mapsto \mathbf{p}_t(\mu) \) is Lipschitz continuous.
  Therefore, for almost all \( t\in I \), the limit
  \begin{equation}
    \label{eq:tmp1}
    \partial_t \mathbf{p}_t(\mu) = \lim_{h\to 0}\frac{ \mathbf{p}_{t+h}(\mu) - \mathbf{p}_t(\mu)}{h}=-\lim_{h\to 0}\frac{\mathbf{p}_{t+h}(\Phi_{t,t+h}(\mu)) - \mathbf{p}_{t+h}(\mu)}{h}
  \end{equation}
  is well-defined.  
  At the same time, as the flow of \( F_t(\cdot, u_t) \), the map \( \Phi \) satisfies
  \begin{equation}
    \label{eq:tmp2}
   \lim_{h\to 0}\frac{W_2\left(\Phi_{t,t+h}(\mu), (\id +h F_t(\mu,u_t))_{\sharp}\mu\right)}{h} = 0,
  \end{equation}
  for almost all \( t\in I \).  
  Clearly, the set of all \( t\in I \) for which both~\eqref{eq:tmp1} and~\eqref{eq:tmp2} hold has full Lebesgue measure. Fix some \( t \) from this set.
  Due to Proposition~\ref{prop:basic}(3), the map \( \mu \mapsto \mathbf{p}_{t+h}(\mu) \) is Lipschitz with modulus \( L\Lip(\ell) \). 
  Therefore, using~\eqref{eq:tmp1} and~\eqref{eq:tmp2}, we may write 
  \[
    \partial_t \mathbf{p}_t(\mu) =-\lim_{h\to 0}\frac{\mathbf{p}_{t+h}\left((\id +h F_t(\mu,u_t))_{\sharp}\mu\right) - \mathbf{p}_{t+h}(\mu)}{h}.
  \]
  By Proposition~\ref{prop:dp_cont}(1), the family \( (\mathbf{p}_{t})_{t \in I} \) is uniformly equidifferentiable. 
  Thus,
  \[
    \partial_t \mathbf{p}_t(\mu) =-\lim_{h\to 0}\left< \bm \nabla \mathbf{p}_{t+h}(\mu), F_t(\mu,u_t) \right>_{\mu} = -\left< \bm \nabla \mathbf{p}_{t}(\mu), F_t(\mu,u_t) \right>_{\mu},
  \]
  where we use Proposition~\ref{prop:dp_cont}(2) to prove the last equality.
\end{proof}
\begin{proposition}
  \label{prop:sysdif}
  Let \( g \) be a time-dependent vector field on \( \mathcal{P}_2 \) satisfying Assumption~\ref{F1}, \( t\mapsto \mu_t \) be any of its trajectories, and \( \bm p \) be defined as in \eqref{p-repr}.
  Then, for a.a. \( t\in I \), the function \( t\mapsto \bm p_t(\mu_t) \) is absolutely continuous, and
  \begin{align}
    \frac{\d}{\d{t}} \bm p_t(\mu_t)
    &=\partial_t \bm p_t(\mu_t) + (\mathbf{d} \bm p_t)_{\mu_t}(g_t)\nonumber\\
    &=\left< \bm \nabla \bm p_t(\mu_t),g_t(\mu_t) - F_t(\mu_t,u_t)\right>_{\mu_t}.\label{der-pmu}
  \end{align}
\end{proposition}
The proof of Proposition~\ref{prop:sysdif} we defer to Appendix~\ref{sec:Wflowdiff}.

\begin{corollary}\label{uni_p}
    The functional \eqref{p-repr} is the only solution to \eqref{eq:Wadjoint} satisfying the product rule \eqref{der-pmu}.
\end{corollary}
\begin{proof}
  Let $\bm p$ be a solution to \eqref{eq:Wadjoint} such that~\eqref{der-pmu} holds true for any \( g \) satifying \ref{F1}. Let \( \Phi \) be the flow of \( F_t(\cdot,u_t) \) and $\bm q_{s,t} \doteq \bm p_t\circ \Phi_{s,t}$. By the differentiation rule \eqref{der-pmu} written for \( \mu_t\doteq \Phi_{s,t}(\vartheta) \) with any \( \vartheta\in \mathcal{P}_2 \), the mapping $t\mapsto \bm q_{s,t}$ is constant. Then, we can write:
\[
    \bm p_t \doteq \bm q_{t,t} = \bm q_{t, T} \doteq \ell\circ \Phi_{t, T},
\]
showing that $\bm p$ is uniquely defined.
\end{proof}
\begin{remark}[Linear case]
  When the driving vector field \( F \) is independent of \( \mu \), i.e. \( F_t(\mathrm{x}, \mu, \mathrm{u}) \equiv f_t(\mathrm{x},  \mathrm{u}) \), we have, in addition to the super-adjoint \eqref{p-repr}, the usual adjoint state given by a solution to the transport equation \eqref{eq:transp} with the terminal condition \( \frac{\delta \ell}{\delta \mu} \). Denote the former by \( \bm p \) and the latter by \( p \). One may easily check that these two objects are related as
\[
    \bm p_t(\mu) = \langle\mu, p_t\rangle, \ \ t \in I.
\]
\end{remark}

\subsection{Exact increment formula and 1-variation}

Consider a pair of generalized controls \( u \) and \( \bar{u} \).
Denote by \( \Phi \) and \( \bar{\Phi} \) the corresponding flows of \eqref{eq:PDE}, and abbreviate \( \mu_t = \Phi_{0,t}(\vartheta) \), \( \bar{\mu}_t = \bar{\Phi}_{0,t}(\vartheta) \) and \( \bar{ \mathbf{p}}_t = \ell(\bar{\Phi}_{t,T}(\vartheta)) \).

Provided by the established regularity of the super-adjoint function $\bm p$, and Definition \ref{p-repr}, we can represent the increment of the functional in the problem $(P)$ on the pair $(\bar{u}, {u})$ as in \eqref{eq:LCexact}, which is written explicitly as
\begin{align}
\mathcal{I}[u] - \mathcal{I}[\bar{u}] & \, = \int_I \left< \bm \nabla \bar{\bm{p}} _t(\mu_t),F_t(\mu_t,u_t)-F_t(\mu_t,\bar{u}_{t})\right>_{\mu_t}\d t\notag\\
& \, \doteq \int_I \mathbf{d}\bar{\bm p}_t \left(F_t(\cdot,u_t)-F_t(\cdot,\bar{u}_{t})\right)\big|_{\mu_t}\d t.\label{mf_incr}
\end{align}

Substituting the convex combination~\eqref{eq:perturb} into~\eqref{mf_incr}, we obtain 
\[
 \mathcal{I}[u^\epsilon] - \mathcal{I}[\bar{u}] \, = \epsilon \int_I \left< \bm \nabla \bar{\bm{p}} _t(\mu_t^\epsilon),F_t(\mu_t^\epsilon,u_t)-F_t(\mu_t^\epsilon,\bar{u}_{t})\right>_{\mu_t^\epsilon}\d t,
\]
where \( \mu^\epsilon \) is the trajectory corresponding to \( u^\epsilon \). As \( \epsilon\to 0 \), we have \( \mu^\epsilon_t\to \bar{\mu}_t \), for any \( t\in I \) (see, e.g.,~\cite{pogodaevNonlocalBalanceEquations2022}). 
Now, thanks to Proposition~\ref{prop:dp_cont}(3), we can pass to the limit obtaining the first-order formula
\begin{equation}
  \label{eq:Wfirst-order}
  \lim_{\epsilon\to 0+}\frac{\mathcal{I}[u^\epsilon] - \mathcal{I}[\bar{u}]}{\epsilon} \, = \int_I \left< \bm \nabla \bar{\bm{p}} _t(\bar{\mu}_t),F_t(\bar{\mu}_t,u_t)-F_t(\bar{\mu}_t,\bar{u}_{t})\right>_{\bar{\mu}_t}\d t.
\end{equation}

We now prove an analog of Proposition~\ref{prop:NLdiff} establishing the connection between the adjoint \( p \) and the super-adjoint \( \mathbf{p} \).

\begin{proposition}
  \label{prop:NLCfirst}
  Let \( u \in \mathcal U \), \( \mu_t \) and \( p_t \) be the corresponding primal and co-trajectories, and \( \bm p_t \) the associated solution to the super-adjoint system~\eqref{eq:Wadjoint}.
  Then \( \bm \nabla \bm p_t(\mu_t) = p_t \) for all \( t\in I \).
\end{proposition}
\begin{proof}
  Let \( \Phi \) be the flow of \( (t,\mu)\mapsto F_t(\mu_t,u_t) \).
  Fix some \( t\in I \).
  The representation formula yields
\[
  \mathbf{p}_t = \ell \circ \Phi_{t,T} = \Phi_{t,T}^\star p_T.
\]
By differentiating with respect to \( \mu \) we obtain, thanks to Lemma~\ref{lem:Wcommut}, that
\[
  \mathbf{d} \mathbf{p}_t = \mathbf{d} \left(\Phi_{t,T}^\star \mathbf{p}_T\right) = \Phi_{t,T}^\star (\mathbf{d}\mathbf{p}_T).
\]
Let \( v_s \) with \( s\in [t,T] \) be the vector field from Proposition~\ref{prop:Wflowdiff}.
In this case, \( v_T = (\Phi_{t,T})_{\star,\mu}v_t \).
Plugging this formula in the previous identity and recalling the definition of the pullback for \( 1 \)-forms, we obtain:
\(
  (\mathbf{d} \mathbf{p}_t)_{\mu}(v_t) = (\mathbf{d}\mathbf{p}_{T})_{\Phi_{t,T}(\mu)}(v_T).
\)
Replacing \( \mu \) with \( \mu_t \) gives
\(
  (\mathbf{d} \mathbf{p}_t)_{\mu_t}(v_t) = (\mathbf{d}\mathbf{p}_{T})_{\mu_{T}}(v_{T})
\),
or, in the 'non-geometric' notation,
\[
  \left\langle\bm \nabla \mathbf{p}_t(\mu_t), v_t\right\rangle_{\mu_t} = \left\langle\bm \nabla \mathbf{p}_T(\mu_T), v_T\right\rangle_{\mu_T}.
\]
If we let \( \widetilde p_t \doteq \bm\nabla \mathbf{p}_t(\mu_t) \), then the latter identity can be written as
\(
\left\langle\widetilde p_t,v_t\right\rangle_{\mu_t} = \left< \ell, v_{T} \right>_{\mu_T}
\)
Recall that the co-trajectory \( p_t \) satisfies the same identity, that is,
\(
\left\langle\widetilde p_t,v_t\right\rangle_{\mu_t} = \left\langle p_t,v_t\right\rangle_{\mu_t}
\). 
Since the initial tangent vector \( v_t \) can be arbitrarily selected from \( L^2_{\mu_t} \), we conclude that \( \widetilde p_t= p_t \).
\end{proof}

With the representations~\eqref{eq:Wfirst-order} and \( \bar{p}_t=\mathbf{\nabla}\bar{\mathbf{p}}_t(\bar{\mu_t}) \) at hand, we can derive the PMP in a manner analogous to Section~\ref{sec:pmp-c}:
\begin{theorem}[PMP]
  \label{thm:W-PMP}
  Suppose that Assumptions~\ref{a3} hold. If \( \bar{u} \) is optimal in~\eqref{eq:PDE}, then there exists a solution \( \bar{\gamma} \) to the Hamiltonian system~\eqref{eq:ham} associated to the vector field \( F_t(\cdot,\bar{u}_t) \) such that the identities
  \[
  \left<\bar{p}_t, F_t(\bar{\mu}_t,\bar{u}_t)\right>_{\bar{\mu}_t} = \min_{\omega\in U} \left<\bar{p}_t,F_t(\bar{\mu}_t,\delta_{\omega}) \right>_{\bar{\mu}_t}
  \]
  hold for a.a. \( t\in I \), where \( \bar{p}_t = \mathcal{B}(\bar{\gamma}_t) \) and \( \bar{\mu}_t=\pi^1_\sharp\bar{\gamma}_t \).
\end{theorem}

To the best of our knowledge, the second-order variation and the corresponding necessary optimality conditions for the stated mean-field control problem have not been addressed in the literature. Theoretically, the second-order variational analysis could be derived from the exact representation \eqref{mf_incr} by analogy with Proposition~\ref{propo2ord}. However, this path and the resulting formulation are expected to be cumbersome to present.

In contrast, the feedback NOC in the spirit of Theorem~\ref{thm:fbc} follows trivially from \eqref{mf_incr} by applying the same arguments:
\begin{theorem}[Feedback NOC]
  \label{thm:W-FNOC}
  Suppose that Assumptions~\ref{a3} hold. Let \( \bar{u} \) be optimal in~\eqref{eq:PDE}. Denote by \( \bar{\mu} \) and \( \mathbf{\bar{p}} \) the corresponding trajectory and super-adjoint. Let a control \( u \) and the corresponding trajectory \( \mu \) satisfy
\[
  \left<\mathbf{\nabla}\bar{\mathbf{p}}_t(\mu_t), F_t({\mu}_t,{u}_t)\right>_{{\mu}_t}  = \min_{\omega\in U} \left<\mathbf{\nabla}\bar{\mathbf{p}}_t(\mu_t),F_t(\mu_t,\delta_{\omega}) \right>_{\mu_t},
\]
for a.a. \( t\in I \). Then, the relation 
\[
\left<\mathbf{\nabla}\bar{\mathbf{p}}_t(\mu_t), F_t({\mu}_t,{u}_t)\right>_{{\mu}_t} = 
  \left<\mathbf{\nabla}\bar{\mathbf{p}}_t(\mu_t), F_t({\mu}_t,\bar{u}_t)\right>_{{\mu}_t}
\]
holds for a.a. \( t\in I \). Moreover, \( \mathcal I[u] = \mathcal I[\bar{u}] \).
\end{theorem}

This result fits the general framework, addressed in Section~\ref{ssec:infinite} for the next generation of the problem $(P)$, we are going to introduce.

\section{Abstract Setup}\label{sec:generapp}

We now aim to extend the immersion-duality argument, illustrated above, towards abstract control problems involving general-form dynamical systems in a metric space $\mathcal X = (\mathcal X, d_{\mathcal X})$. Recall that a forward dynamical system on $\mathcal X$ is represented by its flow $\Phi$, satisfying the chain rule \eqref{eq:family}. We restrict this definition to a fixed finite time horizon $I \doteq [0, T]$, and denote \[
\Delta \doteq \{(s,t) : 0 \leq s \leq t \leq T\}.
\] Since we are dealing with continuous-time dynamics, it is also natural to assume that the state space $\mathcal X$ is path connected, in particular, it enables non-identical flows.

\subsection{External generators and linear super-equation}\label{sec:G}
A key step in our analysis is to derive a ``differential representation'' of the dynamical system $(\mathcal{X}, \Phi)$. If the system were linear, such a representation would be provided by the standard concept of the \emph{generating family} of linear differential operators \cite[Section 4.9]{kolokoltsov2019differential}. However, when $\mathcal{X}$ lacks a linear structure, this classical definition is inapplicable.

A natural idea, then, is to introduce a kind of ``external'' generator for $\Phi$ by first transforming $(\mathcal{X}, \Phi)$ into a linear system on a ``lifted space'', whose generators are defined in the natural way. This paragraph is devoted to the implementation of this strategy.

Denote by $\bm{X} = C_b(\mathcal{X})$ the collection of bounded continuous test functions $\mathcal{X} \to \mathbb{F}$, where $\mathbb{F} \in \{\mathbb{R}, \mathbb{C}\}$. Recall that $\bm{X}$ is a Banach space, as it is equipped with the usual supremum norm, and the dual space $\bm{X}'$ is isomorphic to the collection \( rba(\mathcal{X}) \) of bounded regular finitely additive measures on \( \mathcal{X} \), endowed with the total variation norm \cite[Theorem~7.9.1]{bogachevMeasureTheoryVolume2007}.

As discussed in the Introduction, the system $(\mathcal{X}, \Phi)$ gives rise to two mutually adjoint linear systems:
\[
(\bm \Phi, \bm X') \quad \text{and} \quad (\bm \Psi, \bm X).
\]
The first (primal, forward) one is defined by the pushforward \eqref{pushforward}, and the second (dual, backward) system is introduced through the pullback \eqref{pull}. Note that $\bm{\Psi}_{t,s}$ are bounded (specifically, 1-Lipschitz) linear automorphisms of $\bm{X}$, and the systems are related by the following identity:
\begin{equation}\label{w-*Phi}
    \langle \bm \Phi_{s,t} \mu, \phi \rangle = \langle \mu, \bm{\Psi}_{t,s} \phi \rangle, \quad \mu \in \bm{X}', \quad \phi \in \bm{X}. 
\end{equation}

The system $(\bm X', \bm \Phi)$ serves as a ``distributed version'' of $(\mathcal{X}, \Phi)$ and is equivalent to the latter when the domain of $\vartheta$ is restricted to the set $\delta(\mathcal{X})$. We now show that $(\bm X', \bm \Phi)$ admits a differential representation, provided by the following definition:

\begin{definition}[Differentiability w.r.t. the flow]\label{def:metr_differentiable}
~
\begin{enumerate}
    \item A function $\phi\colon \mathcal{X} \to \mathbb{F}$ is said to be differentiable w.r.t. the flow $\Phi$ (\emph{$\Phi$-differentiable}, for short) at a point $\mathrm{x} \in \mathcal{X}$ if the following limit exists for a.e. $t \in [0,T)$:
    \begin{align}\label{L}
        ({\mathfrak{L}_t{\phi}})(\mathrm{x}) \doteq \partial_h \big|_{h=0}  \left(\bm{\Psi}_{t+h, t} \phi\right)(\mathrm{x}) = \lim_{h \to 0+}\frac{\phi\big(\Phi_{t, t+h}(\mathrm{x})\big) - \phi(\mathrm{x})}{h}.
    \end{align}
    If there exists a subset $J \subseteq I$ of full Lebesgue measure such that the condition \eqref{L} holds for all $t \in J$ and  $\mathrm{x} \in \mathcal{X}$, the function $\phi$ is said to be $\Phi$-differentiable on $\mathcal{X}$.

    \item A function $\phi \in \bm{X}$ that is $\Phi$-differentiable on $\mathcal{X}$ is called \emph{continuously $\Phi$-differentiable} on the same set if $\mathfrak{L}_t{\phi} \in \bm{X}$ for all $t \in J$.

    \item Let $\mathcal{A}$ be an index set. A family $(\phi_\alpha)_{\alpha \in \mathcal{A}}$ of functions $\phi_\alpha \in \bm{X}$ is \emph{uniformly equidifferentiable with respect to $\Phi$} if the relation
    \begin{gather}
        \lim_{h \to 0+}\sup_{\alpha \in \mathcal{A}}\frac{1}{h}\Big\|\left[\bm{\Psi}_{t+h, t}-\id - h\mathfrak{L}_t\right]\phi_\alpha\Big\|_{\bm{X}} = 0
    \end{gather}
    holds for a.a. $t \in [0,T)$. If $\mathcal{A}$ is a singleton, the function $\phi_\alpha \equiv \phi$ is called \emph{uniformly $\Phi$-differentiable}.
\end{enumerate}
\end{definition}

\begin{remark}Some aspects of Definition~\ref{def:metr_differentiable} deserve a comment:
\begin{enumerate}[(a)]
    \item Any $\phi$ satisfying Definition~\ref{def:metr_differentiable}(3) automatically satisfies Definition~\ref{def:metr_differentiable}(2), as this follows from the uniform limit theorem.

    \item The definition \eqref{L} reproduces the classical notion of the Lie derivative \cite{Lee2013}. While the original concept involves differentiating a tensor with respect to a vector field, its application to functions (0-tensors) requires only the knowledge of the flow, not the vector field itself. This flexibility allows for its adaptation to dynamical systems in abstract spaces.

    \item Any function, even a non-measurable one, is continuously differentiable with respect to the trivial flow $\Phi \equiv \id$ with $\mathfrak L_t \phi \equiv 0$.
\end{enumerate}
\end{remark}

Let $C^{1\text{-}\Phi} = C^{1\text{-}\Phi}(\mathcal{X})$ denote the collection of all uniformly $\Phi$-differentiable functions. It is easy to verify that $C^{1\text{-}\Phi}$ is a vector subspace of $\bm{X}$. The class  $C^{1\text{-}\Phi}$ is non-empty (since it contains constant maps), but --- as Examples~\ref{ex2} and \ref{ex_5} below show --- its actual richness essentially depends on the regularity of $\Phi$.

By rephrasing \eqref{L}, we conclude that the map $t \mapsto \bm{\Psi}_{t,s} \phi$ satisfies, for each $\phi \in C^{1\text{-}\Phi}$, any $s \in [0,T)$, and a.e. $t \in [s, T)$, the following ODE on $\bm{X}$:
\begin{align}
    \partial_t( \bm{\Psi}_{t,s} \phi) &\doteq \partial_t[ \phi(\Phi_{s,t})] = \partial_h \big|_{h=0} \phi\left(\Phi_{t, t+h} \circ \Phi_{s,t}\right) \notag\\
    &\doteq \partial_h \big|_{h=0} \left(\bm{\Psi}_{t+h, t} \phi\right)(\Phi_{s,t}) \doteq \left(\mathfrak{L}_t \phi\right)(\Phi_{s,t}) 
    = \bm{\Psi}_{t,s} \mathfrak{L}_t \phi.\label{psiT}
\end{align}
From this representation and the equality \eqref{w-*Phi}, we derive, for any fixed $\vartheta \in \bm X'$:
\[
    \frac{\d}{\d t}\langle \bm \Phi_{s,t} \vartheta, \phi \rangle = \langle \vartheta, \partial_t \bm{\Psi}_{t,s} \phi \rangle = \langle \vartheta, \bm{\Psi}_{t,s} {\mathfrak{L}_t}{\phi} \rangle \doteq \langle \bm \Phi_{s,t} \vartheta, {\mathfrak{L}_t} \phi \rangle. 
\]
When $\phi$ runs over an adequate subspace $\bm{\mathcal D} \subset C^{1\mbox{-}\Phi}$, this relation constitutes the distributional form of a linear differential equation in the dual space $\bm{X}'$, with the set $\bm{\mathcal D}$ representing the class of test functions. 

If we denote $\bm{x}_t \doteq \bm \Phi_{0,t} \vartheta$ and introduce the formal adjoint ${\mathfrak{L}_t}'$ of the operator ${\mathfrak{L}_t}$ via the relation:
\[
\langle {\mathfrak{L}_t}' \mu, \phi \rangle \doteq \langle \mu, \mathfrak{L}_t \phi \rangle, \quad \mu\in \bm X', \quad \varphi \in C^{1\text{-}\Phi},
\]
this equation can be expressed symbolically as
\begin{equation}\label{X}
\dot{\bm{x}}_t = {\mathfrak{L}_t}' \bm{x}_t, \quad \bm{x}_0 = \vartheta.
\end{equation}

Motivated by this discussion, one can say that the family $\mathfrak{L}=({\mathfrak{L}_t})_{t \in I}$ generates the forward flow $\bm{\Phi}$ on $\bm X'$ in the weak* sense. Returning to the original framework, we term $\mathfrak{L}$ an \emph{external generating family} of $\Phi$.

In the rest of the section, we exemplify the utility of Definition~\ref{def:metr_differentiable} and provide a rigorous interpretation of the ``PDE'' \eqref{X}. In particular, Proposition~\ref{propos:ext} below introduces natural (and sufficiently strong)  assumptions on the space $\bm{\mathcal D}$, which guarantee the existence and uniqueness of a solution to \eqref{X}, showing, in turn, that the linear operators $\mathfrak{L}_t$ generate the backward flow $\bm{\Psi}$ on $\bm{\mathcal D}$ in the conventional sense.

\subsubsection{Differentiability along the flow}\label{ssec:externa}

In general, the (uniform equi-) $\Phi$-differentiability of $\phi$ is a joint property of the pair $(\phi, \Phi)$. As a basic regularity assumption, we have:
\begin{tcolorbox}
\begin{assumption}\label{a4}
~
\begin{itemize}  
\item The flow map \(\Phi\colon \Delta \times \mathcal{X} \to \mathcal{X}\), defined for variables \((s, t, \mathrm{x})\), is Lipschitz continuous w.r.t. the second variable \(t \in [s, T]\), with a Lipschitz constant \(\mathrm{Lip}(\Phi)\) independent of \(s\).

\item Additionally, it is Lipschitz continuous in the third variable, uniformly in \((s, t) \in \Delta\), with the same Lipschitz constant.  
\end{itemize}
\end{assumption}
\end{tcolorbox}

With these hypotheses, for any Lipschitz function $\phi$, the map $t \mapsto (\bm{\Psi}_{t, s} \phi)(\mathrm{x}) \doteq \phi\left(\Phi_{s, t}(\mathrm{x})\right)$ remains Lipschitz on $[s, T]$ for all $\mathrm{x} \in \mathcal{X}$. In this case, the value $({\mathfrak{L}_t{\phi}})(\mathrm{x})$ exists at a.e. $t \in [s, T]$, and the following representation holds:
\[
   \left(\left[\bm{\Psi}_{t+h, t} -\id\right]\phi\right)(\mathrm{x}) \doteq {\phi}\big(\Phi_{t, t+h}(\mathrm{x})\big) - {\phi}(\mathrm{x}) = \int_t^{t+h} ({\mathfrak{L}_s{\phi}})(\mathrm{x}) \, ds.
\]
Hence, we can estimate:
\[
    \left\|\left[\bm{\Psi}_{t+h, t}-\id - h\mathfrak{L}_t\right]\phi\right\|_{\bm{X}} \leq  \sup_{\mathrm x \in \mathcal X} \int_t^{t+h}\big|[\mathfrak{L}_s- \mathfrak{L}_t]{\phi}(\mathrm{x})\big| \, ds,
\]
which shows that a sufficient condition for the uniform $\Phi$-differentiability of a Lipschitz function is a ``uniform (w.r.t. $\mathrm{x}$) Lebesgue differentiability'' of the map $t \mapsto (\mathfrak{L}_t \phi)(\mathrm{x })$:
\begin{equation}\label{Ls}
    \lim_{h\to 0^+}\frac{1}{h}\int_t^{t+h}\big\|[\mathfrak{L}_s- \mathfrak{L}_t]{\phi}\big\|_{\bm{X}} \, ds = 0.
\end{equation}
Similarly, a sufficient condition for the uniform $\Phi$-equidifferentiability of a family $(\phi_\alpha)_{\alpha \in \mathcal{A}}$ of uniformly Lipschitz functions $\phi_\alpha$ is derived as
\begin{equation}\label{Ls-2}
    \lim_{h\to 0^+}\frac{1}{h}\int_t^{t+h}\sup_{\alpha \in \mathcal{A}}\big\|[\mathfrak{L}_s- \mathfrak{L}_t]{\phi_\alpha}\big\|_{\bm{X}} \, ds = 0,
\end{equation}
for a.e. $t\in I$.

\begin{example}\label{ex2}
Remaining within assumptions \ref{a4}, let $\mathcal X = \R^n$, and $\Phi$ be the flow of the control-linear vector field \[f_t(\mathrm x) = \sum_{k=1}^m f^k(\mathrm x) \, u_k(t),\] where $f^k\colon \R^n \to \R^n$ are bounded and Lipschitz, and $u =(u_1, u_2, \ldots, u_m) \in L^\infty(I;U)$ with a compact $U \subset \R^m$.  Then, any function $\phi \in C^1(\R^n) \cap C_b(\R^n)$ with a bounded gradient belongs to $C^{1\text{-}\Phi}(\R^n)$. 

This conjecture follows from the fact that operators $\mathfrak{L}_t$ act on elements $\phi \in \bm{\mathcal D}$ as the usual Lie derivatives, $\mathfrak{L}_t \phi = \sum u_k(t)\, \nabla \phi \cdot f^k $, which meet the condition \eqref{Ls} due to the obvious estimate:
\[
     \int_t^{t+h}\big\|[\mathfrak{L}_s- \mathfrak{L}_t]{\phi}\big\|_{\bm X} \d s \leq \|\nabla \phi\|_{\bm X}\,\max_{k}\|f^k\|_\infty \sum_k \int_t^{t+h}\big|u_k(s)- u_k(t)\big| \d s
\]
and Lebesgue's differentiation theorem.
\end{example}

\begin{example}\label{ex0}
Let $\mathcal X = \mathcal P_2$, and $\Phi = \Phi$ be the flow of a nonlocal vector field \( F \colon I \times (\mathbb{R}^n \times \mathcal{P}_2 \times U) \to \mathbb{R}^n \) with the structure 
\[F_t(x, \mu) = \sum_{j=1}^m u_j(t) F^j(x, \mu),
\] where all $F^j$ satisfy assumptions \ref{F1} and \ref{F2}. Then, any function $\phi \in \bm{\mathcal D}$ satisfies \eqref{Ls}, implying the inclusion $\phi \in C^{1\text{-}\Phi}(\mathcal P_2)$. 
\end{example}

\begin{example}\label{ex_5} Let $\mathcal X$, $\Phi$, $f$ and $\phi$ be as in Example~\ref{ex2}. Assume, in addition, that $\nabla \phi$ and $D f$ are bounded. Then, the functional family $(\phi_\alpha)_\Delta$,  $\alpha \doteq (a, b)$,  
$\phi_\alpha \doteq \bm \Psi_{b,a} \phi \doteq \phi\circ \Phi_\alpha$, is uniformly $\Phi$-equidifferentiable. This follows from the estimates:  
\begin{align*}
    \big|[\mathfrak{L}_s-\mathfrak L_t] \phi_\alpha\big| & \leq \sum_k\left|\nabla \phi\circ \Phi_{\alpha}\right|\cdot |D \Phi_{\alpha}|\cdot \left|f^k\right|\cdot \big|u_k(t) - u_k(s)\big|,
\end{align*}
and the representation \cite[Theorems~2.3.2, 2.2.3]{ABressan_BPiccoli_2007a}:
\[
    D\Phi_{a,s} = \id + \displaystyle\int_a^s D f_{\tau}(\Phi_{a,\tau})D\Phi_{a,\tau}   \d \tau,
\]
giving the uniform bound:
\[
|D\Phi_\alpha| \leq  \exp\left\{\sum_k\||D f^k|\|_{\bm X}\|u_k\|_{L_1[a,b]}\right\}
\]
due to Gr\"{o}nwall's lemma.

\end{example}

\begin{example}\label{ex:6}
   In addition to the hypotheses of Example~\ref{ex0}, suppose that the maps $DF^j$ and $\bm D F^j$ are bounded. Then, for any $\phi \in \bm{\mathcal D}$, the family $(\phi \circ \Phi_{\alpha})_{\alpha \in \Delta}$ is uniformly equidifferentiable w.r.t. $\Phi$. This fact is established in Appendix~\ref{apex4}
\end{example}

\subsubsection{Well-posedness}
The following result, consistent with \cite[Theorem 4.10.1]{kolokoltsov2019differential}, demonstrates that the linear PDE \eqref{X} is well-defined under an appropriate regularity of the flow \( \Phi \) (cf. also~\cite{STEPANOV20171044,Gigli2014}).
\begin{proposition}\label{propos:ext}
Let $\bm{\mathcal D} \subset C^{1\text{-}\Phi}(\mathcal{X})$ consist of bounded Lipschitz functions.

\begin{enumerate}
\item Assume \ref{a4}. Then, for any $\vartheta \in \bm{X}'$, the function $\bm{x}\colon I \to \bm{X}'$, defined by the actions 
\[
\langle \bm{x}_t, \phi\rangle \doteq \langle \bm \Phi_{0, t}\vartheta,\phi \rangle \doteq \langle\vartheta,\bm{\Psi}_{t,0}\phi\rangle,\quad t \in I, \quad \phi \in \bm{X},
\]
is a distributional solution to the initial value problem \eqref{X} with the class $\bm{\mathcal D}$ of tests functions, namely, the following Newton-Leibniz formula holds: 
\begin{equation}
\langle \bm{x}_t - \vartheta, \phi \rangle 
= \int_0^t \langle \bm{x}_s, \mathfrak{L}_s \phi \rangle \d s,
\quad t \in I, \quad \phi \in \bm{\mathcal D}.\label{weak} 
\end{equation}

\item In addition to \ref{a4}, suppose the following regularity (see Remark~\ref{rem:a}(b) below):
\begin{tcolorbox}
\begin{assumption}
~
\label{a5}
    \begin{itemize}
    \item For any $\phi\in \bm{\mathcal D}$,  the family \((\phi_{\alpha} \doteq \bm{\Psi}_{b,a} \phi )_{\alpha = (a,b) \in \Delta}\) is uniformly equidifferentiable w.r.t. $\Phi$. 
    
    \item Moreover, the equality
        \[
        \lim_{h\to 0+}\left\|\mathfrak{L}_s \left(\bm{\Psi}_{t, s+h} - \bm{\Psi}_{t,s}\right)\phi\right\|_{\bm{X}} = 0 
        \]
        holds for all $\phi \in \bm{\mathcal D}$, $t \in (0, T]$, and a.a. $0 \leq s < t$.
    \end{itemize}
    \end{assumption}
    \end{tcolorbox}
   Then,
    \begin{enumerate}[2a.]
        \item  the family $\mathfrak{L} = (\mathfrak{L}_t)_{t \in I}$ of linear operators \eqref{L} generates the linear backward dynamical system $\bm{\Psi}$ on $\bm{\mathcal D}$, i.e., for any $\phi \in \bm{\mathcal D}$, every $t \in (0, T]$, and a.a. $0 \leq s < t$, it holds: 
        \begin{equation*}
        \lim_{h \to 0+}\frac{1}{h}\left\|\left(\bm{\Psi}_{t, s+h} - \bm{\Psi}_{t,s} + h \mathfrak{L}_s \bm{\Psi}_{t, s}\right)\phi\right\|_{\bm{X}} = 0. 
        \end{equation*} 
        \item  The function $\bm x$ is a unique solution to \eqref{weak} in the sense that any distributional solution $\bm{y}$  to \eqref{X}, such that all maps $s \mapsto \langle \bm{y}_s, \bm{\Psi}_{t,s} \phi\rangle$, $\phi \in \bm{\mathcal D}$, are absolutely continuous, coincides with $\bm{x}$ on $\bm{\mathcal D}$ (Remark~\ref{rem:a}(b)).
    \end{enumerate}

\end{enumerate}
\end{proposition}
\begin{proof}
  \textbf{1.} First, note that, for each \( \phi \in \bm{\mathcal D} \), 
 the composition \( t \mapsto \phi \circ \Phi_{0, t} \) is Lipschitz by the definition of \( \bm{\mathcal D} \) and hypotheses \ref{a4}. Hence, we have for any \( s,t > 0 \):
\[
    \left|\left\langle \vartheta, (\bm \Psi_{s,0} - \bm \Psi_{t,0})\phi\right\rangle\right| = \left|\int \left(\phi\circ\Phi_{0,s} - \phi\circ\Phi_{0,t}\right)\d \vartheta\right| \leq \|\vartheta\|_{\bm X'} \, \Lip(\phi) \, \Lip(\Phi)\, |t-s|,
\]
where the integral w.r.t. a finitely additive measure is defined in the usual way \cite[\S 4.7(iv)]{bogachevMeasureTheoryVolume2007a}. Thus, the map \(t \mapsto \langle \bm{x}_t , \phi \rangle\) is Lipschitz on $I$.

Let us ensure that \( \bm x \) is indeed a solution to \eqref{weak}. To this end, we apply the relation \eqref{psiT} with $s=0$, written as
\[
\bm \Psi_{t,0} \phi = \phi + \int_0^t \bm \Psi_{s,0} \mathfrak L_s \phi\d s,
\]  
and employ the fact that the Lebesgue integral commutes with linear functionals \cite[Sec.~8, Propos.~7]{dinculeanu2014vector}, yielding:
\begin{align*}
\langle \bm x_t-  \vartheta, \phi\rangle \doteq \langle\vartheta,(\bm \Psi_{t,0}-\id)\phi \rangle = \ &  \Big< \vartheta, \int_0^t \bm \Psi_{s,0} \mathfrak L_s \phi \d s \Big>\\
 = \  & \int_0^t \left\langle \vartheta, \bm \Psi_{s,0} \mathfrak L_s \phi\right\rangle \d s \doteq \int_0^{t}\left< \bm x_s,  \mathfrak L_s \phi\right> \d s,
\end{align*}
as is claimed.

\smallskip

\textbf{2.} To prove Assertion 2a, let us express:
    \begin{align*}
        -\big(\bm \Psi_{t, s+h} - \bm \Psi_{t,s} + h\mathfrak L_s \bm \Psi_{t,s}\big)\phi = \big(\bm \Psi_{s+h,s}  {-} \id  - h\mathfrak L_s\big)\bm \Psi_{t,s+h}\phi + h\mathfrak L_s \big(\bm \Psi_{t, s+h} - \bm \Psi_{t, s}\big)\phi.
    \end{align*}
   Denoting  $\alpha \doteq (a,b)$ and $\phi_\alpha \doteq \bm \Psi_{(b,a)}\phi$, we have the estimate:
      \begin{align*}
       \frac{1}{h}\left\|(\bm \Psi_{t, s+h} - \bm \Psi_{t,s} + h\mathfrak L_s \bm \Psi_{t,s})\phi\right\|_{\bm X}  \leq & \frac{1}{h}\sup_{\alpha \in \Delta}\left\|\big(\bm \Psi_{s+h,s}  - \id  - h\mathfrak L_s\big)\phi_\alpha\right\|_{\bm X}\\ 
       & +\left\|\mathfrak L_s \big(\bm \Psi_{t, s+h} - \bm \Psi_{t, s}\big)\phi\right\|_{\bm X}.
    \end{align*}
   Passing to the limit as $h \to 0+$, both terms vanish due to assumption \ref{a5}.

   \smallskip

   \textbf{3.} It remains to prove Assertion 2b. Let \( \bm y \) be any solution \( I \to \bm X' \) to \eqref{weak}. Fix \( t \in (0,T] \) and \( \phi \in \bm{\mathcal D} \), denote \( \eta_s \doteq \langle \bm y_s, \bm \Psi_{t,s} \phi \rangle \), \( 0 \leq s \leq t \), and assume that \( \eta \) is absolutely continuous on \( [s, T] \). 
    
Let us show that \( \d \eta / \d s = 0 \). To this end, we represent:
    \[
      \frac{\left\langle \bm y_{s+h}, \bm \Psi_{t, s+h} \phi\right\rangle - \left\langle \bm y_s, \bm \Psi_{t,s} \phi\right\rangle}{h} =
      \]
      \[
      \big\langle \bm y_{s+h}, 1/h\left(\bm \Psi_{t,s+h} - \bm \Psi_{t,s} + h \mathfrak L_s \bm \Psi_{t, s}\right) \phi\big\rangle - \left\langle \bm y_{s+h}, \mathfrak L_s \bm \Psi_{t, s} \phi \right\rangle  +  1/h\langle \bm y_{s+h} - \bm y_s, \bm \Psi_{t,s} \phi\rangle,
    \]
when \( 0 < s < s+h < t \), and estimate the first term using the Cauchy-Schwartz inequality:
    \begin{align*}
        \left|\big\langle \bm y_{s+h}, \phi_{s,t,h}\big\rangle\right| \leq \left\|\phi_{s,t,h}\right\|_{\bm X}\|\bm y_{s+h}\|_{\bm X'},
    \end{align*}
    where \( \phi_{s,t,h} \doteq \frac{1}{h}\left(\bm \Psi_{t,s+h} - \bm \Psi_{t,s}   + h \mathfrak L_s \bm \Psi_{t, s}\right) \phi \).
      Letting \( h \to 0+ \), and recalling that the set \( \{\bm y_s\colon s \in I\} \) is \( \|\cdot\|_{\bm X'} \)-bounded, this term
     vanishes due to the previous assertion. The second term tends to the value \(-\left\langle \bm y_{s},\mathfrak L_s \bm \Psi_{t, s} \phi \right\rangle\) because \( \bm y \) is weakly* continuous, and the third expression tends to \( \left\langle \bm y_{s},\mathfrak L_s \bm \Psi_{t, s} \phi \right\rangle \) by the very definition of \( \bm y \). 
    
    Now, since \( s \mapsto \eta_s \) is constant, we have:
\(
    \langle \bm y_t, \phi \rangle \doteq \eta_t = \eta_0 \doteq \langle \vartheta , \bm \Psi_{t,0} \phi \rangle
\),
showing that \( \bm y \) coincides with \( \bm x \) as a functional \( \bm{\mathcal D} \to \F \).

\end{proof}

\begin{remark}\label{rem:a}
~
\begin{enumerate}[(a)]
    \item 
A solution $\bm x$ to the system \eqref{weak} need not be unique outside the subset \( \bm{\mathcal D} \). For example, if \( \bm{\mathcal D} \) is composed of constant functions, a solution is any curve \( \bm{y}\colon I \to  \bm X'\) satisfying the relation \( \bm{y}_t(\mathcal{X}) = \vartheta(\mathcal{X}) \), for all \( t \in I \).

\item The hypotheses \ref{a5} are joint assumptions on the pair \( (\Phi, \bm{\mathcal D}) \), and they are far from being straightforward. The following proposition illustrates a sufficient condition for such regularity in the classical setting. 
\end{enumerate}
\end{remark}

\begin{proposition}\label{propo:ex2}
    Let $\mathcal{X}$, $\Phi$, and $f$ be as in Example~\ref{ex_5}. Additionally, assume that \( Df_t \) is uniformly Lipschitz in \( t \in I \). Then: 
    \begin{enumerate}[i)]
        \item the class \( \bm{\mathcal D} \) of test functions can be specified as \( C^{1,1}_b(\R^n) \), the collection of bounded \( C^1 \)-functions with bounded Lipschitz gradients, and 
        \item with this choice of \( \bm{\mathcal D} \), assumptions \ref{a5} are fulfilled.
    \end{enumerate}
\end{proposition}
The proof is sketched in  Appendix~\ref{app:proofex2}. 
\begin{remark}
The passage $(\mathcal{X}, \Phi) \to (\bm{X}, \bm \Psi)$ underpins the theory of Koopman operators \cite{koopman1931hamiltonian, mauroy2020koopman} and Chronological Calculus \cite{agrachevControlTheoryGeometric2004, kipkaExtensionChronologicalCalculus2015}, which, in turn, align with the framework of Statistical Mechanics \cite{nestruev2003smooth}. In all these frameworks, a nonlinear dynamical system on some underlying space $\mathcal{X}$ is transformed into a linear system on the space of sufficiently regular test functions. In the mechanical context, $\mathcal X$ serves as the state space of a physical system, and the part of test functions is played by observables -- physical characteristics of the system state, measured in experiments.

The dual transformation $(\mathcal{X}, \Phi) \to (\bm{X}', \bm \Phi)$ performs a kind of inverse of the classical method of characteristics: given a nonlinear system, we search for a linear ``PDE" whose characteristics are exactly the trajectories of $\Phi$. To our knowledge, this idea was not systematically utilized before in the area of control theory.
\end{remark} 

\begin{remark}
    Any complete metric space \( \mathcal{X} \) admits a local differential structure, provided by the usual definition of the tangent space \cite{burago2001course,Kirchheim1994,Lytchak2005}. With this structure, one can define dynamics on \( \mathcal{X} \) by appropriately generalizing the concept of a differential equation and calculating its local flow $\Phi$ \cite{ColomboR009, LakshmikanthamDiffeqmetr,tabordiff}. This formalism is connected to the frameworks of quasi-differential equations \cite{Plotnikov1998, Panasyuk1985} and mutations \cite{aubin1993mutational}, and is somewhat analogous to the construction of dynamical systems on manifolds. 

Our approach has been developed in a somewhat opposite direction: in contrast to the cited works, we identified an external differential structure of a given global flow $\Phi$ without referring to the local properties of \(\mathcal{X}\) or requiring its completeness.

\end{remark}

\subsection{Optimal control}

Let us return to the framework of control theory, where the flow \( \Phi \) depends on a functional parameter \( u \colon t \mapsto u_t \) from a suitable class \( \mathcal{U} \). The triple \((\mathcal{X}, \Phi, \, \mathcal{U}) \) is termed a \emph{control system}, and its specific structure is discussed in subsequent paragraphs. 

\subsubsection{Control functions}\label{ssec:controls}

Previously, we operated with ``ordinary'' and generalized controls, being measurable (in an appropriate sense) functions $t \mapsto u_t$, $I \to \bm V'$, to the dual of the corresponding Banach space $\bm V$, subject to the so-called geometric constraint: \(u_t \in U\) for a.a. $t \in I$, where $U  \subset\bm V'$ was closed, convex and bounded in the norm $\|\cdot\|_{\bm V'}$.\footnote{For generalized controls, the triple $(\bm V, \bm V', U)$ is specified as $(\R^m, {\R^m}', \mathcal P(\mathrm U))$ with a fixed compact set $\mathrm U \subset \R^m$.} 

A natural progression, motivated by Examples~\ref{ex1} and \ref{ex:2} below, is to let $\bm V$ be any (separable) Banach space over the complete field $\mathbb K$ of scalars, and set:
\[
{\mathcal U}\doteq {L}^\infty_{w^*}(I; U),
\]
where \(L^\infty_{w*}(I; \bm V')\) stands for the Banach space of weakly* measurable functions $u\colon I \to \bm V'$, i.e., the functions having measurable actions $t \mapsto \langle u_t, \mathrm v\rangle$ on the elements $\mathrm{v} \in \bm V$ (see Appendix~\ref{app:meas-w*}).
\begin{remark}
    Recall that, as soon as $\bm V'$ is infinite-dimensional, the measurability of a function $I \to \bm V'$ can be understood in several (at least four) different senses, see, e.g., \cite[Sec.~II]{diestel1977vector}, and the notation $L^\infty(I, \bm V')$ is typically used for the space of functions, which are measurable in Bochner's sense. It worth stressing that the generic weakly* measurable function, $u\in L^\infty_{w*}(I, \bm V')$, doesn't  have to be an element of $L^\infty(I, \bm V')$ \cite[Sec. II, Example 6]{diestel1977vector}). In particular, $u$ is not essentially separably valued, and it is not approximated by simple functions. 
\end{remark}
\begin{remark}[Why do we need the weak* measurability?]
    The compactness of $\mathcal U$, which is necessary to guarantee the existence of the corresponding optimal control problem, 
 is usually established via the Banach-Alaoglu theorem \cite[Theorem 1.2.1]{diestel1977vector}. Application of this result to measurable controls $I \to \bm V'$ requires the classical duality for Lebesgue-Bochner spaces:
\begin{equation}
(L^1(I; \bm V))' = {L}^\infty(I; \bm V').
\label{d-strong}    
\end{equation}
However, the relation \eqref{d-strong} takes place if and only if $\bm V'$ has the Radon-Nikodym property \cite[Sec. IV.1, Theorem 1]{diestel1977vector}, which is true, say, if it is reflexive or separable. At the same time, some cases of control-theoretical interest, such as generalized controls and Lipschitz vector fields from Example~\ref{ex1} below, do not fit these conditions.  

At the same time, for any Banach space $\bm V$, there holds (refer to Appendix~\ref{app:meas-w*}) the following relation:
\begin{equation}
({L}^1(I; \bm V))' = L^\infty_{w*}(I; \bm V'),
\label{d-w}    
\end{equation}
established by the pairing 
\begin{equation}
\langle\!\langle u, v \rangle\!\rangle \doteq \int \left\langle u_t, v_t\right\rangle_{(\bm V', \bm V)} \d t,
\label{wp*}
\end{equation} 
and enabling the use of the Banach-Alaoglu theorem.\footnote{Recall that convergence in the weak* topology is a sort of pointwise convergence. In particular, $u^k\to u$ in $\sigma({L}^\infty_{w*}(I; \bm V'), {L}^1(I; \bm V))$ reads:
$\langle\!\langle u^k, v \rangle\!\rangle \to  \langle\!\langle u, v \rangle\!\rangle$ for all \(v \in {L}^1(I; \bm V)\).}
\end{remark}

Below, we show that the weak* topology on \( U \) can be metricized by some metric \( d_U \). On the resulting metric space \( (U, d_U) \), we can introduce the Borel \( \sigma \)-algebra and say that the map \( t \mapsto u_t \) is measurable if the preimages of Borel subsets of \( U \) are Borel in \( I \). Furthermore, we show that \( t \mapsto u_t \) is measurable in this sense if and only if it is weak* measurable.

\begin{proposition}
   Assume that $\bm V$ is separable. Then, 
   \begin{enumerate}
    \item $\mathcal U$ is compactly metrizable.
    \item Each representative of the class $u \in \mathcal U$ is measurable in the preimage sense. 
   \end{enumerate}
\end{proposition}

\begin{proof} \textbf{1.} Notice that $U$ is compact in $\bm V'$ w.r.t. the induced weak* topology $\sigma(\bm V', \bm V)$ by the Banach-Alaoglu theorem. In addition, since $\bm V$ is separable, the weak* topology on $U$ is metrizable by a (non-translation-invariant) metric given by
\[
    d_{U}(\mathrm u^1, \mathrm u^2) = \sum_{i \in \mathbb N} 2^{-i} \left|\langle \mathrm u^1 - \mathrm u^2, \mathrm v_i\rangle\right|,
\]
for a dense sequence \(\{v_i\}_{i \in \mathbb N} \subset \bm V\). 
In other words, we can endow $U$ with the structure of a compact (and, consequently, separable)  metric space $(U, d_{U})$. By the same line of arguments, the set ${\mathcal U}$ is compact in the corresponding (weak* subspace) topology $\sigma({L}^\infty_{w^*}, {L}^1)$. Furthermore, since ${L}^1(I; \bm V)$ inherits the separability of $\bm V$, $\mathcal{U}$ can also be viewed as a compact (and separable) metric space. 

\textbf{2.} With a slight abuse of notation, we name a representative of the class $u \in \mathcal U$ by the same letter. Following \cite[Theorem I.4.20]{warga1972optimal}, the measurability of a function $u \colon I \to (U, d_{U})$ to a separable metric space is equivalent to the measurability of the functions $t \mapsto d_{U}(u_t, {\mathrm{u}})$ for all $\mathrm{u} \in \bm V'$. By leveraging the above definition of $d_{U}$ and the concept of weak* measurability, we observe that each $t \mapsto d_{U}(u_t, {\mathrm{u}})$ is a countable sum of measurable maps, and therefore, it is also measurable.
\end{proof}
\begin{example}[Lipschitz vector fields as controls]\label{ex1}
The right-hand side of a Carath\'{e}odory ODE may serve as a control input:
\begin{equation}
        \dot x = f_t(x) \doteq u_t(x), \quad x(0)={\rm x}_0 \in \R^n.\label{ODE-ex1}
\end{equation}
To fit our general construction, let $\bm V$ be the Arens-Eells (Lipschitz free) space $\bm{\mathcal F}(\R^n;\R^m)$ over $\R^n$ (see Remark~\ref{AES} in Appendix~\ref{app:proof-ex12}), and $U$ be a closed ball in the Banach space $\bm V' = \bm{Lip}_0(\R^n, \R^m)$, composed by Lipschitz functions $\mathrm u\colon \R^n \to \R^m$ with the property $\mathrm u(0)=0$, under the Lipschitz norm:
    \[
    \|\mathrm u\|_{\bm{Lip}_0} = 
    \sup\left\{\displaystyle\frac{|\mathrm u(\mathrm x)-\mathrm u(\mathrm y)|}{|\mathrm{x}-\mathrm{y}|}\colon \mathrm x \neq \mathrm y\right\}.
    \]
\end{example}

\begin{example}[$L^{\mathrm{p}}$ vector fields as controls]\label{ex:2}
    Let $\bm V = {L}^{\mathrm{p}}(\R^n; \R^m)$, $1\leq p < \infty$, and $U$ be a closed ball of $\bm V' = {L}^{\rm q}(\R^n; \R^m)$, $\frac{1}{\mathrm p}+\frac{1}{\rm q}=1$. Consider the ODE driven by ``pre-filtered'' feedback controls $u = u_t({\mathrm x})$:
    \begin{gather}
        \dot x = \sum_k f^k_t(x)\,  v^k_t(x)
        \quad x(0)={\rm x}_0 \in \R^n;\label{ODEex}\\ v^k_t \doteq  \eta * u_t^k, \quad u \, {=}\, 
        (u^1, \ldots, u^m)  \in \mathcal U,\notag
    \end{gather}
    where $f^k$ are sublinear Carath\'{e}odory vector fields, and the convolution kernel $\eta$ belongs to $C_c(\R^n)$ and is Lipschitz on its support.
\end{example}
\begin{proposition}\label{prop:ex12}
    The control systems, introduced in Examples \ref{ex1} and \ref{ex:2}, are well-posed, that is, the corresponding ODEs have unique solutions on $I$, for any $u$ from the corresponding set $\mathcal U$, and, moreover, the operators $u \mapsto x^{u}$ are continuous as $\mathcal U \to C_b(\R^n)$. 
\end{proposition}
A proof of this simple assertion is placed in Appendix~\ref{app:proof-ex12}. For the rest of the manuscript, we always assume that $\mathcal U$ agrees with the exposed general construction.

\subsubsection{Problem statement. Immersion}

 Given a control system $(\mathcal X, \Phi, \mathcal U)$, and some fixed initial state \( \mathrm{x}_0 \in \mathcal{X} \), we denote by $x^{u}$ a trajectory $t \mapsto \Phi^u_{0,t}(\mathrm{x}_0)$ of  the control system corresponding to $u \in \mathcal U$. The object of our study is the following dynamic optimization problem:
\[
    \inf\big\{\mathcal{I}[u] \doteq \ell(x_T) : x_t = x_t^{u}, \ u \in \mathcal{U}\big\}, \tag{${P}$}
\]
where \( \ell \colon \mathcal{X} \to \mathbb{R} \) is a given cost function.

In subsequent sections, we rigorously develop the general approach to the variational analysis of problem \((P)\), building on the exposed arguments. This approach relies exclusively on the fundamental properties of the studied model, such as the global metric structure of the space \(\mathcal{X}\) and the concept of external generating family from Section~\ref{ssec:externa}. The lack of a deeper and more substantial structure will require imposing fairly strict assumptions on the initial data. 

The first ingredient is built on the results of Section~\ref{sec:G}, provided by the following basic hypotheses: 
\begin{tcolorbox}
\begin{assumption}
\label{a6}
~

\begin{itemize}
\item There exists a non-empty linear subspace $\bm{\mathcal D} \subset \bigcap\limits_{u \in \mathcal U} C^{1\text{-}\Phi^u}(\mathcal X)$, composed of bounded and Lipschitz functions, and such that:
    \begin{enumerate}[(a)]
    \item \( \ell \in \bm{\mathcal D} \),
    
    \item for each \( u \in \mathcal{U} \), the pair \( (\Phi^u, \bm{\mathcal D}) \) satisfies hypotheses \ref{a4} and \ref{a5}.
    
    \item Moreover, for any \( \mathrm{x} \in \mathcal{X} \), there exists $r= r_{\rm x}$ such that $$
    \sup\left\{ d_{\mathcal X}\left(\mathrm{x}, \Phi^u_{s,t}(\mathrm{x})\right): \, (s,t) \in \Delta, \, u \in \mathcal{U} \right\}< r,$$
    and the operators \( u \mapsto \Phi^u_{s,t} \) are continuous from \( \mathcal{U} \) to \( \mathcal{X} \).
    \end{enumerate}
\end{itemize}
\end{assumption}
\end{tcolorbox}
\begin{remark}
Under assumption \ref{a6}\((\mathrm b) \), a trajectory \( x = x^{u}\) is absolutely continuous on $I$ \cite[Definition~2.1]{ambrosioGradientFlowsMetric2005}. In particular, its metric derivative \cite[Expression (2.2)]{ambrosioGradientFlowsMetric2005}
\[
|\dot x|_t \doteq \lim_{\epsilon \to 0} \frac{1}{\epsilon} d(x_{t+\epsilon}, x_t),
\]
exists for a.a. \( t \in I \), and \( |\dot x| \in L_1(I) \).
\end{remark}

With the original (nonlinear) problem $(P)$ in the metric space $\mathcal X$, we can associate a (state-linear) problem in the dual Banach space $\bm X'\doteq (C_b(\mathcal X))'$:
\[
    \inf\big\{\mathcal J[u] \doteq \langle \bm x_T, \ell \rangle\colon \bm x \doteq  \bm x^{u}, \ u \in \mathcal U\big\},\tag{${\bm{LP}}'$}
\]
where $\ell$ and $\mathcal{U}$ are the same as before, $\vartheta \in \bm X'$ is a fixed element, and $\bm x_t^{u} \doteq \bm \Phi_{0,t}^u\vartheta$ with the maps $\bm \Phi^u$ defined, for any fixed $u \in \mathcal U$, as in \eqref{w-*Phi}. 

The problem $(\bm{LP}')$ is clearly linear in the state variable $\bm{x}$. Moreover, this new problem is equivalent to $(P)$ in the case when $\vartheta = \delta_{{\rm x}_0}$. This $\delta$-version of $(\bm{LP}')$, previously termed the \emph{super-version} of $(P)$, will be denoted by $(\bm{LP}'|P)$.  In general, $(\bm{LP}')$ can be interpreted as a generalization of $(P)$ that allows for ``distributed'' initial conditions. 
  
\subsubsection{(Super-) adjoint trajectory}

The linearity of $(\bm{LP}')$ allows us to adapt several important analytical results, such as duality theory. To leverage this, we introduce the external generators $\mathfrak{L}_t^u$ of the flow $\Phi^u$ on $\bm{\mathcal D}$, as in \eqref{L}, and recall that $\bm x^u$ is a solution of the system \eqref{X} associated to the linear operator $\mathfrak{L}_t = \mathfrak{L}_t^u$, according to Proposition~\ref{propos:ext}. 

Fix a reference control $\bar{u} \in \mathcal{U}$ and define:
\begin{equation}
    \bar{\bm{p}}_t = \bm{\Psi}^{\bar{u}}_{T, t}\ell \doteq \ell \circ \Phi^{\bar{u}}_{t, T}, \quad t \in I.\label{p-rep}
\end{equation}
This map is expected to serve as the adjoint of the corresponding state \( \bar{\bm{x}} = \bm{x}^{\bar{u}} \) and, consequently, as a super-adjoint of the original state \( \bar x = x^{\bar u} \), according to our previous terminology. Establishing this fact relies on the following result:
\begin{lemma}\label{propos:xi}
    Let assumptions \ref{a6} hold, and let $u \in \mathcal{U}$ be arbitrary. Consider a function \( \xi \in C_b(I \times \mathcal{X}) \) satisfying:
    \begin{enumerate}[(a)]
        \item The maps \( t \mapsto \xi_t(\mathrm{x}) \) are Lipschitz on $I$ with a common Lipschitz constant \( \Lip(\xi) \), independent of $\mathrm{x}$. Furthermore, there exists a set \( J \subset I \) of full Lebesgue measure such that, for each \( t \in J \), the derivative \( \partial_t \xi_t(\mathrm{x}) \) exists for all \( \mathrm{x} \in \mathcal{X} \).
        \item The maps \( \mathrm{x} \mapsto \xi_t(\mathrm{x}) \) are Lipschitz on \( \mathcal{X} \) with the same common constant \( \Lip(\xi) \).
        \item The family \( (\xi_t)_{t \in J} \) is uniformly equidifferentiable with respect to \( \Phi^u \).
    \end{enumerate}
    Then, the composition \( t \mapsto (\xi \circ x)_t \), where \( x = x^{u} \), is absolutely continuous. Furthermore, the following Newton-Leibniz formula holds for all \( (s, t) \in \Delta \):
    \begin{align}
        \xi_t(x_t) - \xi_s(x_s) = \int_s^t \left\{ \partial_\tau + \mathfrak{L}^u_\tau \right\} \xi \big|_{x_\tau} \, \mathrm{d} \tau.\label{NL-g}
    \end{align}
\end{lemma}

\begin{proof}
    First, note that due to assumption $(a)$, the value \( |\partial_t \xi_t(x_t)| \) is well-defined at all \( t \in J \), and
    \[
        |\partial_t \xi_t(x_t)| \leq \Lip(\xi),
    \]
    which shows that \( t \mapsto \partial_t \xi_t(x_t) \) belongs to \( L^1(I) \). Similarly, the subset \( \hat{J} \subset J \) of differentiability points of \( x \) has full Lebesgue measure, and by assumption $(b)$, for all \( t \in \hat{J} \),
    \[
        \lim_{h \to 0+} \frac{\left| (\Phi^\star_{t, t+h} \xi_t)(x_t) - \xi_t(x_t) \right|}{h} \leq \Lip(\xi) \, \lim_{h \to 0+} \frac{d(x_{t+h}, x_t)}{h} \doteq \Lip(\xi) \, |\dot{x}|_t,
    \]
    implying that \( t \mapsto (\mathfrak{L}_t \xi_t)(x_t) \) is also in \( L^1(I) \).

    Abbreviate \( \mathfrak{L}_t \doteq \mathfrak{L}_t^u \). Recall that \( \xi \circ x \) is differentiable at any \( t \in \hat{J} \), and represent:
    \begin{equation}\label{proof}
        \frac{\mathrm{d}}{\mathrm{d}t} (\xi \circ x)_t  = \lim_{h \to 0+} \frac{\xi_{t+h}(x_{t+h}) - \xi_{t+h}(x_{t})}{h} + \lim_{h \to 0+} \frac{\xi_{t+h}(x_t) - \xi_t(x_t)}{h}.
    \end{equation}

To compute the first limit, we use assumption $(c)$ and the obvious inequality: 
\begin{align*}
    \frac{\left|\xi_{t+h}(x_{t+h}) - \xi_{t+h}(x_{t}) - h\mathfrak L_{t} \xi_{t+h}(x_{t})\right|}{h} & \doteq \frac{1}{h}\Big|\big(\left[\Phi^\star_{t, t+h}-\id - h\mathfrak L_t\right]\xi_{t+h}\big){(x_t)}\Big|\\
&\leq \sup_{(s, {\mathrm x})\in I \times \mathcal X}\frac{1}{h}\Big|\big(\left[\Phi^\star_{t, t+h}-\id - h\mathfrak L_t\right]\xi_{s}\big){{(\mathrm x)}}\Big|,
\end{align*}
and the uniform equi-differentiability of $(\xi_s)_{s \in \hat J \subset J}$ then implies that
\[
    \lim_{h \to 0+}\frac{\xi_{t+h}(x_{t+h}) - \xi_{t+h}(x_{t})}{h} =  \mathfrak L_{t} \xi_{t}(x_{t}).
\]

Finally, note that the second term of \eqref{proof} coincides with $\partial_t \xi_t(x_t)$, provided by the inclusion $t \in \hat J \subseteq J$.
\end{proof}

A simple corollary of the Lipschitz continuity and the semigroup properties of \(\Phi\) is the following lemma.  
\begin{lemma}\label{lem1}
  Provided by Assumption \((a)\) of Lemma~\ref{propos:xi}, let \(\xi\) be a bounded Lipschitz function \(\mathcal{X} \to \mathbb{F}\). Then:
\begin{enumerate}[(1)]
    \item The map \(s \mapsto (\xi \circ \Phi_{s,t}^{u})(\mathrm{x})\) is Lipschitz on \([0, t]\), uniformly with respect to \(\mathrm{x} \in \mathcal{X}\).
    \item The map \(\mathrm{x} \mapsto (\xi \circ \Phi_{s,t}^{u})(\mathrm{x})\) is Lipschitz, uniformly with respect to \((s, t) \in \Delta\).
\end{enumerate}
\end{lemma}

The next assertion is now straightforward.
\begin{lemma}\label{coro:bmp}
Along with hypotheses \ref{a6}, suppose the following:
\begin{tcolorbox}
\begin{assumption}
\label{A7}~    
\begin{itemize}
\item For any \(u \in \mathcal{U}\), the collection \((\bm{\Psi}_{T,s}^{\bar{u}}\ell)_{0 \leq s \leq T}\) is uniformly equidifferentiable with respect to \(\Phi^u\). 
\item There exists a subset \(\bar{J} \subset I\) of full Lebesgue measure such that the derivative \(\partial_t \ell(\bar{\Phi}_{t,T}(\mathrm{x}))\) is defined on \(\bar{J}\) for all \(\mathrm{x} \in \mathcal{X}\).
\end{itemize}
\end{assumption}
\end{tcolorbox}
Then, the function \(\bar{\bm{p}}\) defined by \eqref{p-rep} satisfies all assumptions of Lemma~\ref{propos:xi}.
\end{lemma}

We culminate this part of the discussion with a key proposition, whose proof is identical to that of Proposition~\ref{lem:psol}:
\begin{proposition}\label{prop:p_ge}
  Under the assumptions of Lemma~\ref{coro:bmp}, \(\bar{\bm{p}}\) is a mild solution to the backward Cauchy problem
\begin{align}\label{p}
    \dot{\bm{p}}_t \doteq \partial_t \bar{\bm{p}}_t = - \mathfrak{L}_t^{\bar u} \, \bm{p}_t, \quad \bm{p}_T = \ell.
\end{align}
\end{proposition}

\subsubsection{(Super-) duality}

Fix two arbitrary controls \( \bar{u}, u \in \mathcal{U} \), and adopt the following notational convention: dependence on \( \bar{u} \) is indicated by a bar, while dependence on \( u \) is implicit. For instance, \( \bm{x} \doteq \bm{x}^u \) and \( \bar{\bm{p}} \doteq \bm{p}^{\bar{u}} \).

Under the assumptions of Lemma~\ref{coro:bmp}, the action map $t \mapsto \langle \bm{x}_t, \bar{\bm{p}}_t \rangle$ is absolutely continuous and satisfies, for a.e. $t \in I$, the product rule \eqref{NL-g}. Together with the representation \eqref{p}, this yields
\begin{align}
\frac{\d}{\d{t}} \langle \bm{x}_t, \bar{\bm{p}}_t \rangle = \big\langle \bm{x}_t, \{\partial_t + \mathfrak{L}_t^u\} \, \bar{\bm{p}}_t \big\rangle = \big\langle \bm{x}_t, \left\{\mathfrak{L}_t^u - \mathfrak{L}_t^{\bar u}\right\} \bar{\bm{p}}_t \big\rangle.
\label{dotcompo}
\end{align}
Setting $u = \bar{u}$ in this expression gives
\[
\frac{d}{dt} \langle \bm{x}_t, \bm{p}_t \rangle \equiv 0,
\]
which shows that the map $t \mapsto \langle \bm{x}_t, \bm{p}_t \rangle$ is constant. In particular, we have the equalities:
\[
\mathcal{J}[u] \doteq \langle \bm{x}_T, \ell \rangle \doteq \langle \bm{x}_T, \bm{p}_T \rangle = \langle \bm{x}_0, \bm{p}_0 \rangle.
\]

This allows us to reformulate $(\bm{LP}')$ in terms of the co-trajectory $\bm{p}_t$, excluding any reference to $\bm{x}_t$, as
\[
\inf\left\{\langle \vartheta, \bm{p}_0 \rangle \colon \bm{p} = \bm{p}^u, \ u \in \mathcal{U} \right\}. \tag{${\bm{LP}}$}
\]
The problem $(\bm{LP})$ is said to be dual to $(\bm{LP}')$. Clearly, both problems have the same minimizing sequences of controls, i.e., are equivalent. 

\subsubsection{Exact increment formulas}\label{sec:imp}

Other important consequences of \eqref{dotcompo} are two representation formulas for the increment of the objective functional $\mathcal J$ of \((\bm{LP}')\) on the pair $(\bar u, u)$, which can be written in terms of the Hamilton-Pontryagin functional 
\[
\bm H_t^u(\vartheta,\phi) \doteq \langle \vartheta, \mathfrak L_t^u\, \phi\rangle
\] as
\begin{align}
    \mathcal J[u] - \mathcal J[\bar u]  = & \  
    \int_I \big(\bm H_t^u(\bm x_t,  \bar{\bm p}_t) - \bm H_t^{\bar u}(\bm x_t,  \bar{\bm p}_t)\big) \d t,\label{eif} \quad  \text{and}\\
    \mathcal J[u] - \mathcal J[\bar u]  = & \ 
    \int_I \left(\bm H_t^{\bar u}(\bar{\bm x}_t,  {\bm p}_t) - \bm H_t^u(\bar{\bm x}_t,  {\bm p}_t)\right) \d t.\label{eif*}
\end{align}
The first (primal) formula is deduced by the duality argument, $\langle \bar{\bm x}_T, \bar{\bm p}_T\rangle = \langle \bar{\bm x}_0, \bar{\bm p}_0\rangle$, in a way similar to Section~\ref{subsec:ee_simple}, 
and the second (dual) formula \eqref{eif*} follows from  \eqref{eif}  by renaming $\bar u \leftrightarrow u$. 

All the said naturally applies to the linear super-version $(\bm{LP}'|P)$ of the nonlinear problem $(P)$. In this case, one can recast the representations \eqref{eif} and \eqref{eif*} in terms of the original states $x=x^{u}$ by specifying $\bm x_t = \delta_{x(t)}$.
For example, \eqref{eif} rewrites:
\begin{align}
    \mathcal I[u] - \mathcal I[\bar u] =  
    \int_I \left(\bar H_t^u(x_t) - \bar H_t^{\bar u}(x_t)\right) \d t,\label{intr-if}
\end{align}
where we incorporate the following notation:
\[
    \bar H_t^u(\mathrm{x}) \doteq  \bm H^u(\delta_{\rm x},  \bar{\bm p}_t)= (\mathfrak L_t^u\,\bar{\bm p}_t)(\rm x).
\]

Note that \eqref{intr-if} boils down to \eqref{incr} and
\eqref{mf_incr} in the corresponding contexts.

\subsubsection{1- and $\infty$-order variational analysis }\label{ssec:infinite}

In the remaining part of the paper, we concentrate on the systems featuring linear dependence on the control variable. Specifically, we incorporate an extra structural assumption:
\begin{tcolorbox}
\begin{assumption}
\label{a8}
~
\begin{itemize}
 
\item For a.a. \(t \in I\), the following equality holds:
\[
\mathfrak{L}_t^u = \widetilde{\mathfrak{L}}_t(u_t),
\]  
where the operators \(\widetilde{\mathfrak{L}}_t(\mathrm{u})\) are defined analogously to \eqref{L}:  
\[
(\widetilde{\mathfrak{L}}_t(\mathrm{u}) \phi)({\mathrm{x}}) \doteq \lim_{h \to 0^+} \frac{\phi\left(\Phi^{u \equiv \mathrm{u}}_{t,t+h}({\mathrm{x}})\right) - \phi({\mathrm{x}})}{h},
\]  
with \(\Phi^{u \equiv \mathrm{u}}_{s,s+h}\) denoting the state evolution from \(t\) to \(t+h\) under a constant control \(\mathrm{u} \in U\). 
   \item Moreover, the function \(\mathrm{u} \mapsto \widetilde{\mathfrak{L}}_t(\mathrm{u})\) is linear.
\end{itemize}
\end{assumption}

\end{tcolorbox}

To streamline the notations, we redefine:  
\[
  \mathfrak{L}_t(\mathrm{u}) \doteq \widetilde{\mathfrak{L}}_t(\mathrm{u}), \qquad \bm H_t(\vartheta,\phi, \mathrm  u) \doteq \langle \vartheta, \mathfrak L_t(\mathrm{u})\, \phi\rangle,  \quad \mbox{and}
\]
\[
  \bar H_t(\mathrm{x}, {\mathrm  u}) \doteq \bm H_t(\delta_{\rm x},\phi, \mathrm  u)\doteq (\mathfrak L_t(\mathrm u)\,\bar{\bm p}_t)(\rm x).
\]

Let $u$ in \eqref{eif} be in the form \eqref{eq:perturb}. The same steps as in Section~\ref{sec:odes} result in the first variation formula
\begin{align}
    \frac{\d}{\d \varepsilon}\big|_{\varepsilon =0}\mathcal J[u^\epsilon] =  
    \int_I \bm H_t(\bar{{\bm x}}_t,  \bar{\bm p}_t, u_t-\bar u_t) \d t,\label{fvf}
\end{align}
implying the classical-form result: 
\begin{theorem}[PMP for problem $(\bm{LP}')$]\label{PMP1}
    Assume that hypotheses \ref{a6}--\ref{a8} hold, and let $\bar u$ be optimal for $(\bm{LP}')$. Then, the following relation holds for a.a. $t\in I$: 
\begin{equation}\label{PMP}
  \min_{\mathrm u \in U}\bm H_t(\bar{{\bm x}}_t,  \bar{\bm p}_t, \bar u_t-\mathrm u) = 0.
\end{equation}
\end{theorem}
By applying this assertion to the delta-problem $(\bm{LP}'|P)$, we immediately have the non-standard formulation of the PMP   
for the original problem $(P)$, involving the super-adjoint state instead of the usual co-trajectory.
 \begin{corollary}[Super-form PMP for problem $(P)$]\label{PMP2} 
Under the same presumption, the following relation holds for a.a. $t \in I$: 
\begin{align}
    \min_{\mathrm u \in U}\bar H_t(\bar x_t, \bar u_t -{\mathrm u})=0.
    \label{pmp_sf}
\end{align}
\end{corollary}
It is important to note that the ``classical'' form of this result is, in general, not applicable to the problem \((P)\).

Now, denote by $\bar{\mathcal U}_{ext}$ the set of controls $u\in \mathcal U$ satisfying, for a.a. $t \in I$, the pointwise minimum condition:
\begin{align}
     \bm H_t(\bm x_t,  \bar{\bm p}_t, u_t)= \min_{\mathrm{u}\in U} \bm H_t(\bm x_t,  \bar{\bm p}_t,\mathrm u).\label{ucomp}
\end{align}
Recall that \eqref{ucomp} is a kind of operator equation on $\mathcal U$, provided by the feedback dependence $\bm x = \bm x^u$. 
\begin{theorem}\label{thm:comp}  
The operator equation \eqref{ucomp} has at least one solution $u \in {\mathcal U}$ for any $\bar u \in \mathcal U$, i.e., \(\bar{\mathcal U}_{ext}\neq \emptyset\).
\end{theorem}
\begin{proof}
Recall that $(\mathcal U =\mathcal U(U), d_U)$ is a compact metric space inheriting the convexity of $U$. 

\textbf{1.} Define the functional $\mathfrak N\colon {\mathcal U}\times {\mathcal U}\to \R$ and the set-valued mapping $\mathfrak M \colon {\mathcal U} \rightsquigarrow {\mathcal U}$ by 
\[
   \mathfrak N[u,v] \doteq \displaystyle\int_{I}\bm H_t(\bm x_t^{u},  \bar{\bm p}_t, v_t)\,dt \ \text{ and } \ \mathfrak M[u] \doteq  \left\{v\in {\mathcal U}\colon \mathfrak N[u, v]=\inf_{w\in \mathcal U} \mathfrak N[u, w]\right\}.
\]
Under the made assumptions, the mapping $\mathfrak N$ is continuous in the product topology and linear in the second argument. It follows that $\mathfrak M$ is upper semicontinuous \cite[Theorem~6, p.~53]{aubin1984differential}, and its values are convex subsets of $\mathcal U$. By applying Kakutani's theorem \cite[Corollary~1, p.~85]{aubin1984differential}, we conclude that there exists $\check u \in \mathcal U$ with the property: $\check u\in \mathfrak M[\check u]$. 

\textbf{2.} Denote \[\eta_t(\mathrm{u}) \doteq \bm H_t(x_t[\check u], \bar{\bm p}_t, \mathrm{u}),\qquad \alpha_t \doteq \displaystyle\min_{{\rm u} \in U}\eta_t(\mathrm u),
\]
and provide the obvious estimate
\begin{align*}
    \int_{I} \bm H_t(x_t[\check u], \bar{\bm p}_t, \check {u})\,dt 
    \doteq 
    \inf_{w\in \mathcal U}\int_I \eta_t(w_t) \d t 
    \geq \int_{I}\min_{\mathrm{u} \in U} \bm H_t(x_t[\check u], \bar{\bm p}_t, \mathrm{u}) \doteq \int_{I} \alpha_t \d t.\label{proof-aux}
 \end{align*}
Noticing that the function \( \eta \colon I \times U \to \R\) is Carath\'{e}odory due to the definition of $\bm H$ and assumptions \ref{a8}, and \( \alpha_t\in  \eta_t(U) \) for a.e. \( t \in I \), it follows from Filippov's lemma~\cite[Theorem 8.2.10]{Aubin2009} that there exists a function $\check w \in \mathcal U$ such that \(\alpha = \eta \circ \check w\). Therefore, the above estimate holds with equality, and $\check u$ is a desired solution of \eqref{ucomp}.
\end{proof}
Appendix~\ref{ssec:fbm} presents an approach for approximating the solution of \eqref{ucomp} using a simple sample-and-hold algorithm.

Provided by the previous assertion, trivial arguments from Section~\ref{ssec:cfbl} lead to the following 
feedback NOC (FNOC).
\begin{theorem}[FNOC for problem $(\bm{LP}')$]\label{thm}
    Under the premise of Theorem~\ref{PMP1}, the relations
\begin{align*}
    \bm H_t(\bm x_t,  \bar{\bm p}_t, \bar u_t) = \bm H_t(\bm x_t,  \bar{\bm p}_t, u_t)\doteq \min_{\mathrm u \in U}\bm H_t(\bm x_t^{u},  \bar{\bm p}_t, \mathrm u)
\end{align*}
    hold for any $u \in \bar{\mathcal U}_{ext}$ and a.e. $t \in I$. Moreover, $\mathcal J[u] = \mathcal J[\bar u]$.
\end{theorem}

As before, all obtained results are directly applied to the problem $(\bm{LP}'|P)$ and, upon an appropriate translation, can be extended to the nonlinear context. Applying this strategy leads to the following optimality principle in the problem $(P)$.
\begin{corollary}[FNOC for problem $(P)$]
Assuming the conditions \ref{a6}--\ref{a8}, the optimality of $\bar{u}$ for $(P)$ implies that
\begin{align*}
    \bar{H}_t(x_t, \bar{u}_t) = \bar{H}_t(x_t, u_t)
\end{align*}
holds at a.e. $t \in I$, for all $u \in \mathcal{U}$ satisfying the operator equation:
\begin{equation}
    \bar{H}_t(x_t, u_t) = \min_{\mathrm{u} \in U} \bar{H}_t(x_t, \mathrm{u}); \quad x = x^u. \label{u-comp}
\end{equation}
\end{corollary}

Although these conditions are not yet sufficient for global optimality, they are demonstrated to be strictly stronger than the PMP even in bi-linear problems (see examples in \cite{CSP-neur, pogodaevExactFormulaeIncrement2024}).

\begin{remark}[Geometric interpretation of the feedback NOCs]

Just as PMP  in control-linear problems is naturally connected to the class of weak variations $u^\epsilon$ of the tested  control $\bar u$, the FNOCs are related to the following class of ``super-strong'' variations:
\begin{align}\label{varu}
    u \triangleright_s \bar u \doteq \left\{
    \begin{array}{cc}
    u_t,  & t \in [0,s),\\
    \bar u_t, & \ t \in [s,T],\end{array} 
    \right. \quad u \in \mathcal U.
\end{align}
Indeed, plugging \eqref{varu} into \eqref{eif} yields:
\[
    \mathcal J[u \triangleright_s \bar u] - \mathcal J[\bar u] =  
    \int_0^s \left[\bm H_t(\bm x_t,  \bar{\bm p}_t, u_t) - \bm H_t(\bm x_t,  \bar{\bm p}_t, \bar  u_t)\right] \d t = \langle \bm x_s, \bar{\bm p}_s\rangle - \langle \vartheta, \bar{\bm p}_0\rangle,
\]
showing that the integrand of \eqref{eif} is exactly the sensitivity $\frac{\d}{\d s} \mathcal J[u \triangleright_s \bar u]$ of the cost functional to the control variation \eqref{varu}. 

This observation gives a natural geometric interpretation of Theorem~\ref{thm}: set
\[
\gamma_s \doteq \bm x_T^{u \, \triangleright_s \bar u},
\] 
and note that the map \(s \mapsto \gamma_s\) defines a continuous curve on the \emph{reachable set} \[\mathcal R_T(\vartheta) \doteq \{\bm x_T^{u}\colon u \in {\mathcal U} \}
\] of the control system  at the final time moment $T$ (see Fig.~\ref{fig:attainable}); by construction, this curves connects the points $\bar{\bm x}_T \doteq \gamma_0$ and $\bm x_T \doteq \gamma_T$, and \[\mathcal J[u \triangleright_s \bar u] = \left\langle \bm \gamma_s, \ell\right\rangle.\] 

In order to check (or discredit) the optimality of $\bar u$, we are to let the map $s \mapsto \left\langle \gamma_s, \ell\right\rangle$ monotonically decrease at the possibly fastest rate. If this map is demonstrated to be absolutely continuous, this means to minimize the quantity:
\[
    \frac{\d}{\d s}\left\langle \gamma_s, \ell\right\rangle \doteq \bm H_s(\bm x_s,  \bar{\bm p}_s, u_s) - \bm H_s(\bm x_s,  \bar{\bm p}_s, \bar  u_s) 
\]
w.r.t. $u_s$ at a.e. $s$. This is, actually, what our feedback optimality principle does. 
\end{remark}

\begin{figure}
  \centering
  \includegraphics[width=0.55\textwidth]{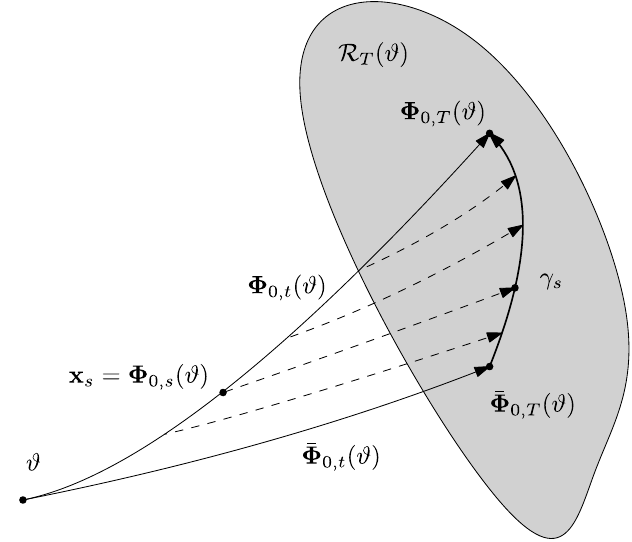}
  \caption{The curve \( \gamma_s \) on the reachable set \( \mathcal{R}_T(\vartheta) \) obtained by switching from \( u \) to \( \bar{u} \) at the time moment \( s \).}
  \label{fig:attainable}
\end{figure}

\section{Conclusions}\label{sec:concl}

We conclude the paper with a summary of the key findings and potential future directions.

\subsection{R\'{e}sum\'{e}. Advantages and shortcomings}

This paper introduces a new phenomenon, which we term ``super-duality'', and highlights its fundamental role in variational analysis, a role that has not been well understood previously. This discovery features elegant connections with classical results and leads to novel optimality principles, reinforcing the conventional concepts of local extremum in dynamic optimization. 


The presented approach is notable for its conceptual simplicity. 
It is particularly surprising that such a straightforward path to classical results has not been explored earlier. Equally unexpected is the fact that, during the ``classical period'' of control theory, the feedback optimality principle~--- although implicitly employed in computational methods \cite{srochko1982computational}~--- was not rigorously formalized. Even in the literature addressing feedback necessary conditions \cite{Dykhta2014}, our \emph{``exact''} version of this principle appears to be absent.

Like any sufficiently general approach, ours has a natural drawback: the regularity requirements we impose, such as uniform equidifferentiability, are significantly stricter than those typically accepted.

\subsection{Numerical algorithms}

A practical outcome of our feedback optimality principles, which lies beyond the scope of the current discussion, is the development of numerical methods for monotone descent. These methods involve the iterative application of feedback controls satisfying condition \eqref{u-comp}, supplemented, if necessary, by a sampling-and-hold algorithm to synthesize the corresponding states.

As evidenced by our computational experience \cite{chertovskihOptimalControlDistributed2023,
chertovskihOptimalControlDiffusion2024}, these algorithms significantly outperform classical indirect ``gradient'' methods based on PMP in state-linear problems involving classical, distributed, and stochastic equations. This advantage arises from the elimination of internal linear search procedures
, leading to a radical economy in recalculating the solutions to the state and adjoint systems.

Extending these numerical experiments to nonlinear settings remains a critical challenge for future research.

\subsection{Generalizations}

While our current focus has been on deterministic systems, the developed approach is applicable to the stochastic control framework, as preliminarily explored in \cite{chertovskihOptimalControlDiffusion2024}. Future work could build on this foundation, particularly by extending the methods to McKean-Vlasov control problems involving non-local Fokker-Planck-Kolmogorov equations.

Another promising avenue would be the possible extension of the results of Section~\ref{sec:nonloc_cont} to nonlocal balance equations on the space of measures, such as those described in \cite{pogodaevNonlocalBalanceEquations2022}. Here, one might begin with semi-linear models of the form \cite{Averboukh2024}, which are reducible to the problem of Section~\ref{sec:nonloc_cont} under an appropriate extension of the state space and a suitable characteristic representation. 

\appendix

\section{Elements of manifold geometry}
\label{app:geom}

  Let \( F \colon M \to N \) be a \( C^{1} \) map between smooth manifolds.
  Its \emph{differential} at a point \( x\in M \) is the linear map \( F_{*,x}\colon T_xM \to T_{F(x)}N \) defined as follows: if \( \gamma \) is a curve on \( M \) with the starting point \( x \) and the initial velocity \( \gamma'(0)=u \), then \( F\circ \gamma \) is a curve on \( N \) with the initial velocity \( (F\circ \gamma)'(0) = F_{*,x}u \).

Any \( C^1 \) map \( F \colon M \to N \) induces a \emph{pullback} operation for 'covariant tensors', such as functions and \( 1 \)-forms, and a \emph{pushforward} operation for 'contravariant tensors', such as vector fields:
  \begin{enumerate}[1)]
    \item A \emph{pullback} of a function \( g\colon N \to \mathbb{R}\) is the function \( F^{*}g\colon M\to \mathbb{R} \) defined by \( F^{*}g\doteq g\circ F \).
    \item A \emph{pushforward} of a vector field \( v \) on \( M \) is the vector field \( F_{*}v \) on \( N \) defined by
    \[
    (F_{*}v)_{F(x)} \doteq F_{*,x}(v_{x}).
    \]

    \item A \emph{pullback} of a \( 1 \)-form \( \omega \) on \( N \) is the \( 1 \)-form \( F^{*}\omega \) on \( M \) defined by
  \[
    (F^{*}\omega)_x(v) \doteq \omega_{F(x)}(F_{*,x}v), \quad v\in T_xM.
  \]
  \end{enumerate}

  We need one more operation \( \mathrm{d} \), called \emph{external derivative}, that maps functions to \( 1 \)-forms.
  If we associate vector fields, with differential operators, then this operation can be defined as \( \mathrm{d}g(v) \doteq v(g)  \), where \( g \) is a \( C^1 \) function, \( v \) is a vector field, and \( v(g) \) is the action of \( v \) on \( g \).

  We will use in the future an important property of the exterior derivative: it is commutative with pullbacks.
  More precisely, we have the following lemma.
  \begin{lemma}\label{lem:commute}
  If \( F\colon M\to N \) and \( g\colon N\to \mathbb{R} \) are \( C^{1} \) maps, then
  \[
    F^\star(\mathrm{d}g) = \mathrm{d}(F^\star g).
  \]
\end{lemma}

\begin{remark}
  This exposition of geometric structures is far from being complete.
  We refer for further details to any book on smooth manifolds, e.g.,~\cite{novikovModernGeometricStructures2006}.
  Remark that in the case we are interested in, i.e., \( M=N= \mathbb{R}^n \), the above notion of differential \( F_{*,x} \) coincides with the usual notion of derivative \( DF(x) \) at a point, while the exterior derivative simply maps \( g\colon \mathbb{R}^n\to \mathbb{R} \) into the \( 1 \)-form \( \mathrm{d}g = \frac{\partial g}{\partial{x^1}}\mathrm{d} x^1 +\cdots + \frac{\partial g}{\partial{x^n}}\mathrm{d} x^n\).
\end{remark}

\section{Weakly* measurable vector-functions}\label{app:meas-w*}

Let $\bm V'$ be the dual of a Banach space $\bm V$. The following definition can be found in \cite[Section~II.1, Definition 1]{diestel1977vector}: 
\emph{A function $u\colon I \to \bm V'$, is said to be weakly* measurable if, for any $\mathrm v \in \bm V$, the action $t \mapsto \langle u_t, \mathrm  v\rangle$, $I \to \mathbb F$, is measurable w.r.t. the Borel 
sigma-algebras $\mathcal B_{\R}$ and $\mathcal B_{\mathbb F}$. } 

We say that two weakly* measurable functions $u$ and $u'$ are equivalent and write $u \sim u'$ if and only if, for all $\mathrm v \in \bm V$, the equality
\[
\langle u_t, \mathrm v\rangle = \langle u_t', \mathrm v\rangle
\]
holds for all $t \in \Omega$, where $\Omega = \Omega(\mathrm v)$ is a subset of $I$ of full Lebesgue measure.\footnote{It is easily checked that $\sim$ is indeed an equivalence relation.}

Denote by ${L}^\infty_{w^*}(I; \bm V')$ the factor of weakly* measurable functions $I \to \bm V'$ modulo the relation $\sim$, such that there exists $c \geq 0$ with the property that, for any $\mathrm v \in \bm V$, the estimate
\[
\left|\langle u_t, \mathrm v\rangle\right| \leq c\|{\rm v}\|_{\bm V}
\]
holds for all $t\in \Omega(\mathrm v)$. Evidently, ${L}^\infty_{w^*}$ is a linear space, and the infimum among all constants $c$ from the above estimate serves as a norm $\|\cdot\|_{{L}^\infty_{w^*}}$ on this space. In particular, if $\bm V$ is separable and $u \in {L}_\infty^{w^*}(I; \bm V')$, the map \[\xi_u \colon t \mapsto \|u_t\|_{\bm V'}\] is in ${L}^\infty(I; \R_+)$, and
\[\|u\|_{{L}^\infty_{w^*}(I; \bm V')} = \|\xi_u\|_{L^\infty(I; \R_+)},\] see e.g. \cite[Remark~10.1.15]{papageorgiou2009handbook}.
 
The following result can be found in \cite[p. 95]{tulcea1969} and \cite[Theorem~10.1.16]{papageorgiou2009handbook}:
\begin{proposition}
The duality relation \eqref{d-w} holds for any Banach space $\bm V$, provided by the pairing \eqref{wp*}.
\end{proposition}

\section{Young-Warga-Gamkrelidze generalized controls}\label{appe:GenU}

The notion of generalized control from Section~\ref{sec:odes} dates back to the works of L. Young \cite{young1969}, J.~Warga \cite{warga1972optimal} and R.V.~Gamkrelidze \cite{Gamkrelidze1978}. 

A \emph{generalized control} \( u \) is a Borel measure on \( I \times U \) with a property: \( \pi^1_\sharp u = \mathfrak{L}^1 \), where \( \pi^1 \colon I \times U \to I \) denotes the projection onto the time interval.  
Any generalized control is uniquely represented by its \emph{disintegration}, i.e., a weakly* measurable family \( (u_t)_{t \in I} \) of Borel probability measures \( u_t \in \mathcal{P}(U) \) satisfying the identity 
\[
\int_{I \times U} \phi(t,\mathrm{u}) \, \mathrm{d} u(t,\mathrm{u}) = \int_I \int_U \phi(t,\mathrm{u}) \, \mathrm{d} u_t(\mathrm{u}) \, \mathrm{d}t,
\]
for any Borel measurable function \( \phi \colon I \times U \to \mathbb{R} \).

The use of generalized controls is motivated by the following notable properties: 
\begin{enumerate}[i)]
    \item The set \( {\mathcal{U}} \) is compact in the topology of weak convergence of probability measures.

    \item The collection \( \{u[v] \colon v \in L^\infty(I;U)\} \) is dense in \( {\mathcal{U}} \).
\end{enumerate}

\section{Feedback controls. Sample-and-hold solutions}\label{ssec:fbm}

The functions $u\in \mathcal U$ are so-called \emph{open-loop} controls (or program strategies): the corresponding control mechanism consists in pre-setting an instruction (a program) for the system evolution over the whole planned period. We now recall a different concept, which assumes the option to observe current system's states and intervene in the dynamics ``on-flight''. 
\begin{definition}
By a \emph{feedback control} of system $\Phi$ we mean a family $\mathfrak{u} \doteq (u^{\rm x})_{\mathrm x \in \mathcal X}$ of admissible controls $u^{\rm x} \in \mathcal U$, parameterized by the points of the state space $\mathcal X$.
\end{definition}
Given a feedback control $\mathfrak u$, we can design the corresponding trajectory $x[\mathfrak{u}]$ via the following ``sample-and-hold'' algorithm. 

Consider a sequence $(\pi^N)_{N \in \mathbb N}$ of partitions $\pi^N=\{0=t_0^N < t_1^N<\ldots t_N^N =T\}$ of the time segment $I$ such that $\pi^{N} \subset \pi^{N+1}$, and $|\pi^N| \doteq \max\limits_{k=0}^N\left(t_{k+1}^N-t_k^N\right) \to 0$ as $N \to \infty$. For each $\pi^N$, we define a polygonal arc $x^N\doteq x[\mathfrak{u}; \pi^N]\colon I \to {\mathcal X}$ recurrently in $0\leq k \leq N-1$:
\[
    x^N_t := \Phi_{t_k^N, t}\left[u^{\mathrm x^N_k}\right]({\mathrm x}^N_k), \quad  t \in [t_k^N, t_{k+1}^N]; \quad {\mathrm x}_k^N:= x^N({t_k}^N),
\]
initialized by the known value ${\mathrm x}_0$ at $t_0^N=0$. Simultaneously, we design an open-loop control $u^N \doteq u[\mathfrak u; \pi^N] \in \mathcal U$:
\[
    u^N_t := u^{N,k}_t \doteq u^{\mathrm x_k^N}_t, \quad  t \in [t_k^N, t_{k+1}^N],
\]
and notice that $x^N = x^{u^N}$.
\begin{remark}

In the literature, the term ``feedback control" typically refers to a loop of the form \( u = u_t(\mathrm{x}) \). While our notion deviates from this classical interpretation, we retain the same terminology for convenience.

Similarly, our definition of a sample-and-hold solution is non-standard. While resembling the well-known Krasovskii-Subbotin framework \cite{krasovskii2011game}, it operates with ``piecewise open-loop'' approximations rather than piecewise constant ones.

\end{remark}

Recall that $\mathcal U$ is a compactly metrizable space, and, for metric spaces, the notions of topological and sequential compactness are equivalent. Thus, the sequence $(u^N)$ contains a subsequence converging to some element $u\doteq u[\mathfrak u] \in \mathcal U$. Passing to this subsequence (that we not relabel) and leveraging the continuity of $u \mapsto x^{u}$, we 
conclude that the corresponding trajectories $x^N$ converge to $x[\mathfrak u]\doteq x^{u[\mathfrak u]} \doteq x^{u}$ in the pointwise sense.
\begin{remark}
    The process $(x[\mathfrak{u}], u[\mathfrak{u}])$ is a partial limit of $(x^N, u^N)$, i.e., the map $\mathfrak{u} \mapsto (x[\mathfrak{u}], u[\mathfrak{u}])$ is, generically, multi-valued. 
\end{remark}

We now suggest a constructive approach to resolving the operator equation \eqref{u-comp} based on the following lemma. 
\begin{lemma}\label{lem:sampl}
Let a function $h \colon I \times \mathcal X \times U \to \R$, $h = h_t({\rm x}, {\mathrm u})$, be non-negative, measurable in $t$, continuous in $({\rm x, u})$,  and locally Lipschitz in $\rm x$ uniformly w.r.t. $(t, \mathrm u) \in I \times U$. Furthermore, assume that, for any converging sequence $(u^k) \subset \mathcal U$ with a limit $u \in \mathcal U$, it holds:
\[
    \lim_{k \to \infty}\int_I \left[h_t(x_t^{u}, u_t) - h_t(x_t^{u}, u^k_t)\right] \d t = 0,
\]
and there exists a family $\mathfrak{u}=(u^{\rm x})$ of controls $u^{\rm x} \in \mathcal U$ such that
\[
 h_t(\mathrm x, u^{\mathrm x}_t) = 0 \mbox{ for all }\mathrm x\in \mathcal S,\mbox{ and a.e. } t \in I.
\]    
Then, the equality 
\[
 h_t(x_t^{u}, u_t) = 0\mbox{ for a.e. }t\in I,
\]
holds for any open-loop control \(u = u[\mathfrak{u}]\) generated by $\mathfrak{u}$.
\end{lemma}
\begin{proof}
     Consider the sequence $(x^N, u^N)$ defined by the sample-and-hold method, using the feedback $\mathfrak u$. Represent:
    \[
      \int_I h_t(x^N_t, u^N_t) \d t = \sum_k \int_{{I_k^N}}\left[h_t(x^N_t, u^{N,k}_t) - h_{t}({\mathrm x}^N_{k}, u^{N,k}_t)\right]\d t, 
    \]
    and recall that      \(h_{t}({\mathrm x}^N_{k}, u^{N,k}_t) = 0\) for a.e. $t \in {I_k^N} \doteq [t_k^N, t_{k+1}^N]$, all $k = 0, \ldots, N$, and any $N \geq 1$, by assumption of the assertion.  

 Each term in the latter sum is estimated as
    \begin{align*}
                \int_{{I_k^N}}\left[h_t(x^N_t, u^{N,k}_t) - h_{t}(\mathrm x^N_{k}, u^{N,k}_t)\right]\d t \leq & \ \Lip_{\mathcal K}(h) \int_{{I_k^N}}d(x^N_t,\mathrm x^N_{k})\notag\\ 
                \leq & \  \Lip_{\mathcal K}(h) \Lip(\Phi) |\pi^N|^2\notag\\ 
                \doteq & \,  O\left(\frac{1}{N^2}\right)>0.
    \end{align*}
    Therefore,
\[
        0\leq \int_I h_t(x^N_t, u^N_t) \d t \leq O\left(\frac{1}{N^2}\right) N \doteq  O\left(\frac{1}{N}\right)>0.
    \]
    Let $x=x^{u}$. As a consequence,
    \begin{align*}
    0 & \leq 
        \int_I h_t(x_t, u_t)\d t\\ 
        & \leq  \left|\int_I \left[h_t(x_t, u_t) - h_t(x^N_t, u^N_t)\right] \d t\right|+ O\left(\frac{1}{N}\right)\\ 
              & \leq \left|\int_I \left[h_t(x_t, u_t) - h_t(x_t, u^N_t)\right] \d t\right| + \int_I \left|h_t(x_t, u_t^N) - h_t(x^N_t, u^N_t)\right| \d t+ O\left(\frac{1}{N}\right)\\
              & \leq \left|\int_I \left[h_t(x_t, u_t) - h_t(x_t, u^N_t)\right] \d t\right| + \Lip(h)\int_I d\big(x_t, x^N_t\big)\d t+ O\left(\frac{1}{N}\right).
    \end{align*}
    Passing to a partial limit as $u^{N_j} \to u$,
    we conclude that
    \[
      \int_I h_t(x_t, u_t)\d t =  0.
    \]
  To complete the proof, it only remains to leverage the non-negativity of the integrand.
  
\end{proof}
\begin{remark}
    A natural example of the map $h$ with the regularity, required by the above assertion, is
$$h_t(\mathrm x, \mathrm u)= \left\langle{\rm u}, \mathfrak a_t(\mathrm x)\right\rangle _{(\bm V', \bm V)} + \bm b_t(\mathrm x),
$$
where $\mathfrak a\colon I \times \mathcal X \to \bm V$, $\mathfrak a={\mathfrak a}_t(\rm x)$, and ${\bm b}\colon I \times \mathcal X \to \R$, $\bm b={\bm b}_t(\mathrm x)$ --- ensuring the non-negativity of $h$ --- are measurable in $t$, and locally Lipschitz in $\rm x$ uniformly w.r.t. $t \in I$; moreover, the function $t \mapsto \sup_{\rm x \in \mathcal X}\|\mathfrak{a}_t (\mathrm x)\|_{\bm{V}}$ is in $L^1(I)$. Note that these assumptions imply, in particular, the convergence
\[
    \int_I \left[h_t(x_t^{u}, u_t) - h_t(x_t^{u}, u^k_t)\right] \d t \doteq \langle\!\langle u - u^{k},   \mathfrak{a}\circ x\rangle\!\rangle \to 0\mbox{ as }u^k \to u\mbox{ within }\mathcal U.
\]

\end{remark}

\begin{proposition}
    Let $\bar H$ satisfy all hypotheses of Lemma~\ref{lem:sampl}, except for non-negativity, and $\mathfrak{u}=(u^{\rm x})$ be defined as a measurable family of solutions $u^{\rm x} \in \mathcal U$ to~
\[
 \bar H_t (\mathrm x, u^{\mathrm x}_t) = \min_{{\mathrm u} \in U} \bar H_t(\mathrm{ x}, {\mathrm u}).
\]    
Then, any open-loop control \(u = u[\mathfrak{u}]\) generated by $\mathfrak{u}$ validates \eqref{u-comp}, i.e., $u\in {\bar{\mathcal U}}_\text{ext}$.   
\end{proposition}
\begin{proof}
    The result follows from the previous assertion, provided by the choice
    \(
        \displaystyle h_t(\mathrm x, \mathrm{u}) \doteq  \bar H_t ({\rm x, u}) - \min_{{\mathrm u} \in U} \bar H_t(\mathrm{ x}, {\mathrm u}).
    \)
\end{proof}
A similar result is obtained in \cite{pogodaevExactFormulaeIncrement2024} for the classical setting and Krasovskii-Subboting formalism.

\section{Proofs related to Section~\ref{sec:nonloc_cont}}\label{sec:Wflowdiff} 

\subsection{Proof of Proposition~\ref{prop:Wflowdiff}}

The proof consists of three steps, each addressed in Lemmas~\ref{lem:O2vf}--\ref{lem:s2}. In the first step, we establish a result analogous to Lemma~\ref{lem:testlip}. Without loss of generality, we assume that \( t_0 = 0 \).

\begin{lemma}
  \label{lem:O2vf}
  Let \( F_t(\mathrm{x},\mu) = F_{t}(\mathrm{x},\mu,u_{t}) \) and \( \Phi \) be the corresponding flow.
  Fix \( \mu\in \mathcal{P}_2 \) and let \( X_t \doteq X_{0,t}^{\mu} \).
  The estimate
  \begin{multline}\label{eq:L1est}
    \Big\|F_{t} \left( X_{t} + \varepsilon v, (X_t + \varepsilon v)_{\sharp}\mu  \right) - F_{t}(X_t,X_{t\sharp}\mu)  - \varepsilon DF_{t}\left(X_t,X_{t\sharp}\mu\right)v\\
    - \varepsilon\int \mathbf{D}F_t\left(X_t,X_{t\sharp}\mu,X_{t}(\mathrm{y})\right)v(\mathrm{y})\d\mu(\mathrm{y})\Big\|_{L^1_\mu} \le 2\left(\Lip(DF_t)+\Lip(\mathbf{D}F_t)\right)\|v\|^2_{\mu}\varepsilon^2
  \end{multline}
  holds for all \( t \in I \), \( v \in {L}^2_{\mu} \), and \( \varepsilon \in \mathbb{R} \).
\end{lemma}
\begin{proof}
  We split the proof into several steps.

  \textbf{1.} Consider the identity:
  \begin{align}
    F_{t} &\left(X_t(\mathrm{x}) + \varepsilon v(\mathrm{x}), (X_t + \varepsilon v)_{\sharp}\mu  \right) - F_{t}\left(X_t(\mathrm{x}),X_{t\sharp}\mu\right)\notag\\
    &= F_{t} \left( X_t(\mathrm{x}) + \varepsilon v(\mathrm{x}), (X_t + \varepsilon v)_{\sharp}\mu  \right) - F_{t} \left( X_t(\mathrm{x}), (X_t + \varepsilon v)_{\sharp}\mu  \right)\notag\\
    &+ F_{t} \left( X_t(\mathrm{x}), (X_t + \varepsilon v)_{\sharp}\mu  \right) - F_{t}(X_t(\mathrm{x}),X_{t\sharp}\mu).
    \label{eq:aux0}
  \end{align}
  By the mean value theorem, the first difference in the right-hand side takes the form:
  \begin{displaymath}
    \varepsilon\int_{0}^{1} DF_{t} \left( X_t(\mathrm{x}) + s\varepsilon v(\mathrm{x}), (X_t + \varepsilon v)_{\sharp}\mu \right)v(\mathrm{x}) \d s,
  \end{displaymath}
 and, due to the definition of \( \frac{\delta F_t}{\delta \mu} \), the second difference writes:
  \begin{align*}
    \int_{0}^{1}
    &\int \frac{\delta F_{t}}{\delta \mu} \left(X_t(\mathrm{x}), \mu_{\varepsilon,\tau}, \mathrm{y} \right) \d \left( (X_t + \varepsilon v)_{\sharp}\mu - X_{t\sharp}\mu \right)(\mathrm{y}) \d \tau\\
    &= \int_{0}^{1}\int \left[\frac{\delta F_{t}}{\delta \mu} \left(X_t(\mathrm{x}), \mu_{\varepsilon,\tau}, X_t(\mathrm{y}) + \varepsilon v(\mathrm{y}) \right) - \frac{\delta F_{t}}{\delta \mu} \left(X_t(\mathrm{x}), \mu_{\varepsilon,\tau}, X_t(\mathrm{y}) \right)\right] \d \mu(\mathrm{y}) \d \tau\\
    &= \varepsilon\int_{0}^{1}\int_{0}^{1}\int \mathbf{D}F_{t} \left(X_t(\mathrm{x}), \mu_{\varepsilon,\tau}, X_t(\mathrm{y}) + s\varepsilon v(\mathrm{y})\right)v(\mathrm{y}) \d \mu(\mathrm{y}) \d s \d \tau,
  \end{align*}
  where \( \mu_{\varepsilon,\tau} = (1-\tau) X_{t\sharp} \mu + \tau (X_t + \varepsilon v)_{\sharp}\mu \).

  \textbf{2.} Since \( DF_t \) and \( \mathbf{D}F_t \) are Lipschitz, it holds:
  \begin{align}
    \left| DF_{t} \left( X_t(\mathrm{x}) + s\varepsilon v(\mathrm{x}), (X_t + \varepsilon v)_{\sharp}\mu \right) v(\mathrm{x}) - DF_{t}(X_t(\mathrm{x}), X_{t\sharp}\mu) v(\mathrm{x}) \right|\notag \\
    \leq \Lip(DF_{t}) \left( s\varepsilon |v(\mathrm{x})| + W_2\left( (X_t + \varepsilon v)_{\sharp}\mu, X_{t\sharp}\mu \right) \right) |v(\mathrm{x})| \notag \\
    \leq \varepsilon \, \Lip(DF_{t}) \left( |v(\mathrm{x})|^2 + \|v\|_{\mu} |v(\mathrm{x})| \right),
    \label{eq:aux1}\\
    \hspace{-12pt}\Big| \int \mathbf{D}F_{t} \left(X_t(\mathrm{x}), \mu_{\varepsilon,\tau}, X_t(\mathrm{y}) + s\varepsilon v(\mathrm{y})\right) v(\mathrm{y}) \d \mu(\mathrm{y}) \notag \\
    - \int \mathbf{D}F_{t}\left(X_t(\mathrm{x}), X_{t\sharp}\mu, X_t(\mathrm{y})\right) v(\mathrm{y}) \d \mu(\mathrm{y}) \Big| \notag \\
    \leq \Lip(\mathbf{D}F_{t}) \left( W_2\left( X_{t\sharp}\mu, \mu_{\varepsilon,\tau} \right) \int |v(\mathrm{y})| \d \mu(\mathrm{y}) + \varepsilon \|v\|_{\mu}^2 \right). \label{eq:aux2}
  \end{align}

  \textbf{3.} Let us estimate the value \( W_2 \left( \mu, \mu_{h,\tau} \right) \).
  To this end, recall that
  \begin{equation}
    \label{eq:convex}
    W_{2}^{2} \left( (1 - \tau) \mu_0 + \tau \mu_1, \nu \right) \leq (1 - \tau) W_{2}^{2} (\mu_0, \nu) + \tau W_{2}^{2} (\mu_1, \nu),
  \end{equation}
  for all \( \mu_0, \mu_1, \nu \in \mathcal{P}_{2}(\mathbb{R}^n) \) and all \( \tau \in [0,1] \).
  This inequality becomes evident if we note that, for any \( \Pi_0 \in \Gamma_o(\mu_0, \nu) \) and \( \Pi_1 \in \Gamma_o(\mu_1, \nu) \), the convex combination \( (1 - \tau)\Pi_0 + \tau \Pi_1 \) is a transport plan between \( (1 - \tau) \mu_0 + \tau \mu_1 \) and \( \nu \).
  In our case,~\eqref{eq:convex} implies that
  \begin{equation}
    \label{eq:aux3}
    W_2 \left( X_{t\sharp}\mu, \mu_{\varepsilon,\tau} \right) \leq \sqrt{\tau} W_2 \left( X_{t\sharp}\mu, (X_t + \varepsilon v)_{\sharp}\mu \right) \leq \sqrt{\tau} \varepsilon \|v\|_{\mu}.
  \end{equation}
  
  \textbf{4.} By combining the inequalities~\eqref{eq:aux0}--\eqref{eq:aux2}, \eqref{eq:aux3} and integrating with respect to \( \mu \), we obtain the desired estimate.
\end{proof}

\begin{remark}\label{rem:problem}
  The estimate~\eqref{eq:L1est} provided by Lemma~\ref{lem:O2vf} is crucial in showing that the flow \( \Phi \) is differentiable in the sense of Definition~\ref{def:diff_mp}. Suppose that we would like to demonstrate that \( \Phi \) is differentiable in the stronger sense, i.e., when the distance \( W_1 \) in Definition~\ref{def:diff_mp} is replaced by \( W_2 \) (see Remark~\ref{rem:stronger}). In this case, if we decide to follow our proof strategy, the \( L^1_{\mu} \) norm in~\eqref{eq:L1est} would need to be replaced by the \( L^2_{\mu} \) norm. However, this is not possible as the following example shows.

  Consider the vector field $F(\mathrm{x})=\mathrm{x}^2$ on $\mathbb R$. In this case,
    \[
      F(X_t(\mathrm{x})+\varepsilon v(\mathrm{x})) - F(X_t(\mathrm{x})) - \varepsilon DF(X_t(\mathrm{x}))v(\mathrm{x}) = \varepsilon^2 v^2(\mathrm{x}).
    \]
    This expression belongs to $L^2_\mu$ only if \( v\in L^4_{\mu} \).
\end{remark}

In the second step, we demonstrate that the linear equation~\eqref{eq:Ww} for \( w \) is well-defined, meaning that it admits a unique solution for any \( v \in L^1_{\mu} \).
\begin{lemma}
  \label{lem:s1}
  Let \( A \colon I \times \mathbb{R}^n \to \mathcal{L}(\mathbb{R}^n;\mathbb{R}^n) \) and \( B\colon I \times \mathbb{R}^n \times \mathbb{R}^n \to \mathcal{L}(\mathbb{R}^n;\mathbb{R}^n) \) be bounded Carath\'eodory functions, and \( v\in L^1_{\mu} \).
  There exists a unique function \( w\colon I \times \mathbb{R}^n\to \mathbb{R}^n \) that satisfies, for all \( \mathrm{x}\in \mathbb{R}^n \), the equation
  \begin{equation}
    \label{eq:Alin}
    \partial_t w_t(\mathrm{x}) = A_t(\mathrm{x}) w_t(\mathrm{x}) + \int B_t(\mathrm{x},\mathrm{y}) w_t(\mathrm{y})\d \mu(\mathrm{y})\quad \text{a.e. on } I
  \end{equation}
  and the initial condition \( w_{0}(\mathrm{x}) = v(\mathrm{x}) \).
  Moreover, \( v \mapsto w_t \) is a linear bounded map from \( {L}^2_{\mu} \) to itself.
\end{lemma}
\begin{proof}
\textbf{1.} Consider the following operator:
\[
\mathcal{F}(w)_t(\mathrm{x}) \doteq v(\mathrm{x}) + \int_{0}^t \left\{A_s(\mathrm{x}) w_s(\mathrm{x}) + \int B_s(\mathrm{x},\mathrm{y}) w_s(\mathrm{y}) \d \mu(\mathrm{y}) \right\} \, ds, \quad \mathrm{x} \in \mathbb{R}^n,
\]
which acts from the Banach space \( C\left(I;L^1_{\mu}\right) \) to itself.

Indeed, if \( w \in C\left(I;L^1_{\mu}\right) \), then
\[
|\mathcal{F}(w)_t(\mathrm{x})| \le |v(\mathrm{x})| + \int_{0}^t \left( \|A_s\|_{\infty} |w_s(\mathrm{x})| + \|B_s\|_{\infty} \int |w_s(\mathrm{y})| \d \mu(\mathrm{y}) \right) \, ds.
\]
After integrating, we obtain:
\[
\int |\mathcal{F}(w)_t(\mathrm{x})| \, \mu(\mathrm{x}) \le \int |v(\mathrm{x})| \d \mu(\mathrm{x}) + \int_{0}^t \left( \|A_s\|_{\infty} + \|B_s\|_{\infty} \right) \int |w_s(\mathrm{y})| \d \mu(\mathrm{y}) \, ds.
\]
Therefore, \( \mathcal{F}(w)_t \in L^1_{\mu} \) whenever \( w \in C\left(I;L^1_{\mu}\right) \).

Next, let us check that \( t \mapsto \mathcal{F}(w)_t \) is continuous as a map from \( I \) to \( L^1_{\mu} \).

Indeed, if \( t_1 < t_2 \), then
\[
\int |\mathcal{F}(w)_{t_2}(\mathrm{x}) - \mathcal{F}(w)_{t_1}(\mathrm{x})| \d \mu(\mathrm{x}) \le \int_{t_1}^{t_2} \left( \|A_s\|_{\infty} + \|B_s\|_{\infty} \right) \|w_s\|_{L^1_{\mu}} \, ds.
\]
The integral on the right-hand side clearly vanishes as \( t_1 \to t_2 \).

\textbf{2.} We claim that \( w \to \mathcal{F}(w) \) is a contraction when \( C(I;L^1_{\mu}) \) is equipped with the norm
\begin{equation}
  \label{eq:lambdanorm}
\|w\|_{\lambda} = \max_{t \in I} \left\{ e^{-\lambda t} \|w_t\|_{L^1_{\mu}} \right\},
\end{equation}
and \( \lambda > 0 \) is sufficiently large.

Indeed, for any pair \( w^1, w^2 \in C(I;L^1_{\mu}) \), we have:
\begin{align*}
&\int \left| \mathcal{F}(w^1)_t(\mathrm{x}) - \mathcal{F}(w^2)_t(\mathrm{x}) \right| \d \mu(\mathrm{x})\\ 
&\quad \le \int_{0}^t \left( \|A_s\|_{\infty} + \|B_s\|_{\infty} \right) \int |w^1_s(\mathrm{x}) - w^2_s(\mathrm{x})| \d \mu(\mathrm{x}) \d s.
\end{align*}
Equivalently,
\[
\left\| \mathcal{F}(w^1)_t - \mathcal{F}(w^2)_t \right\|_{L^1_{\mu}} \le \int_{0}^t \left( \|A_s\|_{\infty} + \|B_s\|_{\infty} \right) e^{\lambda s} e^{-\lambda s} \|w^1_s - w^2_s\|_{L^1_{\mu}} \, ds,
\]
which implies
\[
\left\| \mathcal{F}(w^1)_t - \mathcal{F}(w^2)_t \right\|_{L^1_{\mu}} \le \int_{0}^t \left( \|A_s\|_{\infty} + \|B_s\|_{\infty} \right) e^{\lambda s} \, ds \cdot \|w^1 - w^2\|_{\lambda}.
\]
Using the Cauchy-Schwarz inequality, we then get
\begin{align*}
\int_{0}^t \left( \|A_s\|_{\infty} + \|B_s\|_{\infty} \right) e^{\lambda s} \, ds
&\le \sqrt{\int_{0}^t \left( \|A_s\|_{\infty} + \|B_s\|_{\infty} \right)^2 \, ds} \cdot \sqrt{\int_{0}^t e^{2\lambda s} \, ds} \\
&\le \sqrt{\int_{0}^t \left( \|A_s\|_{\infty} + \|B_s\|_{\infty} \right)^2 \, ds} \cdot \frac{e^{\lambda t}}{\sqrt{2\lambda}}.
\end{align*}
Therefore,
\[
  \left\| \mathcal{F}(w^1) - \mathcal{F}(w^2) \right\|_{\lambda} \le \frac{1}{\sqrt{2\lambda}} \sqrt{\int_{0}^T \left( \|A_s\|_{\infty} + \|B_s\|_{\infty} \right) \, ds} \cdot \|w^1 - w^2\|_{\lambda},
\]
which shows that \( \mathcal{F} \) is a contraction if \( \lambda \) is sufficiently large.

Now, by the Banach fixed-point theorem, \( \mathcal{F} \) has a unique fixed point, which, by construction, is the only solution of~\eqref{eq:Alin} that satisfies the initial condition \( w_0 = u \).

 \textbf{3.} The linearity of the map \( v \mapsto w_t \) is obvious.
 Let us verify that \( v\in {L}^2_{\mu} \) implies that \( w_t\in {L}^2_{\mu} \), for all \( t\in I \).

 Fix an arbitrary \( \mathrm{x}\in \mathbb{R}^n \).
Then, for all \( t\in I \), it holds
\begin{equation}
  \label{eq:Awle}
    |w_t(\mathrm{x})| \le |v(\mathrm{x})| + \int_{0}^t \left(\|A_s\|_{\infty} |w_s(\mathrm{x})| + \|B_{s}\|_{\infty}\int |w_s(\mathrm{y})|\d \mu(\mathrm{y})\right)\d s.
\end{equation}
  Next we integrate with respect to \( \mathrm{x} \):
  \[
   \int |w_t(\mathrm{x})|\d \mu(\mathrm{x}) \le \int|v(\mathrm{x})|\d \mu(\mathrm{x}) + \int_{0}^t \left(\|A_s\|_{\infty}+ \|B_{s}\|_{\infty}\right)\int |w_s(\mathrm{y})|\d \mu(\mathrm{y})\d s,
 \]
 so that Gr\"onwall's inequality yields:
 \[
   \int |w_t(\mathrm{x})|\d \mu(\mathrm{x})\le \exp \left\{\int_{0}^t \left(\|A_s\|_{\infty}+ \|B_{s}\|_{\infty}\right)\d s  \right\}\int|v(\mathrm{x})|\d \mu(\mathrm{x}).
 \]
 We use the previous inequality to estimate the last term form the right in~\eqref{eq:Awle}:
 \[
    |w_t(\mathrm{x})| \le |v(\mathrm{x})| + \int_{0}^t \|A_s\|_{\infty}|w_s(\mathrm{x})|\d s + g(t)\int|v(\mathrm{x})|\d \mu(\mathrm{x}),
  \]
  where
  \[
    g(t) \doteq \int_{0}^t\|B_s\|_{\infty}\,\exp \left\{\int_{0}^s \left(\|A_\tau\|_{\infty}+ \|B_{\tau}\|_{\infty}\right)\d \tau  \right\}\d s.
  \]
  Once again Gr\"onwall's inequality gives
  \begin{align}
    |w_t(\mathrm{x})| \le \left( |v(\mathrm{x})| + g(t) \int|v(\mathrm{x})|\d \mu(\mathrm{x}) \right) e^{\int_{0}^t\|A_s\|_{\infty}\d s}.\label{w-infti}
  \end{align}
  Squaring yields that
  \[
    |w_t(\mathrm{x})|^2 \le C\left( |v(\mathrm{x})|^{2} + |v(\mathrm{x})| \int|v(\mathrm{x})|\d \mu(\mathrm{x}) + \left(\int|v(\mathrm{x})|\d \mu(\mathrm{x})\right)^2\right)
  \]
  for some \( C>0 \) depending only on \( \|A\|_{\infty} \) and \( \|B\| _{\infty}\).
  By integrating, we obtain:
  \begin{align}
    \int|w_t(\mathrm{x})|^2 \d \mu(\mathrm{x})
    &\le C\left( \int|v(\mathrm{x})|^{2}\d\mu(\mathrm{x}) + 2\left(\int|v(\mathrm{x})|\d \mu(\mathrm{x})\right)^2\right)\notag\\
    &\le 3C \int|v(\mathrm{x})|^{2}\d\mu(\mathrm{x}),\label{eq:wu}
  \end{align}
completing the proof.
\end{proof}

The functions
\begin{align*}
  A_{t}(\mathrm{x}) \doteq DF_t\left(X^{\mu}_{0,t}(\mathrm{x}), X^{\mu}_{0,t\sharp}\mu\right), \quad
  B_t(\mathrm{x},\mathrm{y})\doteq \bm{D}F_t\left(X^{\mu}_{0,t}(\mathrm{x}), X^{\mu}_{0,t\sharp}\mu, X^{\mu}_{0,t}(\mathrm{y})\right)
\end{align*}
obviously satisfy the assumptions of the previous lemma.
Hence,~\eqref{eq:Ww} is indeed well-defined and \( v\mapsto w_t \) is a linear bounded map from \( {L}^2_{\mu} \) to itself.
The representation formula \( \mu_t = (X^{\mu}_{0,t})_{\sharp}\mu \) implies the identity
\[
  \int\left|w_t\circ (X_{0,t}^{\mu})^{-1}\right|^{2}\d \mu_t = \int |w_t(\mathrm{x})|^2\d \mu(\mathrm{x}),
\]
meaning that \( v \mapsto w_t\circ (X_{0,t}^{\mu})^{-1} \) is a linear bounded map from \( {L}^2_{\mu} \) to \( L^2_{\mu_t} \).

In the third step, we show that this map is indeed a derivative of \( \Phi_{0,t} \) at \( \mu \).

\begin{lemma}
  \label{lem:s2}
Under the assumptions of Proposition~\ref{prop:Wflowdiff} and for any \( v\in {L}^2_{\mu} \) and \( \varepsilon\in \mathbb{R} \), it holds
\[
W_1\left(\Phi_{0,t}\circ (\id + \varepsilon v)_{\sharp}\mu, (\id + \varepsilon w_t)_{\sharp} \Phi_{0,t}(\mu)\right) \le C\|v\|^2_{\mu}\varepsilon^2,
\]
for some \( C>0 \) depending only on \( F \).
\end{lemma}
\begin{proof}
  \textbf{1.} Let \( v\in {L}^2_{\mu} \) and \( \mu_{\varepsilon}\doteq(\id+\varepsilon v)_{\sharp}\mu\).
  We will write, for brevity, \( X^{\varepsilon}_t \doteq X^{\mu_{\varepsilon}}_{0,t} \).
Recall that \( X^\varepsilon \) satisfies
\[
  \dot X^{\varepsilon}_{t}(\mathrm{x}) = F(X^{\varepsilon}_{t}(\mathrm{x}),X^{\varepsilon}_{t\sharp}\mu_{\varepsilon}),\quad X^{\varepsilon}_{0}(\mathrm{x})=\mathrm{x}.
\]
By replacing \( \mathrm{x} \) with \( \mathrm{x}+\varepsilon v(\mathrm{x}) \), we obtain
\[
  \dot X^{\varepsilon}_t\left(\mathrm{x}+\varepsilon v(\mathrm{x})\right) = F_t\left(X^{\varepsilon}_t\left(\mathrm{x}+\varepsilon v(\mathrm{x})\right),X^{\varepsilon}_{t\sharp}\mu_{\varepsilon}\right),\quad X^{\varepsilon}_0\left(\mathrm{x}+\varepsilon v(\mathrm{x})\right)=\mathrm{x}+\varepsilon v(\mathrm{x}).
\]
Now let \( \varphi^{\varepsilon}_t \doteq X_t^{\varepsilon}\circ (\id + \varepsilon v) \).
Then it holds
\[
  \dot \varphi_t^{\varepsilon}(\mathrm{x}) = F_t\left(\phi_t^{\varepsilon}(\mathrm{x}), \phi_{t\sharp}^{\varepsilon}\mu\right), \quad \phi^{\varepsilon}_0(\mathrm{x}) = \mathrm{x} + \varepsilon v(\mathrm{x}).
\]

\textbf{2.} Consider the following map
\[
  \mathcal{T}(\varepsilon,\varphi)_t(\mathrm{x}) \doteq \mathrm{x} + \varepsilon v(\mathrm{x}) + \int_0^t F_s\left(\phi_s(\mathrm{x}),\phi_{s\sharp}\mu\right)\d s,
\]
where \( \epsilon\in [0,1] \) and \( \varphi \in C(I;{L}^1_{\mu}) \).
The inclusion \( v\in L^2_\mu \) and the boundedness of \( F \) imply that the image of \( \mathcal{T} \) lies in \( C(I;{L}^1_{\mu}) \).

We state that \( \varphi\mapsto \mathcal{T}(\varepsilon,\phi) \) is a contraction if \( C(I;{L}^1_{\mu}) \) is equipped with the norm~\eqref{eq:lambdanorm} and \( \lambda>0 \) is sufficiently large.

Indeed, for any pair of functions \( \phi^{1},\phi^{2} \in C(I;{L}^1_{\mu})\) it holds
\begin{align*}
  \left| \mathcal{T}(\varepsilon,\phi^1)_t(\mathrm{x}) - \mathcal{T}(\varepsilon,\phi^2)_t(\mathrm{x})\right|
  &\le M\int_0^t\left\{|\phi^1_s(\mathrm{x}) - \phi^2_s(\mathrm{x})| + W_2\left(\phi^1_{s\sharp}\mu,\phi^2_{s\sharp}\mu\right)\right\}\d s\\
  &\le M\int_0^t\left\{|\phi^1_s(\mathrm{x}) - \phi^2_s(\mathrm{x})| + \left\|\phi^1_{s}-\phi^2_{s}\right\|_{\mu}\right\}\d s,
\end{align*}
thanks to Assumption \ref{F1}.
After integrating, we obtain
\begin{align*}
  \left\| \mathcal{T}(\varepsilon,\phi^1)_t - \mathcal{T}(\varepsilon,\phi^2)_t\right\|_{L^1_\mu}
  &\le M\int_0^t\left\{\int\left|\phi^1_s - \phi^2_s\right|\d \mu + \left\|\phi^1_{s}-\phi^2_{s}\right\|_{\mu}\right\}\d s.
\end{align*}
Now it follows from
\[
\int \left|\phi^1_s-\phi^2_s\right|\d \mu \le \left\|\phi^1_s-\phi^2_s\right\|_{\mu}
\]
that
\begin{align*}
  \left\| \mathcal{T}(\varepsilon,\phi^1)_t - \mathcal{T}(\varepsilon,\phi^2)_t\right\|_{L^1_\mu}
  &\le 2M\int_0^t \left\|\phi^1_{s}-\phi^2_{s}\right\|_{\mu}\d s\\
  &\le 2M\int_0^t e^{\lambda s} e^{-\lambda s}\left\|\phi^1_{s}-\phi^2_{s}\right\|_{\mu}\d s\\
  &\le 2M\int_0^t e^{\lambda s} \d s \left\|\phi^1-\phi^2\right\|_{C(I;{L}^1_{\mu})}\\
  &\le \frac{2M}{\lambda}e^{\lambda t} \left\|\phi^1-\phi^2\right\|_{C(I;{L}^1_{\mu})}
\end{align*}
Therefore,
\[
  \left\| \mathcal{T}(\varepsilon,\phi^1) - \mathcal{T}(\varepsilon,\phi^2)\right\|_{C(I;{L}^1_{\mu})}\le \frac{2M}{\lambda}\left\|\phi^1-\phi^2\right\|_{C(I;{L}^1_{\mu})},
\]
meaning that \( \varphi\mapsto \mathcal{T}(\varepsilon,\phi) \) is a contraction if \( \lambda > 2M\).

\textbf{3.} Now we will use Banach's contraction principle as it is stated in~\cite[Theorem A.2.1]{ABressan_BPiccoli_2007a}.

By definition, \( \phi^{\varepsilon} \) defined in the first step is a fixed point of \( \mathcal{T}(\varepsilon,\cdot) \).
In terms of~\cite[Theorem A.2.1]{ABressan_BPiccoli_2007a}, we choose \( y = \phi^0 + \varepsilon w \) and \( x(\varepsilon) = \phi^{\varepsilon} \) thus getting
\begin{equation}
  \label{eq:Aphi}
\left\|\phi^0+\varepsilon w - \phi^{\varepsilon}\right\|_{\lambda} \le \kappa \left\|\phi^0+ \varepsilon w - \mathcal{T}\left(\varepsilon,\phi^0+\varepsilon w\right)\right\|_{\lambda}
\end{equation}
for \( \kappa = \frac{\lambda}{\lambda-2M} \).
Since \( \phi^0_t = X_t\doteq X^{\mu}_{0,t} \), the quantity inside the norm in the right-hand side reads:
\[
 \int_0^t F_s\left(X_s, X_{s\sharp}\mu\right)\d s + \varepsilon w_t - \varepsilon v - \int_0^t F_s\left(X_s+\varepsilon w_{s}, (X_s+\varepsilon w_s)_{\sharp}\mu\right)\d s.
\]
By~Lemma~\ref{lem:O2vf}, it holds
\begin{align*}
  \Big\|
F_s&\left(X_s+\varepsilon w_{s}, (X_s+\varepsilon w_s)_{\sharp}\mu\right)-
  F_s\left(X_s, X_{s\sharp}\mu\right) \\
  &-
  \varepsilon DF\left(X_s,X_{s\sharp}\mu\right)w_s + \varepsilon\int \mathbf{D}F\left(X_s,X_{s\sharp}\mu,X_s(\mathrm{y})\right)w_s(\mathrm{y})\d \mu\Big\|_{L^1_{\mu}} \le 4M\|w_s\|^2_{\mu}\varepsilon^{2}.
\end{align*}
Recalling the definition of \( w \) and the estimate~\eqref{eq:wu}, we conclude that the right-hand side of~\eqref{eq:Aphi} is less that \( C\|v\|^2_{\mu}\varepsilon^{2} \) for some \( C>0 \) depending only on \( F \).

\textbf{4.} Recall the relation with the original notation: \(\phi^0_t = X^{\mu}_{0,t} \) and \( \phi^{\varepsilon}_t = X_{0,t}^{\mu_\varepsilon}\circ(\id +\varepsilon v) \).
Hence \( \|\phi^0+\varepsilon w - \phi^{\varepsilon}\|_{\lambda}\le 4C\|v\|^2_{\lambda}\varepsilon^{2} \) implies that
\[
\left\|X_{0,t}^{\mu}+\varepsilon w_t - X_{0,t}^{\mu_{\varepsilon}}\circ (\id + \varepsilon u)\right\|_{L^1_\mu}\le C\|v\|^2_{\mu}\varepsilon^{2},
\]
for all \( t\in I \).
Therefore,
\begin{align*}
  W_1&\left(\Phi_{0,t}\circ (\id + \varepsilon v)_{\sharp}\mu, (\id + \varepsilon w_t)_{\sharp} \Phi_{0,t}(\mu)\right)\\
  &=
    W_1\left(\left(X^{\mu_{\varepsilon}}_{0,t}\circ (\id + \varepsilon v)\right)_{\sharp}\mu, \left((\id + \varepsilon w_t) \circ X^{\mu}_{0,t}\right)_{\sharp}\mu\right)\\
     &\le
     \left\|X_{0,t}^{\mu}+\varepsilon w_t - X_{0,t}^{\mu_{\varepsilon}}\circ (\id + \varepsilon v)\right\|_{L^1_{\mu}}\\
  &\le C\|v\|^2_{\mu}\varepsilon^{2},
\end{align*}
completing the proof.
\end{proof}

\subsection{Proof of Proposition~\ref{prop:Wpull}}

Let \( v_a\in L^2_{\mu_a} \) and \( p_b\in L^2_{\mu_b} \) be arbitrary tangent vectors.
Suppose that \( v_t \), \( t\in [a,b] \), is constructed as in Proposition~\ref{prop:Wflowdiff}.
Thanks to the standard well-posedness results~\cite{bonnetDifferentialInclusionsWasserstein2021,bonnetNecessaryOptimalityConditions2021,averboukhPontryaginMaximumPrinciple2022,pogodaevNonlocalBalanceEquations2022,chertovskihOptimalControlNonlocal2023,MR4097258} for nonlocal continuity equations, the Hamiltonian equation~\eqref{eq:ham} has a unique solution \( \gamma \).
Hence \(\psi_t\) is well-defined by
\[
  p_t(\mathrm{x}) = \int y \d \gamma^\mathrm{x}_t(\mathrm{y}).
\]
Let us show that \( p_t \in L^2_{\mu_t} \).
We first employ Jensen's theorem for vector-valued functions~\cite{perlmanJensensInequalityConvex1974} to get
\[
  \left|\int \mathrm{y} \d \gamma^\mathrm{x}_t(\mathrm{y})\right|^2\le \int |\mathrm{y}|^2\d \gamma_t^\mathrm{x}(\mathrm{y}).
\]
Then, by the definition of the disintegration, we obtain
\[
  \int|p_t(\mathrm{x})|^2\d\mu_t(\mathrm{x})\le \int\left(\int |\mathrm{y}|^2\d \gamma_t^\mathrm{x}(\mathrm{y})\right) \d \mu_t(\mathrm{x}) = \iint |\mathrm{y}|^2\d \gamma_t(\mathrm{x},\mathrm{y}).
\]
The right-hand side is finite because \( \gamma_t\in \mathcal{P}_2(\mathbb{R}^n \times \mathbb{R}^n) \).

If the initial tangent vector \( v_a \) is continuously differentiable as a function of \( x \), then any \( v_t \), \( t\in [a,b] \), is also continuously differentiable (one can employ the arguments similar to those in~\cite[Remark 3.1]{chertovskihOptimalControlNonlocal2023}).
Computations from~\cite[Section 3.4]{chertovskihOptimalControlNonlocal2023} now imply that
\[
  \int \mathrm{y} \cdot v_t(\mathrm{x})\d\gamma_t(\mathrm{x},\mathrm{y}) = \int \mathrm{y} \cdot v_a(\mathrm{x})\d \gamma_a(\mathrm{x},\mathrm{y}) \quad \forall t\in [a,b]
\]
or equivalently
\[
\left<p_t, v_t \right>_{\mu_t} = \left< p_a,v_a \right>_{\mu_a} \quad \forall t\in [a,b].
\]
In the general case, where \( v_a \) is an arbitrary element of \( L^2_{\mu_a} \), the approximation argument can be used: we construct a sequence \( v_a^k\) of continuously differentiable maps converging to \( v_a \) in \( L^2_{\mu_a} \).
This sequence satisfies
\[
\left<p_t, v^k_t \right>_{\mu_t} = \left< p_a,v^k_a \right>_{\mu_a} \quad \forall t\in [a,b].
\]
Since the \( v_{t}^k=(\Phi_{a,t})_{\star,\mu_{a}}v_{a}^k \) and the derivative \( (\Phi_{a,t})_{\star,\mu_{a}} \colon L^2_{\mu_a}\to L^2_{\mu_t} \) is continuous, we conclude that \( v^k_t\to v_{t} \) in \( L^2_{\mu_t} \), for any \( t\in [a,b] \).

We temporarily set \( A\doteq (\Phi_{a,b})_{\star,\mu}\) in order to simplify the notation and note that \( A\colon L^2_{\mu_{a}}\to L^2_{\mu_b} \) is a linear bounded operator.
Its adjoint \( A^{\prime}\colon L^2_{\mu_b} \to L^2_{\mu_{a}} \) is defined by the rule
\[
  \left\langle A \xi, \eta\right\rangle_{\mu_b} = \left< \xi, A^{\prime}\eta \right>_{\mu_a},\quad \xi\in L^2_{\mu_a},\quad \eta\in L^2_{\mu_b}.
\]
If we choose \( \xi = v_a \) and \( \eta = p_b \), we obtain
\[
  \left\langle v_{a},p_a\right\rangle_{\mu_a} = \left\langle v_{b},p_b\right\rangle_{\mu_b} =
  \left\langle Av_{a},p_b\right\rangle_{\mu_b} = \left\langle v_{a},A^{\prime}p_b\right\rangle_{\mu_a}.
\]
Since this holds for an arbitrarily chosen \( v_a\in L^2_{\mu_a} \), we conclude that \( p_a = A^{*}p_b \).
It remains to note that \( A^{*} \) is exactly the pullback operator \( (\Phi_{a,b})^{\star}_{\mu_b} \).

\subsection{Proof of Proposition~\ref{prop:dp_cont}}\label{app:pr-dp_cont}

  First, recall that \( \ell\in \bm{\mathcal D} \) guaranties that gradient \( (\mu,x)\mapsto\bm\nabla\ell(\mu,x) \) is bounded and Lipschitz continuous.

  \begin{lemma}
    \label{lem:terminal_map}
    The map \( \mathcal{T} \colon \mathcal{P}_2(\mathbb{R}^n)\to \mathcal{P}_2(\mathbb{R}^n \times \mathbb{R}^n) \) defined by
    \[
      \mathcal{T}(\vartheta) = (\id, \bm\nabla \ell(\vartheta))_{\sharp}\vartheta
    \]
    is continuous.
  \end{lemma}
  \begin{proof}
    \textbf{1.}
  Let \( \vartheta_k\to \vartheta \).
  Thanks to~\cite[Theorem 6.9]{zbMATH05306371}, it suffices to show that \( \int \varphi \d \mathcal{T}(\vartheta_k)\to \int \phi \d \mathcal{T}(\vartheta) \), for any continuous quadratically bounded function \( \phi \).
  Recall that \( \varphi \) is quadratically bounded if
\[
  |\varphi(\mathrm{x},\mathrm{y})| \le C\left(1+|\mathrm{x}|^2 + |\mathrm{y}|^{2}\right) \quad \forall \mathrm{x},\mathrm{y} \in \mathbb{R}^{n},
\]
for a certain \( C>0 \).

Now, due to the definition of \( \mathcal{T} \), we must demonstrate that
\begin{align*}
  \int \varphi \left( \mathrm{x}, \bm\nabla \ell\left(\vartheta_k,\mathrm{x}\right) \right)\vartheta_k(\mathrm{x}) \to
  \int \varphi \left( \mathrm{x}, \bm\nabla \ell\left(\vartheta,\mathrm{x}\right) \right)\d \vartheta(\mathrm{x}).
  \end{align*}
  We will prove this in two steps, by showing that
  \begin{align}
    \int \left[\varphi \left( \mathrm{x}, \bm\nabla \ell\left(\vartheta_k,\mathrm{x}\right) \right) -
    \varphi \left( \mathrm{x}, \bm\nabla \ell\left(\vartheta,\mathrm{x}\right) \right)\right]\d \vartheta_k(\mathrm{x})\to 0,\label{eq:gamma1}\\
    \int \varphi \left( \mathrm{x}, \bm\nabla \ell\left(\vartheta,\mathrm{x}\right) \right)\d\left[ \vartheta_k - \vartheta\right](\mathrm{x})\to 0.\label{eq:gamma2}
  \end{align}

  \textbf{2.} Let us prove~\eqref{eq:gamma1}.
  Since \( \varphi \) is continuous, for any \( \varepsilon>0 \) there exists \( \delta>0 \) such that
  \[
    \left|\varphi \left( \mathrm{x}, \bm\nabla \ell\left(\vartheta_k,\mathrm{x}\right)\right) -
    \varphi \left( \mathrm{x}, \bm\nabla \ell\left(\vartheta_k,\mathrm{x}\right) \right)\right|< \varepsilon
  \]
  as soon as
  \begin{align*}
    \left|\bm\nabla \ell\left(\vartheta_k,\mathrm{x}\right)
    -\bm\nabla \ell\left(\vartheta,\mathrm{x}\right)
    \right|\le \Lip(\bm\nabla \ell)W_2\left(\vartheta_k,\vartheta\right)<\delta.
  \end{align*}
  Therefore, for any \( \varepsilon>0 \) there exists \( N\in \mathbb{N} \) such that
\[
  \int \left[\varphi \left( \mathrm{x}, \bm\nabla \ell\left(\vartheta_k,\mathrm{x}\right) \right) -
  \varphi \left( \mathrm{x}, \bm\nabla \ell\left(\vartheta, \mathrm{x}\right) \right)\right]\d \vartheta_k(\mathrm{x})< \varepsilon,
\]
for all \( k\ge N \).
Thus~\eqref{eq:gamma1} indeed holds.

\textbf{3.} Since both functions \( \varphi \) and \( x\mapsto \bm \nabla \ell(\vartheta,x) \) are continuous, their composition \( \mathrm{x}\mapsto \phi \left( \mathrm{x}, \bm\nabla \ell\left(\vartheta,\mathrm{x}\right)\right) \)  is continuous as well.
  Let us check that the composition is quadratically bounded.
  In fact, it follows from
  \[
    \left|\varphi \left( \mathrm{x}, \bm\nabla \ell\left(\vartheta,\mathrm{x}\right)\right)   \right|\le C \left( 1+ |\mathrm{x}|^2 + \left|\bm\nabla \ell\left(\vartheta,\mathrm{x}\right)\right|^2 \right) \quad \forall \mathrm{x}\in \mathbb{R}^n
   \]
   and the boundedness of \( \bm \nabla \ell \).
   Now,~\eqref{eq:gamma2} follows from~\cite[Theorem 6.9]{zbMATH05306371}.
 \end{proof}

 Denote by \( \bm \Gamma \) the set of all trajectories of the Hamiltonian equation~\eqref{eq:ham} (without any fixed terminal condition).
 It is known that \( \bm \Gamma \) is a closed subset of \( C\left(I;\mathcal{P}_2(\mathbb{R}^n \times \mathbb{R}^n)\right) \).
 The standard (supremum) distance on this space we denote by \( d_{C} \).

 In oreder to prove the second statement of Proposition~\ref{prop:dp_cont}, we fix some \( s\in I \). 

 \begin{lemma}
   \label{lem:rhs}
   The map \( \mathcal{F}\colon I \times \bm \Gamma \mapsto \mathbb{R} \) defined by
   \[
     \mathcal{F}(t,\gamma) = \iint y\cdot g_s(\pi^1_{\sharp}\gamma_t,\mathrm{x})\d \gamma_t(\mathrm{x},\mathrm{y})
   \]
   is continuous.
 \end{lemma}
 \begin{proof}
   Let \( t^k\to t \) and \( \gamma^k\to \gamma \).
  Consider the following difference
  \begin{align}
    \iint \mathrm{y}\cdot g_s(\pi^1_{\sharp}\gamma^k_{t^k},\mathrm{x})\d \gamma^k_{t^k}(\mathrm{x},\mathrm{y}) &-
    \iint \mathrm{y}\cdot g_s(\pi^1_{\sharp}\gamma_t,\mathrm{x})\d \gamma_{t}(\mathrm{x},\mathrm{y})\notag\\
    &\le
    \iint \mathrm{y}\cdot \left[g_s\left(\pi^1_{\sharp}\gamma^k_{t^k},\mathrm{x}\right)-g_s\left(\pi^1_{\sharp}\gamma_{t},\mathrm{x}\right)\right]\d \gamma^{k}_{t^k}(\mathrm{x},\mathrm{y})\notag\\
    &+
    \iint y\cdot g_s\left(\pi^1_{\sharp}\gamma_t,\mathrm{x}\right)\d \left[\gamma^k_{t^k}(\mathrm{x},\mathrm{y})-\gamma_t(\mathrm{x},\mathrm{y})\right]\label{eq:key1}
  \end{align}
  and show that every term in the right-hand side vanishes as \( k\to \infty \).

  Due to Lipschitz continuity of \( g_s \) and \( \pi^1 \), it holds
\begin{align*}
\left|g_s\left(\pi^1_{\sharp}\gamma^k_{t^k},\mathrm{x}\right)-g_s\left(\pi^1_{\sharp}\gamma_{t},\mathrm{x}\right)\right| & \leq MW_2(\gamma^k_{t^k},\gamma_t)\\
& \leq M \left( W_2(\gamma^k_{t^k},\gamma_{t^k}) + W_2(\gamma_{t^k},\gamma_t) \right)\to 0,
\end{align*}
for all \( \mathrm{x}\in \mathbb{R}^n \).
Moreover, by~\cite[Theorem 6.9]{zbMATH05306371}, \( \iint \mathrm{y}\d\gamma^k_{t^k}\to \iint \mathrm{y}\d \gamma_{t} \).
  Therefore,
  \begin{equation*}
    \iint \mathrm{y}\cdot \left[g_s\left(\pi^1_{\sharp}\gamma^k_{t^k},\mathrm{x}\right)-g_s\left(\pi^1_{\sharp}\gamma_{t},\mathrm{x}\right)\right]\d \gamma^{k}_{t^k}(\mathrm{x},\mathrm{y})\to 0.
  \end{equation*}

  Consider now the second term in the right-hand side of~\eqref{eq:key1}.
  By our assumptions, the map \( (\mathrm{x},\mathrm{y})\mapsto \mathrm{y}\cdot g_s(\pi^1_{\sharp}\gamma_t,\mathrm{x}) \) is continuous.
  Moreover, it is linearly (and thus quadratically) bounded since \( |\mathrm{y}\cdot g_s(\mu,\mathrm{x})|\le |\mathrm{y}|M \), where \( M \) is an upper bound of \( |g| \).
  Thus
  \begin{equation}
    \label{eq:key}
    \iint \mathrm{y}\cdot g_s\left(\pi^1_{\sharp}\gamma_t,\mathrm{x}\right)\d \left[\gamma^k_{t_k}(\mathrm{x},\mathrm{y})-\gamma_t(\mathrm{x},\mathrm{y})\right]\to 0,
  \end{equation}
by~\cite[Theorem 6.9]{zbMATH05306371}.
\end{proof}

\subsubsection{Proof of Statement (2)}
\label{sssec:statement2}

The rest of the proof of Statement (2) consists in two steps.

\textbf{Step 1.} Denote by \( \mu[t_0,\vartheta_0] \) the trajectory of the original system corresponding to the condition \( \mu_{t_0}=\vartheta_0 \).
Since trajectories depend continuously on the initial data, the map
\[
  (t_0,\vartheta_0)\mapsto \mu_T[t_0,\vartheta_0]
\]
is continuous as \( I \times \mathcal{P}_{2}\to \mathcal{P}_2 \).
Then we apply \( \mathcal{T} \) to the terminal measure \( \mu_T[t_0,\vartheta_0] = \Phi_{t_0,T}(\vartheta_0) \) and obtain a terminal measure \( \gamma_T=\gamma_T[t_0,\vartheta_0] \) for~\eqref{eq:ham}.
Finally, we construct a trajectory \( \gamma = \gamma[t_0,\vartheta_0] \) of~\eqref{eq:ham} issuing from this terminal measure, see the scheme below:
\[
  I \times \mathcal{P}_2(\mathbb R^n)\xrightarrow{(t_0,\vartheta_0)\mapsto \mu_T[t_0,\vartheta_0]} \mathcal{P}_2(\mathbb R^n) \xrightarrow{\quad\mathcal{T}\quad}\mathcal P_2(\mathbb R^n \times\mathbb R^n)
\]
\[  
\xrightarrow{\;\eqref{eq:ham}\;}\bm \Gamma \subset C\left(I;\mathcal P_2(\mathbb R^n \times\mathbb R^n)\right).
\]
It follows from Lemma~\ref{lem:terminal_map} and the continuous dependence of trajectories of~\eqref{eq:ham} on initial data that the map \( (t_0,\vartheta_0)\mapsto \gamma[t_0,\vartheta_0] \) is continuous.
Now, we apply the maps
\[
  I \times \mathbf{\Gamma} \xrightarrow{\quad\mathcal{F}\quad} \mathbb{R}
\]
constructed in Lemma~\ref{lem:rhs} to the pair \( (t_0,\gamma[t_0,\vartheta_0]) \).
Since \( \mathcal{F} \) is continuous, we conclude that
\[
  (t_0,\vartheta_0) \mapsto \iint \mathrm{y}\cdot g_s\left(\pi^1_{\sharp}\gamma_{t_0}[t_0,\vartheta_0],\mathrm{x}\right)\d \gamma_{t_0}[t_0,\vartheta_0](\mathrm{x},\mathrm{y})
\]
is continuous.
It remains to recall~\cite{chertovskihOptimalControlNonlocal2023} that
\[
\pi^1_{\sharp} \gamma_{t_0}[t_0,\vartheta_0] = \mu_{t_0}[t_0,\vartheta_0] = \vartheta_0.
\]
That is,
\[
  (t_0,\vartheta_0) \mapsto \iint \mathrm{y}\cdot g_s\left(\vartheta_0,\mathrm{x}\right)\d \gamma_{t_0}[t_0,\vartheta_0](\mathrm{x},\mathrm{y})
\]
is continuous.

\begin{figure}
  \centering
  \includegraphics[width=0.5\textwidth]{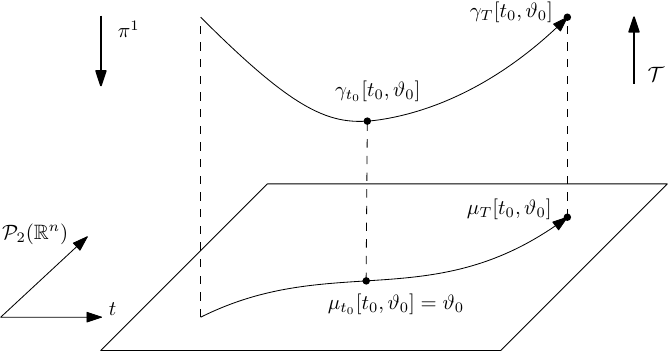}
  \caption{Relation between \( \mu[t_0,\vartheta_0] \) and its lift \( \gamma[t_0,\vartheta] \).}
  \label{fig:lift}
\end{figure}

\textbf{Step 2.} Now we are going to understand how the above map is related with \( \mathbf{d}p_t(g) \).

By Lemmas~\ref{lem:Wcommut}, \ref{lem:testlip} and Proposition~\ref{prop:Wflowdiff}, the function \( \mathbf{p}_{t_0} \) is differentiable and
  \[
    \mathbf{d}\mathbf{p}_{t_0} = - \mathbf{d}\left(\Phi_{t_0,T}^{\star}\ell\right) = \Phi_{t_0,T}^{\star}(\mathbf{d}\ell).
  \]
  Recall that \( (\Phi_{t_0,T})^{\star}_{\vartheta_0} \) maps the tangent vector \( p_T \doteq \bm \nabla\ell(\Phi_{t_0,T}(\vartheta_0)) = \bm \nabla\ell(\mu_{T}[t_0,\vartheta_0]) \) at the point \( \Phi_{t_0,T}(\vartheta_0) \) into the tangent vector
  \[
    p_{t_0}(\mathrm{x}) \doteq \int \mathrm{y}\d \gamma_{t_0}^x[t_0,\vartheta_0](\mathrm{y})
  \]
  at \( \vartheta_0 \).
  Therefore,
\[
  \bm d \mathbf{p}_{t_0}(g_s)\big|_{\vartheta_0} = \left< p_{t_0},g_s(\vartheta_0) \right>_{\vartheta_0} =
  \iint \mathrm{y}\cdot g_s(\vartheta_0,\mathrm{x})\d \gamma_{t_0}[t_0,\vartheta_0](\mathrm{x},\mathrm{y}),
\]
completing the proof.

\subsubsection{Proof of Statement (3)}
\label{sssec:statement3}

Using the notation of Section~\ref{sssec:statement2}, we are going to prove that the family of maps
  \[
    f_{t_0}(\vartheta_0) \doteq \left< \bm \nabla \mathbf{p}_{t_0}(\vartheta_0), g_{t_0}(\vartheta_0) \right>_{\vartheta_0} = \iint \mathrm{y}\cdot g_{t_0}(\vartheta_0,\mathrm{x})\d \gamma_{t_0}[t_0,\vartheta_0](\mathrm{x},\mathrm{y}),
    \]
    \[
    t_0\in I,\quad  \vartheta_0\in \mathcal{P}_2,
  \]
  is equicontinuous, that is, for any \( \varepsilon>0 \) there exists \( \delta>0 \) such that the implication
  \[
    W_2(\vartheta_0,\vartheta_0')<\delta \quad \implies \quad |f_{t_0}(\vartheta_0)-f_{t_0}(\vartheta_0')|<\varepsilon
  \]
  holds for all \( t_0\in I \) and \( \vartheta_0,\vartheta_0'\in \mathcal{P}_2 \).

  Let \( \Pi\in \mathcal{P}_2(\mathbb{R}^{2n} \times \mathbb{R}^{2n}) \) be an optimal transport plan between \( \gamma_{t_0}[t_0,\vartheta_0] \) and \( \gamma_{t_0}[t_0,\vartheta_0'] \). 
    Thus, \( \pi^1_\sharp \Pi = \gamma_{t_0}[t_0,\vartheta_0]  \) and \( \pi^2_\sharp \Pi = \gamma_{t_0}[t_0,\vartheta_0']  \). 
    Then 
    \begin{align*}
      &\hspace{-0.3cm}\left|f_{t_0}(\vartheta_0) - f_{t_0}(\vartheta_0')\right|\\
       &= 
       \left|\iint \mathrm{y}\cdot g_{t_0}(\vartheta_0,\mathrm{x})\d\gamma_{t_0}[t_0,\vartheta_0](\mathrm{x},\mathrm{y}) -\iint \mathrm{y}\cdot g_{t_0}(\vartheta_0',\mathrm{x})\d\gamma_{t_0}[t_0,\vartheta_0'](\mathrm{x},\mathrm{y})\right|\\
       &=
       \iiiint \left|\mathrm{y}_1\cdot g_{t_0}(\vartheta_0,\mathrm{x}_1) - \mathrm{y}_2\cdot g_{t_0}(\vartheta_0',\mathrm{x}_2)\right|\d \Pi(\mathrm{x}_1,\mathrm{y}_1,\mathrm{x}_2,\mathrm{y}_2)  \\
       &\le
       \iiiint \left|\mathrm{y}_1\cdot \left(g_{t_0}(\vartheta_0,\mathrm{x}_1)-  g_{t_0}(\vartheta_0',\mathrm{x}_2)\right)\right|\d \Pi(\mathrm{x}_1,\mathrm{y}_1,\mathrm{x}_2,\mathrm{y}_2)\\
       &+
       \iiiint \left|\left(\mathrm{y}_1- \mathrm{y}_2\right)\cdot g_{t_0}(\vartheta_0',\mathrm{x}_2)\right|\d \Pi(\mathrm{x}_1,\mathrm{y}_1,\mathrm{x}_2,\mathrm{y}_2)\\
       &\le
       \left(\iiiint \left|\mathrm{y}_1\right|^2\d \Pi(\mathrm{x}_1,\mathrm{y}_1,\mathrm{x}_2,\mathrm{y}_2)\right)^{1/2}
     \\&\qquad\cdot \left(\iiiint \left|g_{t_0}(\vartheta_0,\mathrm{x}_1)-  g_{t_0}(\vartheta_0',\mathrm{x}_2)\right|^2\d\Pi(\mathrm{x}_1,\mathrm{y}_1,\mathrm{x}_2,\mathrm{y}_2)\right)^{1/2}\\
       &+
       \left(\iiiint \left|\mathrm{y}_1- \mathrm{y}_2\right|^2\d \Pi(\mathrm{x}_1,\mathrm{y}_1,\mathrm{x}_2,\mathrm{y}_2)\right)^{1/2}\\ 
       &\qquad\cdot
       \left(\iiiint \left|g_{t_0}(\vartheta_0',\mathrm{x}_2)\right|^2\d \Pi(\mathrm{x}_1,\mathrm{y}_1,\mathrm{x}_2,\mathrm{y}_2)\right)^{1/2}\\ 
       &\le
       \left(\iint \left|\mathrm{y}_1\right|^2\d \gamma_{t_0}[t_0,\vartheta_0](\mathrm{x}_1,\mathrm{y}_1)\right)^{1/2}\\ 
       &\qquad\cdot M\left(\iiiint \left(|\mathrm{x}_1-\mathrm{x}_2| + W_2(\vartheta_0,\vartheta_0')\right)^2\d\Pi(\mathrm{x}_1,\mathrm{y}_1,\mathrm{x}_2,\mathrm{y}_2)\right)^{1/2}\\
       &+
       \left(\iiiint |\mathrm{y}_1- \mathrm{y}_2|^2\d \Pi(\mathrm{x}_1,\mathrm{y}_1,\mathrm{x}_2,\mathrm{y}_2)\right)^{1/2}\cdot
M.
    \end{align*}

  Recall that \( (t_0,\vartheta_0)\mapsto \gamma[t_0,\vartheta_0] \) is continuous. Hence the continuous function 
  \[
    (t_0,\vartheta_0)\mapsto \iint \left|\mathrm{y}_1\right|^2\d \gamma_{t_0}[t_0,\vartheta_0](\mathrm{x}_1,\mathrm{y}_1)
  \]
  maps \( I \times \mathcal{K} \) to a compact subset of \( \mathbb{R} \). 
  In particular, there exists \( C_1>0 \) such that 
\[
  \iint \left|\mathrm{y}_1\right|^2\d \gamma_{t_0}[t_0,\vartheta_0](\mathrm{x}_1,\mathrm{y}_1) \le C_1 \quad \forall (t_0,\vartheta_0)\in I \times \mathcal{K}.
\] 
Moreover, the map  \( (t_0,\vartheta_0)\mapsto \gamma[t_0,\vartheta_0] \) restricted to the compact set \( I \times \mathcal{K} \) is uniformely continuous.
In particular, for each \( \epsilon>0 \) there exists \( \delta>0 \) such that
\[
\forall \vartheta_0,\vartheta_0'\in \mathcal{K} \quad \forall t_0\in I \quad
  W_2(\vartheta_0,\vartheta_0')< \delta \quad\implies \quad W_2\left(\gamma_{t_0}[t_0,\vartheta_0], \gamma_{t_0}[t_0,\vartheta_0']\right)<\epsilon.
\]

By the very definition of \( \Pi \), the quantities 
\[
  \iiiint |\mathrm{y}_1- \mathrm{y}_2|^2\d \Pi(\mathrm{x}_1,\mathrm{y}_1,\mathrm{x}_2,\mathrm{y}_2), \quad 
  \iiiint |\mathrm{x}_1- \mathrm{x}_2|^2\d \Pi(\mathrm{x}_1,\mathrm{y}_1,\mathrm{x}_2,\mathrm{y}_2) 
\]
are no greater than \(  W_2^2 \left(\gamma_{t_0}[t_0,\vartheta_0], \gamma_{t_0}[t_0,\vartheta_0']\right)\).

Therefore, for any \( \epsilon>0 \) there exists \( \delta>0 \) such that the conditions \(\vartheta_0,\vartheta_0'\in \mathcal{K}\), \(t_0\in I\),
  \(W_2(\vartheta_0,\vartheta_0')< \delta\) imply that 
\[
  \left|f_{t_0}(\vartheta_0) - f_{t_0}(\vartheta_0')\right|\le C_1 M (\epsilon^2+2 \epsilon \delta +\delta^2)^{1/2} + \epsilon M,
\]
which proves equicontinuity.

\subsection{Proof of Proposition~\ref{prop:sysdif}}

As a preliminary step, we prove a technical lemma.

\begin{lemma}
  \label{lem:carat}
  Let \( f \colon [0,1] \times [0,1] \to \mathbb{R} \) be Lebesgue measurable in the first variable and equicontinuous in the second, i.e., for any \( \varepsilon>0 \) there exists \( \delta>0 \) such that
\[
|t_1-t_2|<\delta \quad \implies \quad  |f(s,t_{1}) - f(s,t_2)|< \varepsilon,
\]
for all \( t_{1},t_2\in [0,1] \) and almost all \( s\in [0,1] \).
Then
\[
\lim_{h\to 0}\frac{1}{h}\int_t^{t+h}\left|f(s,t)-f(t,t)\right|\d s=0,
\]
for almost all \( t\in [0,1] \).
\end{lemma}
\begin{proof}
  Note that
  \begin{align*}
\frac{1}{h}\int_t^{t+h}\left|f(s,t)-f(t,t)\right|\d s & \le
\frac{1}{h}\int_t^{t+h}\left|f(s,t)-f(s,s)\right|\d s\\
& +
\frac{1}{h}\int_t^{t+h}\left|f(s,s)-f(t,t)\right|\d s.      
  \end{align*}
  Since \( f \) is a Carath\'eodory function, the map \( t \mapsto f(t,t) \) is Lebesgue measurable.
  Therefore,
\[
\frac{1}{h}\int_t^{t+h}\left|f(s,s)-f(t,t)\right|\d s \to 0,
\]
for a.e. \( t\in [0,1] \).

By our assumptions, for any \( \varepsilon>0 \) there exists \( \delta>0 \) such that
\[
|s-t|<\delta \quad \implies \quad |f(s,t)-f(s,s)|<\varepsilon.
\]
Thus for all \( h<\delta \) we have
\[
  \frac{1}{h}\int_t^{t+h}\left|f(s,t)-f(s,s)\right|\d s<
  \frac{1}{h}\int_t^{t+h}\varepsilon\d s = \varepsilon.
\]
Hence \( \frac{1}{h}\int_t^{t+h}\left|f(s,t)-f(s,s)\right|\d s \to 0\) for all \( t\in [0,1] \).
This observation completes the proof.
\end{proof}

\begin{proofof}{Proposition~\ref{prop:sysdif}}
  \textbf{1.} We start with the obvious identity:
  \begin{align*}
    \mathbf{p}_{t+h}(\mu_{t+h}) - \mathbf{p}_t(\mu_t) =
    \left[\mathbf{p}_{t+h}(\mu_{t+h}) - \mathbf{p}_{t+h}(\mu_t)\right] + \left[\mathbf{p}_{t+h}(\mu_t) - \mathbf{p}_t(\mu_t)\right].
  \end{align*}
  Consider the first difference on the right-hand side:
  \begin{align*}
   \mathbf{p}_{t+h}(\mu_{t+h}) - \mathbf{p}_{t+h}(\mu_t) &=  \mathbf{p}_{t+h}(\mu_{t+h}) - \mathbf{p}_{t+h}((\id+h v_t)_{\sharp}\mu_t)
   \\&+ \mathbf{p}_{t+h}((\id+h v_{t})_{\sharp} \mu_t) - \mathbf{p}_{t+h}(\mu_t).
  \end{align*}
It is known that
  \[
    \lim_{h\to 0}\frac{W_2\left(\mu_{t+h},(\id + h g_t)_{\sharp}\mu_t\right)}{h} = 0,
  \]
  for a.e. \( t\in I \).
  Thus for all such time moments \( t \) it holds
  \begin{align*}
    \lim_{h\to0} & \frac{\left|\mathbf{p}_{t+h}(\mu_{t+h}) - \mathbf{p}_{t+h}((\id +h v_{t})_{\sharp}\mu_t)\right|}{h}\\
    &\le
    \Lip(\mathbf{p}_{t+h})\lim_{h\to0}\frac{W_2\left(\mu_{t+h},(\id +h v_{t})_{\sharp}\mu_t\right)}{h} = 0.
  \end{align*}

  Thanks to Proposition~\ref{prop:dp_cont}(1), the family \( \mathbf{p}_s = \ell \circ \Phi_{s,T}\) is uniformly equidifferentiable w.r.t. $s$, meaning that
  \[
    \left|\mathbf{p}_{s}((\id +h v_t)_{\sharp}\mu_t) - \mathbf{p}_{s}(\mu_t) - \left<\bm \nabla \mathbf{p}_{s}(\mu_t), v_t\right>_{\mu_t}\right|\le C\|v_t\|^2_{\mu_t}h^2,
  \]
  for all \( s\in I \) and \( h\in \mathbb{R} \) (here we use the fact that \( C \) depends only on \( F \) and \( \ell \)).
  As a consequence,
 \[
   \lim_{h\to 0}\frac{\mathbf{p}_{t+h}((\id+h v_{t})_{\sharp} \mu_t) - \mathbf{p}_{t+h}(\mu_t)}{h} = \lim_{h\to 0} \left< \bm \nabla \mathbf{p}_{t+h}(\mu_t),v_t \right>_{\mu_t} = \left< \bm \nabla \mathbf{p}_{t}(\mu_t),v_t \right>_{\mu_t}.
 \]
 The latter identity follows from Proposition~\ref{prop:dp_cont}(2).

 \textbf{2.} In the second step we study the limit
  \[
    \lim_{h\to0}\frac{\mathbf{p}_{t+h}(\mu_t) - \mathbf{p}_t(\mu_t)}{h}.
  \]

  We start with the identity
  \[
    \mathbf{p}_{t+h}(\mu) - \mathbf{p}_t(\mu) = \int_t^{t+h} \partial_t \mathbf{p}_{s}(\mu)\d s
  \]
  holding for all \( t\in I \) and \( \mu\in \mathcal{P}_2 \) due to Lipschitz continuity of \( t\mapsto \mathbf{p}_t(\mu) \).
  Now, that fact that \( \mathbf{p} \) satisfies~\eqref{eq:Wadjoint} allows us to write
  \[
    \mathbf{p}_{t+h}(\mu) - \mathbf{p}_t(\mu) = \int_t^{t+h} \left< \bm \nabla \mathbf{p}_{s}(\mu), F_s(\mu,u_s) \right>_{\mu}\d s.
  \]
  As a consequence,
  \[
    \frac{\mathbf{p}_{t+h}(\mu_t) - \mathbf{p}_t(\mu_t)}{h} = \frac{1}{h}\int_t^{t+h} f(s,t) \d \tau,\quad \text{where} \quad f(s,t) \doteq \left< \bm \nabla \mathbf{p}_{s}(\mu_t), F_s(\mu_t,u_s) \right>_{\mu_t}.
  \]
  Thanks to Proposition~\ref{prop:dp_cont}(3), \( f \) is measurable in \( s \) and equicontinuous in \( t \) (here we use the fact that \( t\mapsto \mu_t \) maps \( I \) into a compact subset of \( \mathcal{P}_2 \)).
  Thus, by Lemma~\ref{lem:carat},
\[
  \lim_{h\to 0}\frac{\mathbf{p}_{t+h}(\mu_t) - \mathbf{p}_t(\mu_t)}{h} = f(t,t) = \left< \bm \nabla \mathbf{p}_{t}(\mu_t), F_t(\mu_t,u_t) \right>_{\mu_t},
  \]
  for almost all \( t \in I \).

\section{Proofs related to Section~\ref{sec:generapp}}\label{app:4}

\subsection{Proof of the assertion in Example~\ref{ex:6}}\label{apex4}

The integrand in the corresponding condition \eqref{Ls-2}, written as
\[
    \lim_{h \to 0+}\frac{1}{h}\int_t^{t+h}\sup_{\mu \in \mathcal P_2, \, \alpha \in \Delta}\big|[\mathfrak{L}_s- \mathfrak{L}_t](\phi\circ \Phi_\alpha)(\mu)\big| \d s=0,
\]
is majorated by the function
\[
     \big|u(t)-u(s)\big| \,\max_j \|F^j(\mu)\|_{\mu} \left\|\bm \nabla (\phi\circ \Phi_\alpha)(\mu)\right\|_{\mu},
\]
where we employed the definition \eqref{Lnlc} and the Cauchy–Schwarz inequality. 

Now, we observe that
\[
    \bm \nabla (\phi\circ \Phi_\alpha) = [(\Phi_\alpha)_{\star}]' \bm \nabla \phi\left(\Phi_\alpha\right)
\]
where the prime denotes taking the adjoint of an operator, and recall that, for any bounded linear map $\mathfrak{A}$, its operator norm $\|\mathfrak A\|$ equals the norm of \(\mathfrak A'\). Thus,  
\[
    \|[(\Phi_\alpha)_{\star}]' \bm \nabla \phi\left(\Phi_\alpha\right)\|_{\mu} \leq \|[(\Phi_\alpha)_{\star}]'\| \|\bm \nabla \phi \left(\Phi_\alpha\right)\|_{{\Phi_\alpha(\mu)}} \leq \|(\Phi_\alpha)_{\star}\| \|\bm \nabla \phi\|_{\infty}.
\]

Now, by leveraging Proposition~\ref{prop:Wflowdiff} and using the inequality \eqref{w-infti}, the expression on the right is estimated by
 \begin{align*}
\|\bm \nabla \phi\|_{\infty}\left(1+T\|\bm D F\|_{\infty}\,e^{T\big(\|DF\|_{\infty}+ \|\bm DF\|_{\infty}\big)}\right)
    e^{T\|DF\|_\infty},
\end{align*}
and the result follows.

\subsection{Proof of Proposition~\ref{propo:ex2}}\label{app:proofex2}

The first part of the proposition was established in Example~\ref{ex_5}. To prove the second part, we analyze the following expression:
\begin{align*}
    &\left|\nabla \phi \circ \Phi_{s+h, t} \, D \Phi_{s+h, t} f_s - \nabla \phi \circ \Phi_{s, t} \, D \Phi_{s, t} f_s\right| \leq \\
    &\qquad \, |\nabla \phi \circ \Phi_{s+h, t} - \nabla \phi \circ \Phi_{s, t}|\exp\left(\int_{s+h}^t |Df_\tau \circ \Phi_{s+h,\tau}| \, d \tau \right) \|f\|_\infty \\
    &\quad + |\nabla \phi \circ \Phi_{s, t}| \, |y_t| \, \|f\|_\infty,
\end{align*}
where \( y_t \doteq D\Phi_{s+h,t} - D\Phi_{s,t} \). 

The first term can be estimated using the Lipschitz continuity of \( \nabla \phi \) and \( \Phi \), leading to:
\[
|\nabla \phi \circ \Phi_{s+h, t} - \nabla \phi \circ \Phi_{s, t}| \leq \Lip(\nabla \phi) \Lip(\Phi)^2 h,
\]
and the exponential term accounts for the growth induced by \( D\Phi_{s+h,t} \). Thus, this term is bounded by
\[
\|f\|_\infty \Lip(\nabla \phi) \Lip(\Phi)^2 \exp(\|Df\|_\infty h) h.
\]

To estimate the second term, observe that \( y_t \) satisfies the following ODE:
\[
\dot{y}_t = Df_t(\Phi_{s+h,t}) y_t + \big(Df_t(\Phi_{s+h,t}) - Df_t(\Phi_{s,t})\big) D\Phi_{s,t}, \quad y_{s+h} = \id - D\Phi_{s, s+h}.
\]
Using the Lipschitz continuity of \( Df \), we have:
\[
|Df_t(\Phi_{s+h,t}) - Df_t(\Phi_{s,t})| \leq \Lip(Df) \Lip(\Phi)^2 h \doteq C_1 h.
\]
Applying Grönwall's lemma to this ODE yields:
\[
|y_t| \leq h \, |\id - D\Phi_{s,s+h}| \, C_1 |D\Phi_{s,t}| \exp(\|Df\|_\infty t).
\]
Since \( |D\Phi_{s,t}| \) is bounded by \( \exp(\|Df\|_\infty (t-s)) \), we can conclude:
\[
|y_t| \leq C_1 h \exp(\|Df\|_\infty t),
\]
where \( C_1 \) depends on \( \Lip(Df) \) and \( \Lip(\Phi) \).

Combining these results, we obtain the bound:
\[
\big|\mathfrak{L}_s \left(\bm{\Psi}_{t, s+h} - \bm{\Psi}_{t,s}\right)\phi\big| \doteq \big|\nabla \phi \circ \Phi_{s+h, t} \, D \Phi_{s+h, t} f_s - \nabla \phi \circ \Phi_{s, t} \, D \Phi_{s, t} f_s\big| \leq C_2 h,
\]
where \( C_2 \) is a constant depending only on \( \|f\|_\infty \), \( \|Df\|_\infty \), and the Lipschitz constants of \( \nabla \phi, Df \), and \( \Phi \). This completes the proof.

\end{proofof}

\subsection{Proof of Proposition~\ref{prop:ex12}}\label{app:proof-ex12}

\textbf{1.} In Example~\ref{ex1}, the control set $U$ is specified as a closed ball in the space \( \text{Lip}_0(\mathbb{R}^n; \mathbb{R}^m) \).

\begin{remark}\label{AES}
Recall that the pre-dual of the space \( \text{Lip}_0(\mathbb{R}^n; \mathbb{R}^m) \) of vector-valued Lipschitz functions is identified as the direct sum
\[
\bm{\mathcal{F}}(\mathbb{R}^n; \mathbb{R}^m) = \bigoplus_{i=1}^m \mathcal{F}(\mathbb{R}^n),
\]
where \( \mathcal{F}(\mathbb{R}^n) \) denotes the Arens-Eells space over \( \mathbb{R}^n \) \cite{Godefroy}.

The Arens-Eells space is  the completion of the space of finitely supported signed measures on \( \mathbb{R}^n \) with respect to the norm
\[
\|\mu\|_{\bm{\mathcal{F}}} = \sup \left\{ \left| \langle f, \mu \rangle \right|: f \in \text{Lip}_0(\mathbb{R}^n; \mathbb{R}), \|f\|_{\text{Lip}} \leq 1 \right\},
\]
where the pairing \( \langle f, \mu \rangle \) is established via the Lebesgue integral
\[
\langle f, \mu \rangle \doteq \int f \, \mathrm{d} \mu.
\]
We emphasize that, here, the functions \( f \) act on signed measures \( \mu \) as linear functionals.
\end{remark}

Without loss of generality, we can take: \( U \doteq \{\mathrm{u} \in \bm{Lip}_0 : \Lip(\mathrm{u}) \leq 1\} \). 

The existence of a unique solution of the ODE \eqref{ODE-ex1} on the entire interval \( I \) is evident; thus, it suffices to verify the continuity of the mapping \( u \mapsto x^u \).

For any \( \bar{u}, u \in \mathcal{U} \), we estimate:
\begin{align*}
    \big|x(t) - \bar{x}(t)\big| &= \left|\int_{0}^t \big[u_s(x(s)) - \bar{u}_s(\bar{x}(s))\big] \d{s}\right| \\
    &\leq \int_{0}^t \big|u_s(x(s)) - u_s(\bar{x}(s))\big| \d{s} + \left|\int_{0}^t \big[u_s(\bar{x}(s)) - \bar{u}_s(\bar{x}(s))\big] \d{s}\right| \\
    &\leq \int_{0}^t \big|x(s) - \bar{x}(s)\big| \d{s} + \left|\int_{0}^t \langle u_s - \bar{u}_s, \delta_{\bar{x}(s)}\rangle_{(\bm{V}', \bm{V})} \right|,
\end{align*}
which, by Gr\"onwall's lemma, yields the bound
\[
    \big|x(t) - \bar{x}(t)\big| \leq e^t \left|\int_{0}^t \langle u_s - \bar{u}_s, \delta_{\bar{x}(s)}\rangle_{(\bm{V}', \bm{V})} \right| \leq e^T \big|\langle\!\langle u - \bar{u}, \delta_{\bar{x}(\cdot)}\rangle\!\rangle\big|.
\]

To complete the proof, replace \( u \) with an element of the sequence \( (u^k) \subset \mathcal{U} \) converging to \( \bar{u} \), and leverage the definitions of the space $\bm{\mathcal F}$ and the weak* convergence.



\textbf{2.} In Example~\ref{ex:2}, the set $U$ is chosen as a closed ball of the dual $\bm V' = {L}^{\rm q}(\R^n; \R^m)$ of the Lebesgue space $\bm V = {L}^{\mathrm{p}}(\R^n; \R^m)$, $1\leq p < \infty$, $\frac{1}{\mathrm p}+\frac{1}{\rm q}=1$.

Let \( F \) be a matrix having \( f^k \) as the $k^{\mbox{\footnotesize th}}$ column. By assumption, $F$ is Carath\'{e}odory and sublinear with a certain constant $C_F$. A straightforward analysis based on the estimates
\begin{align*}
    \left|x(t)\right| &\leq \int_{0}^t \left|F_s(x(s))v_s(x(s))\right|\d s\\
    &\leq C_F \int_{0}^t (1+|x(s)|) \left|\int \eta(x(s) - {\rm x})u_s({\rm x})\d {\mathrm{x}}\right| \d{s}\\
    &\leq C_F \int_{0}^t (1+|x(s)|) \|\eta(x(s)-\cdot)\|_{L^{\mathrm p}}\,\|u_s\|_{L^{\mathrm q}} \d{s},
\end{align*}
and on the Gr\"onwall inequality shows that the system \eqref{ODEex} admits a unique solution on the entire interval \( I \). Furthermore, the trajectory tube of this control system is contained in a ball \( K \doteq \mathbb{B}_r \) of a suitably large radius \( r \). 

To establish the continuity of the operator \( u \mapsto x^u \), we observe that, for each \( \mathrm{p} \geq 1 \), the elements \( \mathrm{u} \) of \( U \) are uniformly bounded in \( L^1(K) \). Then, for any \( u, \bar{u} \in \mathcal{U} \), we have the following estimates:
\begin{align*}
    \big|x(t) - \bar{x}(t)\big| &\leq \int_{0}^t \big|F_s(x(s)) - F_s(\bar{x}(s))\big|\cdot \big|v_s(x(s))\big| \d{s} \\
    &\quad + \int_{0}^t \big|F_s(\bar{x}(s))\big|\cdot \big|v_s(x(s)) - v_s(\bar{x}(s))\big| \d{s} \\
    &\quad + \left|\int_{0}^t F_s(\bar{x}(s)) \big[v_s(\bar{x}(s)) - \bar{v}_s(\bar{x}(s))\big] \d{s}\right| \\
    &\leq \Lip_{K}(F) \int_{0}^t \big|x(s) - \bar{x}(s)\big| \left|\int u_s({\rm x}) \, \eta(x(s)-{\rm x}) \d{\rm x}\right| \d{s} \\
    &\quad + \sup_{I \times K}|F_s| \int_{0}^t \left|\int  u_s({\rm x}) \big[\eta(x(s) - \mathrm x)- \eta(\bar x(s) - {\rm x})\big]\d\mathrm{x}\right| \d{s} \\
    &\quad + \left|\int_{0}^t F(\bar{x}(s)) \left(\int \eta(\bar{x}(s) - {\rm x})\big[u_s({\rm x}) - \bar{u}_s({\rm x})\big] \d {\rm x}\right)\d{s}\right|\\
    &\leq \Lip_{K}(F) \ \sup_{K}|\eta| \int_{0}^t \|{u}_s\|_{L^1} \big|x(s) - \bar{x}(s)\big| \d{s} \\
    &\quad + \sup_{I \times K}|F_s| \ \Lip_{K}(\eta) \int_{0}^t \|u_s\|_{L^1} \big|x(s) - \bar{x}(s)\big| \d{s} \\
    &\quad + \left|\int_{0}^t F(\bar{x}(s)) \left(\int \eta(\bar{x}(s) - {\rm x})\big[u_s({\rm x}) - \bar{u}_s({\rm x})\big] \d {\rm x}\right)\d{s}\right|,
\end{align*}
where \( |\cdot| \) denotes any norm in \( \R^n \), and the corresponding matrix norm is denoted similarly. Applying Gr\"onwall's lemma, we obtain:
\[
    \big|x(t) - \bar{x}(t)\big| \leq C \big|\langle\!\langle u - \bar{u}, F(\bar{x}) \eta(\bar{x} - \cdot)\rangle\!\rangle\big|
\]
for a certain \( K \)-uniform constant \( C > 0 \). Proceeding as above completes the proof.

\small{

	\bibliography{references.bib}

}
\end{document}